\documentclass[12pt,reqno]{amsart}

\usepackage[margin=1in,footskip=0.5in]{geometry}

\usepackage{mathtools}
\usepackage{amssymb}
\usepackage{bm}
\usepackage{amsthm}
\usepackage{thmtools}

\usepackage{graphicx}
\usepackage{tikz}
\usepackage[all]{xy}

\usepackage{enumitem}
\usepackage{xcolor}

\usepackage{ytableau}
\usepackage{genyoungtabtikz}

\usepackage[colorlinks=true,citecolor=blue,linkcolor=red,urlcolor=black,linktocpage,unicode]{hyperref}

\usetikzlibrary{patterns,decorations.pathmorphing,decorations.pathreplacing}

\theoremstyle{plain}
\numberwithin{equation}{section}

\newtheorem{thm}{Theorem}[section]
\newtheorem{lem}[thm]{Lemma}
\newtheorem{cor}[thm]{Corollary}
\newtheorem{prop}[thm]{Proposition}
\newtheorem{defn}[thm]{Definition}
\newtheorem{rem}[thm]{Remark}
\newtheorem{ex}[thm]{Example}
\newtheorem{conj}[thm]{Conjecture}
\newtheorem{prob}[thm]{Problem}

\DeclareMathOperator{\rk}{rk}



\declaretheoremstyle[
 headfont=\normalfont\bfseries,
 bodyfont=\itshape,
 numbered=no,
 headpunct={.},
 postheadspace=1em
 ]{nonumberplain}

\declaretheorem[style=nonumberplain, name=Theorem A]{thmA}
\declaretheorem[style=nonumberplain, name=Theorem B]{thmB}
\declaretheorem[style=nonumberplain, name=Theorem C]{thmC}

\newcommand{\B}{\bullet}
\newcommand{\cod}[1]{0.46*#1}
\newcommand{\cord}[4]{\draw[dashed] (\cod{#1}, \cod{#2}) -- (\cod{#3}, #4)}
\newcommand{\dod}[1]{\cod{#1} + 0.23}
\newcommand{\dord}[2]{\node at (\dod{#1}, #2) {$\bullet$}}
\newcommand{\dif}[1]{\stackrel{(#1)}{\mathbf{X}}\hskip-0.5em^{(p)}}
\newcommand{\pdif}[1]{\stackrel{(#1)}{X}\hskip-0.5em^{(p)}}

\selectfont

\begin{document}

\title{Combinatorial Structure in Nevanlinna Theory}
\author{Shuhei Katsuta}
\thanks{This work was supported by JSPS KAKENHI Grant Number JP24KJ1283.}
\address{Graduate School of Mathematics, Nagoya University. Furocho, Chikusaku, Nagoya, Japan, 464-8602.}
\email{katsuta.shuhei.c9@s.mail.nagoya-u.ac.jp}
\date{}

\begin{abstract}
  In \emph{Meromorphic Functions and Analytic Curves}, H. and F. J. Weyl identified an intriguing connection between holomorphic curves and their associated curves, which they referred to as the ``peculiar relation''. In this paper, we present a generalization of the Weyl peculiar relation and investigate a combinatorial structure underlying the Weyl--Ahlfors theory via standard Young tableaux. We also provide an alternative proof of the Second Main Theorem from the viewpoint of comparing the order functions $ iT_p $ and $ T_i\{\mathbf{X}^{(p)}\} $.
\end{abstract}

\maketitle

\section{Introduction}
Classical Nevanlinna theory, initiated by  R. Nevanlinna, studies the value distribution of holomorphic curves in complex projective space---originally in $ \mathbb{P}^1 $, and more generally in $ \mathbb{P}^n $. This theory provides powerful methods for proving key results about holomorphic curves, such as the Casorati--Weierstrass theorem,
Picard's theorem, and Borel's theorem. To develop this theory, various approaches have been proposed. One of them is the \textbf{Weyl--Ahlfors theory,} introduced by H. and F. J. Weyl \cite{Weyl_1}, \cite{Weyl_2} and L. V. Ahlfors \cite{Ahlfors}. In this paper, our study is primarily based on this approach.
A distinctive feature of the Weyl--Ahlfors theory is the introduction of the \textbf{$ p $-th associated curve} $ \mathbf{X}^{(p)} $, which arises from the holomorphic curve $ \mathbf{x} : \mathbb{C} \to \mathbb{P}^n $ by taking the wedge product with its higher-order derivatives.
In addition, Ahlfors introduced the innovative idea of employing singular metrics in the averaging process of integral geometry. On the basis of these ideas, they provided an alternative proof of Cartan's Second Main Theorem \cite{Cartan}, and also proved the Second Main Theorem for associated curves;
for details, see \cite{Ahlfors}, \cite{Weyl_2}, and H.-H. Wu \cite{Wu}.

The Weyl--Ahlfors theory has been studied from various perspectives. For example, in \cite{Weyl_2}, it was generalized to the case where the domain of $ \mathbf{x} $ is an open Riemann surface. W. Stoll \cite{Stoll_1}, \cite{Stoll_2} extended the domain from $ \mathbb{C} $ to $ \mathbb{C}^m $ $ (m > 1)$, and further to parabolic manifolds.
The theory was also studied with an emphasis on curvature by M. Cowen and P. Griffiths \cite{Cowen-Griffiths}. Other studies from the differential-geometric point of view include, for example, those by S. S. Chern \cite{Chern}, H.-H. Wu \cite{Wu}, and Y.-T. Siu \cite{Siu_1}, \cite{Siu_2}.

Another important aspect of the Weyl--Ahlfors theory is its connection to number theory, as pointed out by P. Vojta \cite{Vojta_1} (known as \emph{Vojta's dictionary}); see also J. Noguchi and J. Winkelmann \cite{Noguchi-Winkelmann}, and M. Ru \cite{Ru}. In place of associated curves, lattice parallelepipeds in the geometry of numbers on Grassmannians play a similar role in the proof of Schmidt's Subspace Theorem; see W. M. Schmidt's lecture notes \cite{Schmidt}.

In this paper, we study this theory from a combinatorial perspective. In particular, we establish a new connection with \textbf{Young diagrams}. This connection arises in the course of examining a generalization of an intriguing relation appearing in Chapter~3, Section~8 of the book by H. and F. J. Weyl \cite{Weyl_2}.
This relation describes the connection between several order functions $ T_i $ of $ \mathbf{x} $ and the single order function $ T_i\{\mathbf{X}^{(p)}\} $ of the associated curve $ \mathbf{X}^{(p)} $. They discovered this relation as a natural generalization of the equations $ S^p = S^1\{\mathbf{X}^{(p)}\}$,
$ \Omega_{p} = \Omega_1\{\mathbf{X}^{(p)}\} $,
and $ T_p = T_1\{\mathbf{X}^{(p)}\} $; see Section~\ref{sec:Weyl_pec_rel}. Indeed, in \cite{Weyl_2}, they write:
\begin{center}
  ``\emph{It is natural to ask whether the higher $ v_q $, $ S^q $ of the curve $ \mathfrak{C}_p  $ in $ \binom{n}{p} $-space may be expressed by the same quantities for $ \mathfrak{C} $. Closer examination shows that such relations prevail only for $ q = 2 $ besides $ q = 1 $.}''(\cite{Weyl_2}, p.~161).
\end{center}
(Note that the notation in this statement differs from ours.)
Moreover, they also write:
\begin{center}
  ``\emph{a strange relation of which we are not aware whether it is known even for
    rational and algebraic curves.}''(\cite{Weyl_2}, p.~162).
\end{center}
Hence, following their terminology, we refer to this as the \textbf{Weyl peculiar relation} (Theorem~\ref{thm:pec_rel}). It remains an interesting problem to clarify what kinds of relations hold when $ q \geq 3 $ (in this paper, we denote $ q $ by $ i $). At present, it seems difficult to establish such generalized relations as equalities; therefore, we attempt to formulate them as inequalities. One of the main results of this paper is the following:
\begin{thmA}[Theorem~\ref{thm:gen_pec_rel}, Generalized Weyl Peculiar Relation for $ T_i\{\mathbf{X}^{(p)}\} $]\label{thm:thmA}
  Let $ n $ and $ p $ $ (\leq n) $ be positive integers, and let $ i $ be an integer satisfying $ 1 \leq i \leq p(n - p + 1) $. Assume that $ \mathbf{X}^{(p)} $ is non-degenerate as a holomorphic curve. Then the following inequality holds:
  \begin{equation}\label{eq:thmA}
    \sum_{s = 1}^{i - 1} \max_{\sigma \in \binom{[n + 1]}{p}_{(k_s)}}\left(\sum_{k = 1}^n n_{\lambda(\sigma)}(k)(T_{k - 1} - 2T_k + T_{k + 1})\right)
    + i T_p
    \leq T_i\{\mathbf{X}^{(p)}\} + O(1),
  \end{equation}
  where the set $ \binom{[n + 1]}{p}_{(k_s)} $ is defined in \textup{Section~\ref{sec:assoc_curve}}, and the integer $ n_{\lambda(\sigma)}(k)$ is defined in \textup{Definition~\ref{def:phi}} for each $ \sigma \in \binom{[n + 1]}{p}_{(k_s)} $ and for each integer $ 1 \leq k \leq n $.
\end{thmA}
Here, a combinatorial quantity $ n_{\lambda(\sigma)}(k) $ arises, which is associated with the Young diagram $ \lambda(\sigma) $. Moreover, this inequality is best possible in the sense that there exists a holomorphic curve  for which equality holds in \eqref{eq:thmA}; see Proposition~\ref{prop:sp_curve}. For our theorem to hold in a meaningful way, it is necessary that the left-hand side of \eqref{eq:thmA} be positive. The following theorem ensures this.
\begin{thmB}[Theorem~\ref{thm:pos_ineq}]\label{thm:thmB}
  Let $ 0 < \epsilon < 1 $. Then, for each integer $ 1 \leq i \leq p(n - p + 1) $, the following inequality holds:
  \begin{equation*}
    (1 - \epsilon)\min(T_p, T_{n - p + 1}) < \sum_{s = 1}^{i - 1}\max_{\sigma \in \binom{[n + 1]}{p}_{(k_s)}}\left(\sum_{k = 1}^n n_{\lambda(\sigma)}(k)(T_{k - 1} - 2T_k + T_{k + 1})\right) + iT_p \, //,
  \end{equation*}
  where the symbol $ // $ indicates that the inequality holds for all $ r \in \mathbb{R} $ outside a subset of $ \mathbb{R} $ of finite \textup{(\textit{Lebesgue})} measure. In this case, the exceptional set depends on $ n $, $ \mathbf{x} $, $ p $, $ i $, and $ \epsilon $.
\end{thmB}
Furthermore, we examine the combinatorial (and geometric) background underlying the relation \eqref{eq:thmA}. In classical Nevanlinna theory, research on the associated curves has suggested the presence of combinatorial structures.
For example, H. Fujimoto \cite{Fujimoto_1} established a truncated version of a defect relation \eqref{eq:trunc_def_rel}, which generalizes Ahlfors's defect relation \eqref{eq:def_rel}, by analyzing the weight of $ I = (i_0, i_1, \ldots, i_k) $ (see \cite{Fujimoto_1}, p.~148, Definition~4.1; similar quantities appear in this paper as $ d(i_0, i_1, \ldots, i_{p - 1}) $ \eqref{eq:d(I)}) and the behavior of the Wronskian (see \cite{Fujimoto_1}, p.~148, Lemma~4.2 (\emph{Fujimoto's trick}) and p.~151, Proposition~5.3).
More recently, using similar methods, D.T. Huynh and S.-Y. Xie \cite{Huynh-Xie} proved the defect relation for degenerate entire curves, which improves the result of W. Chen \cite{Chen}.
In the Weyl--Ahlfors theory, the \textbf{Pl\"ucker formula} for holomorphic curves (Theorem~\ref{thm:Plucker_formula}) plays a crucial role in deriving the Second Main Theorem. Separately, in connection with the Pl\"{u}cker formula, the work of J.-L. Gervais and Y. Matsuo \cite{Gervais-Matsuo} along with subsequent studies, such as
J. Jost and G. Wang \cite{Jost-Wang} and A. Eremenko \cite{Eremenko}, is of interest. In these works, the applications to integrable systems (Toda lattice) are explored. In addition, we note a connection to the geometry of flag varieties, as discussed by P. Vojta \cite{Vojta_2}, p.~14, Proposition~3.7, p.~15, Remarks~3.9 and Remark~3.10.
In the present paper, the Pl\"ucker formula enables us to regard the collection of order functions $\{T_i\}_{i = 0}^{n + 1}$ as a combinatorial object. Indeed, the formula shows that the second-order differences $ T_{i - 1} - 2T_i + T_{i + 1} $ are ``negative'', and hence the sequence $ \{T_i\}_{i = 0}^{n + 1} $ is ``concave''; see Remark~\ref{rem:T-seq}.
\begin{figure}[htbp]
  \begin{equation*}
    \begin{matrix}
      T_1                     & T_2                     & \cdots & T_n                                                              \\
      T_1\{\mathbf{X}^{(2)}\} & T_2\{\mathbf{X}^{(2)}\} & \cdots & T_n\{\mathbf{X}^{(2)}\} & T_{n + 1}\{\mathbf{X}^{(2)}\} & \cdots \\
      T_1\{\mathbf{X}^{(3)}\} & T_2\{\mathbf{X}^{(3)}\} & \cdots & T_n\{\mathbf{X}^{(3)}\} & T_{n + 1}\{\mathbf{X}^{(3)}\} & \cdots \\
      \vdots                  & \vdots                  & \ddots & \vdots                  & \cdots                                 \\
      T_1\{\mathbf{X}^{(n)}\} & T_2\{\mathbf{X}^{(n)}\} & \cdots & T_n\{\mathbf{X}^{(n)}\}
    \end{matrix}
  \end{equation*}
  \caption{System of order functions}
  \label{fig:order_functions}
\end{figure}

Figure~\ref{fig:order_functions} shows the central objects considered in this paper, namely, a system of order functions. While the rows correspond to the classical Weyl--Ahlfors theory, the columns represent a new perspective. In particular, the Weyl peculiar relation appears in the second column.

From a different perspective, \eqref{eq:thmA} can be viewed as a comparison between $ iT_p $ and $ T_i\{\mathbf{X}^{(p)}\} $. One may summarize Theorem A with the slogan:
\begin{equation*}
  iT_p + \sum_{s = 1}^{i - 1}(\text{``negative'' term}) \leq T_i\{\mathbf{X}^{(p)}\}.
\end{equation*}
On the other hand, one can verify that $ T_i\{\mathbf{X}^{(p)}\} < (i + \epsilon)T_p \, // $ \eqref{eq:fund_ineq_3}. We also consider possible refinements of this inequality.
\begin{thmC}[Theorem \ref{thm:gen_SMT}]\label{thm:thmC}
  Let $ 2 \leq i \leq \binom{n + 1}{p} $ be an integer, and set $ h \coloneqq \binom{n + 1}{p} - 1 $. Let $ \{\mathbf{B}_j^{(i - 1)}\}_{j = 1}^d \subseteq \bigwedge^{i - 1}\mathbb{C}^{\binom{n + 1}{p}} $ be a finite set of nonzero decomposable $ (i - 1) $-vectors in general position for $ 1 $ \textup{(\textit{see} Definition~\ref{def:in_gen_pos})}. Then, for any $ \epsilon > 0 $, the following inequality holds:
  \begin{equation*}
    \begin{split}
      T_i\{\mathbf{X}^{(p)}\}(r) + \sum_{j = 1}^d \frac{i - 1}{\binom{h}{i - 2}}\widetilde{m}_1\{\mathbf{X}^{(p)}\}(r, \mathbf{B}_j^{(i - 1)})
      + \sum_{j = 1}^d\sum_{k = 1}^{i - 2}\frac{h - i + 1}{\binom{h - k + 1}{i - 1 - k}} & M_1^{(k)}\{\mathbf{X}^{(p)}\}(r, \mathbf{B}_j^{(i - 1)}) \\
      & < (i + \epsilon)T_p(r) \, //,
    \end{split}
  \end{equation*}
  where the proximity function $ \widetilde{m}_1\{\mathbf{X}^{(p)}\}(r, \mathbf{B}^{(i - 1)}) $ is defined in \textup{Definition~\ref{def:prox_func}}, and its higher-dimensional analogue, $ M_1\{\mathbf{X}^{(p)}\}(r, \mathbf{B}^{(i - 1)}) $ is given in \textup{Definition~\ref{def:Psi-M}}.
\end{thmC}
In fact, Theorem C is a form of the Second Main Theorem for $ \mathbf{X}^{(p)} $ (Theorem~\ref{thm:SMT}). We derive this result by applying Theorem~\ref{thm:omega_M_small}, due to Ahlfors \cite{Ahlfors} and H. and F. J. Weyl \cite{Weyl_2}, to $ \mathbf{X}^{(p)} $.
In slogan form, Theorem C says:
\begin{equation*}
  T_i\{\mathbf{X}^{(p)}\} + \sum_{j = 1}^d (\text{proximity function}) \leq iT_p.
\end{equation*}
Figure \ref{fig:conceptual_diagram} is a conceptual diagram summarizing the main results of this paper.
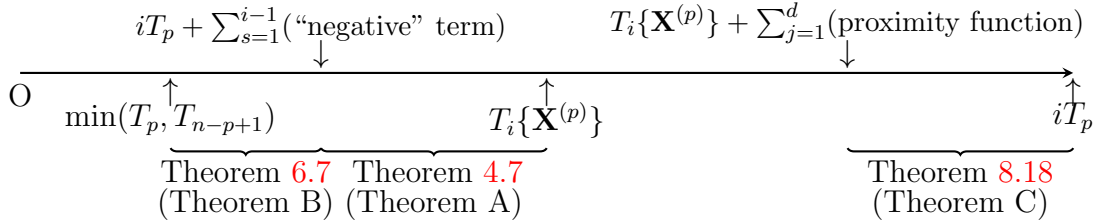
\begin{figure}[htbp]
  \centering
  \begin{tikzpicture}
    \draw[->,>=stealth,thick](-7,0) node[below]{O}--(7,0);
    \draw(-5,-0.25) node[below]{$\min(T_p, T_{n - p + 1})$};
    \draw(-5,-0.25) node{$ \uparrow $};
    \draw(-3,0.25) node{$ \downarrow $};
    \draw(-3,0.25) node[above,font=\small]{$ iT_p + \sum_{s = 1}^{i - 1}(\text{``negative'' term}) $};
    \draw(0,-0.25) node{$ \uparrow $};
    \draw(0,-0.25) node[below]{$ T_i\{\mathbf{X}^{(p)}\} $};
    \draw(4,0.25) node{$ \downarrow $};
    \draw(4,0.25) node[above,font=\small]{$ T_i\{\mathbf{X}^{(p)}\} + \sum_{j = 1}^d(\text{proximity function})$};
    \draw(7,-0.25) node{$ \uparrow $};
    \draw (7,-0.25) node[below]{$ iT_p $};
    \draw [thick,decoration={brace,mirror,raise=0.5cm},decorate] (-5,-0.5) -- (-3,-0.5) [anchor=north,xshift=-1cm,yshift=-0.5cm] node{Theorem~\ref{thm:pos_ineq}};
    \draw [thick,decoration={brace,mirror,raise=0.5cm},decorate] (-3,-0.5) -- (0,-0.5) [anchor=north,xshift=-1.45cm,yshift=-0.5cm] node{Theorem~\ref{thm:gen_pec_rel}};
    \draw [thick,decoration={brace,mirror,raise=0.5cm},decorate] (4,-0.5) -- (7,-0.5) [anchor=north,xshift=-1.5cm,yshift=-0.5cm] node{Theorem~\ref{thm:gen_SMT}};
    \draw(-4,-1.75) node{$ (\text{Theorem B}) $};
    \draw(-1.5,-1.75) node{$ (\text{Theorem A}) $};
    \draw(5.5,-1.75) node{$ (\text{Theorem C}) $};
  \end{tikzpicture}
  \caption{Main inequalities: conceptual diagram}
  \label{fig:conceptual_diagram}
\end{figure}

Although it will not be difficult to prove these results in a more general setting, we consider only the most fundamental case $ \mathbf{x} : \mathbb{C} \to \mathbb{P}^n $.

We outline the structure of this paper. In Section~\ref{sec:assoc_curve}, we introduce the concept of \textbf{stationary indices} to analyze the order of vanishing of $ \mathbf{x} $. For associated curves, upper bounds on the stationary indices are described in terms of \textbf{Maya diagrams} or Young diagrams. To compute these bounds, we define a function $ \phi_p(\lambda) $, where $ \lambda $ is a Young diagram. In addition, we observe that the higher-order derivatives of associated curves can be computed using \textbf{standard Young tableaux}.
In Section~\ref{sec:Weyl_pec_rel}, after defining some fundamental functions in the Weyl--Ahlfors theory, we recall the Weyl peculiar relation and outline its proof in a way that is amenable to generalization. In Section~\ref{sec:gen_pec_rel}, following the approach of the Weyls' original argument, we establish a generalized version of the Weyl peculiar relation in the form of an inequality for each
$ 1 \leq i \leq p(n + p - 1) $. In the rest of Section~\ref{sec:gen_pec_rel}, we make several observations on our results and discuss their geometric aspects, particularly from the viewpoint of Schubert calculus as described in P. Griffiths and J. Harris \cite{Griffiths-Harris}. In Section~\ref{sec:balanced_sum_formula}, we study the sequence $ \{T_k\}_{k = 0}^{n + 1} $ from a combinatorial perspective and establish the \textbf{balanced sum formula}, which describes the relationship between the second-order differences $ T_{k - 1} - 2T_k + T_{k + 1} $ and $ T_p $.
In Section~\ref{sec:weighted_balanced_sum_formula}, we investigate certain combinatorial properties on standard Young tableaux by introducing the (finite) \textbf{Young lattice}. On the basis of these investigations, we prove the \textbf{weighted balanced sum formula}, which plays a key role in proving Theorem~\ref{thm:pos_ineq} (Theorem B).
In Section~\ref{sec:geom_ineq}, using exponential curves, we provide an elementary geometric interpretation of the generalized Weyl peculiar relation and conclude the section by proving that the inequality in Theorem~\ref{thm:gen_pec_rel} (Theorem A) is best possible.
In Section~\ref{sec:SMT}, we present an alternative approach to proving the \textbf{Second Main Theorem} in the same spirit as the proof of the generalized Weyl peculiar relation. We then generalize this argument to obtain Theorem~\ref{thm:gen_SMT} (Theorem C).

\section{Associated curves and Stationary indices}\label{sec:assoc_curve}
\subsection{Definitions of Stationary Indices and Associated Curves}
\ \par
Let $ n $ be a positive integer. A map
\begin{equation*}
  \mathbf{x} : \mathbb{C} \to \mathbb{P}^{n}, z \mapsto (x_0(z) : x_1(z): \cdots: x_n(z))
\end{equation*}
is called a \textbf{holomorphic curve} if each component $ x_i(z) $ $ (i = 0, 1, \ldots, n)$ is a holomorphic function. Let $ \Delta(r) \subseteq \mathbb{C} $ be the open disc of radius $ r > 0 $ centered at the origin.
We obtain a \textbf{reduced representation} $ \mathbf{x}_{\mathrm{red}} $ of a holomorphic curve $ \mathbf{x} = \mathbf{x}(z)$ by removing common zeros from each $ x_i $:
\begin{equation*}
  \mathbf{x}_{\mathrm{red}}: \mathbb{C} \to \mathbb{C}^{n + 1}, \quad z \mapsto (y_0(z), y_1(z), \ldots, y_n(z)).
\end{equation*}
This reduced representation is defined up to multiplication by a nowhere-vanishing entire function. Throughout the rest of this paper, we will use a reduced representation whenever considering $ \mathbf{x} $ as a map from $\mathbb{C}$ to $ \mathbb{C}^{n + 1} $ and write $ \mathbf{x}_{\mathrm{red}} $ (resp. $ y_i $)
simply as $ \mathbf{x} $ (resp. $ x_i $). Additionally, we assume that $ \mathbf{x} $ is \textbf{non-degenerate}, meaning that the image of $ \mathbf{x} $ is not contained in any hyperplane in $ \mathbb{P}^n $. In this setting, for each $ z_0 \in \mathbb{C} $, we can choose an appropriate local coordinate system on $ \mathbb{C}^{n + 1} $ such that each $ x_i $ admits a power series expansion of the form
\begin{equation*}
  x_i(z) = (z - z_0)^{\delta_i} + \cdots \quad (i = 0, 1, \ldots, n), \quad 0 = \delta_0 < \delta_1 < \cdots < \delta_n.
\end{equation*}
This coordinate system is called the \textbf{normal form} of $ \mathbf{x} $.
The integer
\begin{equation*}
  v_i = v_i(z_0) \coloneqq \delta_i - \delta_{i - 1} \quad (i = 1, \ldots, n)
\end{equation*}
is called the \textbf{stationary index of order $ i $} of $ \mathbf{x} $ at $ z_0 $.
We begin by proving the existence of the normal form and the well-definedness of the stationary indices.
\begin{lem}\textup{(\cite{Weyl_2}, p.~41; \cite{Griffiths-Harris}, p.~266)}
  By a suitable coordinate transformation of $ \mathbb{C}^{n + 1} $, one can obtain the normal form of $ \mathbf{x} $. Moreover, the stationary indices are well-defined; that is, they are uniquely determined.
\end{lem}
\begin{proof}
  Assume that each component $ x_i $ of $ \mathbf{x} $ is represented by a power series of the form $ \sum_{j = 0}^{\infty}a_{ij}(z - z_0)^j $. Define the $ ((n + 1) \times \infty) $-matrix $ A = (a_{ij}) $ $ (0 \leq i \leq n, j \geq 0) $. Since $ \mathbf{x} $ is non-degenerate, we have $ \mathrm{rank}(A) = n + 1 $.
  By applying elementary row operations, we obtain the row echelon form $ B = (b_{ij}) $ $ (0 \leq i \leq n, j \geq 0) $ of $ A $, where all pivots are equal to $ 1 $. Since we assume that $ \mathbf{x} $ is a reduced representation, we have $ b_{00} = 1 $. Hence, the coordinate system $ y_i \coloneqq \sum_{j = 0}^{\infty}b_{ij}(z - z_0)^j $  for $ i = 0, 1, \ldots, n $ gives the normal form of $ \mathbf{x} $ at $ z_0 $.
  The order of vanishing $ \delta_i $ is determined as the minimal integer $ k \, (\geq 0)$ such that $ b_{ik} = 1 $; that is, $ b_{ik} $ is the pivot in the $ (i + 1) $-th row of $ B $, and the uniqueness of the stationary indices is evident. This construction also shows that they are independent of the choice of reduced representation.
\end{proof}
For each integer $ 1 \leq p \leq n + 1 $, we denote the $ p $-th derivative of $ \mathbf{x} $ by $ \mathbf{x}^{(p)} $. We also denote by $\mathbf{X}^{(p)}$ the $p$-th wedge product of $\mathbf{x}, \mathbf{x}^{(1)}, \mathbf{x}^{(2)}, \ldots, \mathbf{x}^{(p - 1)}$:
\begin{equation*}
  \mathbf{X}^{(p)} \coloneqq \mathbf{x} \wedge \mathbf{x}^{(1)} \wedge \mathbf{x}^{(2)} \wedge \cdots \wedge \mathbf{x}^{(p - 1)}.
\end{equation*}
The map $ \mathbf{X}^{(p)} $ defines a holomorphic curve from $ \mathbb{C} $ into $ \mathbb{P}\left(\bigwedge^p \mathbb{C}^{n + 1}\right) = \mathbb{P}^{\binom{n + 1}{p} - 1} $ via the \textbf{Pl\"{u}cker embedding}:
\begin{equation*}
  \mathbf{X}^{(p)} : \mathbb{C} \to \mathrm{Gr}(p, n + 1) \hookrightarrow \mathbb{P}^{\binom{n + 1}{p} - 1},
\end{equation*}
where $ \mathrm{Gr}(p, n + 1) \coloneqq \{p\text{-dimensional subspaces in} \, \mathbb{C}^{n + 1}\} $ is the Grassmannian.
If $ p = 0 $, we set $ \mathbf{X}^{(0)} \coloneqq 1 $, and if the integer $ p $ does not belong to $ \{0, 1, \ldots, n + 1\} $, we set $ \mathbf{X}^{(p)} \coloneqq 0 $. The holomorphic curve $ \mathbf{X}^{(p)} $ is called the \textbf{$p$-th associated curve} of $ \mathbf{x} $, the $ p $-th derived curve, or the holomorphic curve of rank $ p $.

$ W \coloneqq \mathbf{X}^{(n + 1)} $ is called the \textbf{Wronskian} of $ \mathbf{x} $. The non-degeneracy condition of $ \mathbf{X}^{(p)} $ for all $ 0 \leq p \leq n + 1 $ is equivalent to $ W $ not being identically zero.

We write the Pl\"{u}cker coordinates of $ \mathbf{X}^{(p)} $ by
\begin{equation*}
  X_{i_0, i_1, \ldots, i_{p - 1}}^{(p)} \quad (0 \leq i_0 < i_1 < \cdots < i_{p - 1} \leq n).
\end{equation*}
Then, by computing the power series expansion of $ X_{i_0, i_1, \ldots, i_{p - 1}}^{(p)} $, we obtain the following lemma.
\begin{lem}[\cite{Weyl_2}, p.~43; H. Fujimoto \cite{Fujimoto_2}, p.~126]\label{lem:d_p}
  The order of vanishing of $ X_{i_0, i_1, \ldots, i_{p - 1}}^{(p)} $ at $ z_0 $ is given by
  \begin{equation}\label{eq:d(I)}
    d(i_0, i_1, \ldots, i_{p - 1}) \coloneqq \delta_{i_0} + (\delta_{i_1} - 1) + \cdots + (\delta_{i_{p - 1}} - p + 1).
  \end{equation}
  For simplicity of notation, we set
  \begin{equation*}
    d_p \coloneqq d(0, 1, \ldots, p - 1) \quad (1 \leq p \leq n + 1), \quad d_0 \coloneqq 0.
  \end{equation*}
\end{lem}
Although the following lemma is straightforward, it is essential in proving the concavity of the order functions; see Remark~\ref{rem:T-seq}.
\begin{lem}[\cite{Weyl_1} p.~522; \cite{Weyl_2}, p.~43]\label{lem:d-seq}
  \begin{equation*}
    d_{p - 1} - 2d_p + d_{p + 1} = v_p - 1 \geq 0.
  \end{equation*}
  This means that $ \{d_p\}_{p = 0}^{n + 1} $ is a convex sequence.
\end{lem}
\subsection{Maya Diagrams and Young Diagrams}
\ \par
We introduce the following useful notation:
\begin{equation*}
  \binom{[n + 1]}{p} \coloneqq \{(i_0, i_1, \ldots, i_{p - 1}) \in \mathbb{Z}^p \mid 0 \leq i_0 < i_1 < \cdots < i_{p - 1} \leq n\}.
\end{equation*}
For each $ k \in \mathbb{Z} $, we define a subset of $ \binom{[n + 1]}{p} $ by
\begin{equation*}
  \binom{[n + 1]}{p}_{(k)} \coloneqq \left\{(i_0, i_1, \ldots, i_{p - 1}) \in \binom{[n + 1]}{p} \Bigg| \, i_0 + i_1 + \cdots + i_{p - 1} = k\right\}.
\end{equation*}
Let $ k_0 < k_1 < \cdots < k_i < \cdots < k_q $ be all integers for which the set $\binom{[n + 1]}{p}_{(k_i)}$ $ (0 \leq i \leq q)$ is non-empty. In addition, we define $ \binom{[\infty]}{p} $ (resp. $ \binom{[\infty]}{p}_{(k_i)} $) as the union of all sets $ \binom{[k+1]}{p} $ (resp. $\binom{[k+1]}{p}_{(k_i)}$) for $ k \geq p - 1 $.

Obviously, the only element in the set $ \binom{[n + 1]}{p}_{(k_0)} $ is $ (0, 1, \ldots, p - 1) $. Hence, $ k_0 $ is given by $ 0 + 1 + \cdots + (p - 1) = \frac{p(p - 1)}{2} $.
Therefore, by Lemma~\ref{lem:d_p}, removing the common factor $ (z - z_0)^{d_p} $ from all Pl\"{u}cker coordinates $ X_{i_0, i_1, \ldots, i_{p - 1}}^{(p)} $ yields a reduced representation $ \mathbf{X}_{\mathrm{red}}^{(p)} $ of $ \mathbf{X}^{(p)} $. On the other hand, $ k_i \, (0 \leq i \leq q)$ is given by $ \frac{p(p - 1)}{2} + i $. Since the only element of the set $ \binom{[n + 1]}{p}_{(k_q)} $ is $ (n - p + 1, \ldots, n - 1, n) $, we have
\begin{equation*}
  \frac{p(p - 1)}{2} + q = (n - p + 1) + \cdots + (n - 1) + n = np - \frac{p(p - 1)}{2}.
\end{equation*}
Therefore, we deduce that
\begin{equation}\label{eq:dim_Grassmann}
  q = np - p(p - 1) = p(n - p + 1) = \dim \mathrm{Gr}(p, n + 1).
\end{equation}

An element of the set $ \binom{[n + 1]}{p} $ can be represented by a \textbf{Maya diagram} (introduced by M. Sato) of length $ n + 1 $, which consists of a sequence of $ n + 1 $ boxes arranged in a row, with $p$ balls (fermions), where at most one ball is placed in each box.
More precisely, the Maya diagram corresponding to $ (i_0, i_1, \ldots, i_{p - 1}) \in \binom{[n + 1]}{p}$ is obtained by placing one ball into each of the boxes numbered $ i_0 + 1, i_1 + 1, \ldots, i_{p-1} + 1 $. Clearly, this correspondence is bijective. \emph{Note that, in our notation, the indices of the boxes with a ball in the Maya diagram differ by $ 1 $ from the corresponding components of the elements of the set $ \binom{[n + 1]}{p} $}.
We denote the Maya diagram corresponding to $ \sigma \in \binom{[n + 1]}{p} $ by the same symbol $ \sigma $.
\begin{ex}
  Let $n = 12$ and $p = 8$. An example of this identification is shown in \textup{Figure~\ref{fig:int_seq-Maya_diag}}.
  \begin{figure}[htbp]
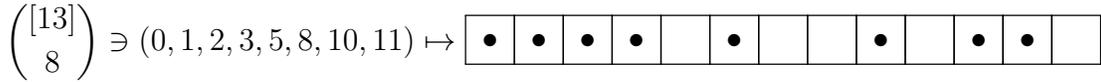

    \centering

    \begin{equation*}
      \binom{[13]}{8} \ni (0, 1, 2, 3, 5, 8, 10, 11) \mapsto \lower1ex\hbox{\ytableaushort{\B \B \B \B \ \B \ \ \B \ \B \B \ }}
    \end{equation*}

    \caption{Correspondence between a sequence of integers and a Maya diagram}\label{fig:int_seq-Maya_diag}
  \end{figure}

\end{ex}
Depending on the context, we identify the set $ \binom{[n + 1]}{p} $ with the set of \textbf{Young diagrams} that fit inside the $ p \times (n - p + 1) $ rectangle.
The procedure for this identification is as follows. For a given Maya diagram $ \sigma \in \binom{[n + 1]}{p} $, construct a stepwise path in the $ p \times (n - p + 1) $ grid, starting from the bottom-left corner and ending at the top-right corner, according to the following rule: read the Maya diagram from left to right; move one step upward for each ball, and one step to the right for each empty box.
This path forms part of the contour of a Young diagram, and the corresponding Young diagram $ \lambda(\sigma) $ is uniquely determined.  We write
\begin{equation*}
  (\lambda_0, \lambda_1, \ldots, \lambda_{p - 1}) \quad (\lambda_0 \geq \lambda_1 \geq \cdots \geq \lambda_{p - 1} \geq 0)
\end{equation*}
to represent the Young diagram $ \lambda $, and denote its size by $ |\lambda| $. We sometimes use the notation
\begin{equation*}
  [\lambda_0^{m_1 - m_0}, \lambda_{m_1 + 1}^{m_2 - m_1}, \ldots, \lambda_{m_{s - 1} + 1}^{m_s - m_{s - 1}}] \coloneqq (\lambda_0, \lambda_1, \ldots, \lambda_{p - 1})
\end{equation*}
if $ \lambda_{m_{k - 1} + 1} = \lambda_{m_{k - 1} + 2} = \cdots = \lambda_{m_{k}} $ $ (1 \leq k \leq s, \, 0 \leq m_{k} \leq p - 1, \, m_0 \coloneqq - 1)$. (By convention, we omit the parts with $ \lambda_i = 0 $.) Under this notation, the above correspondence is given by
\begin{equation*}
  \sigma = (i_0, i_1, \ldots, i_{p - 1}) \mapsto \lambda(\sigma) = (i_{p - 1} - p + 1, i_{p - 2} - p + 2, \ldots, i_1 - 1, i_0).
\end{equation*}
Thus, the size of the Young diagram $ \lambda(\sigma) $ corresponding to $ \sigma = (i_0, i_1, \ldots, i_{p - 1}) $ is given by
\begin{equation}\label{eq:size_Young}
  |\lambda(\sigma)| = i_0 + (i_1 - 1) + \cdots + (i_{p - 1} - p + 1).
\end{equation}
\begin{ex}\label{ex:Young_diag}
  An example of the correspondence between a Maya diagram and a Young diagram is shown in \textup{Figure~\ref{fig:Maya_diag-Young_diag}}.
  \begin{figure}[htbp]
    \centering
    \begin{align*}
      \raise6ex\hbox{\ytableaushort{\B \B \B \B \ \B \ \ \B \ \B \B \ }} \ \raise7ex\hbox{$\mapsto$}
      \begin{tikzpicture}
        \node[rotate=0] (a) at (0, 0) {
          $\ydiagram{4, 4, 3, 1}$
        };
      \end{tikzpicture}
    \end{align*}
    \caption{Correspondence between a Maya diagram and a Young diagram}\label{fig:Maya_diag-Young_diag}
  \end{figure}
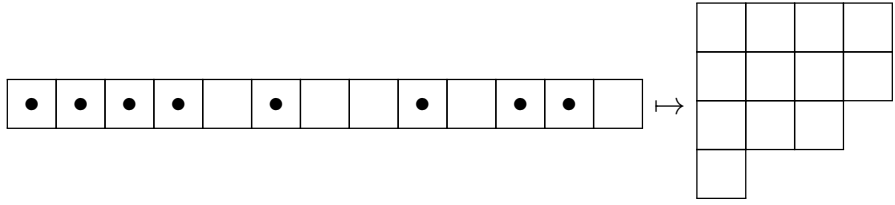
  In addition, \textup{Figure~\ref{fig:Russian_conv}} is useful for understanding this correspondence. Since it is essentially equivalent to a Young diagram in the Russian style, we will also refer to such a figure as following the Russian convention.
  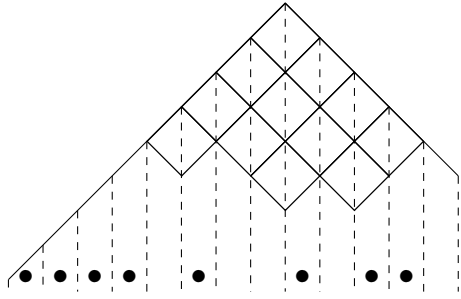
\begin{figure}[htbp]
    \centering
    \begin{tikzpicture}
      \draw[domain=0:\cod{5}] plot(\x, -\x + 1.84);
      \draw[domain=\cod{-8}:0] plot(\x, \x + 1.84);
      \node[rotate=-45] (a) at (0, 0) {
        $\ydiagram{4, 4, 3, 1}$
      };
      \cord{-8}{-4}{-8}{-2};
      \cord{-7}{-3}{-7}{-2};
      \cord{-6}{-2}{-6}{-2};
      \cord{-5}{-1}{-5}{-2};
      \cord{-4}{0}{-4}{-2};
      \cord{-3}{1}{-3}{-2};
      \cord{-2}{2}{-2}{-2};
      \cord{-1}{3}{-1}{-2};
      \cord{0}{4}{0}{-2};
      \cord{1}{3}{1}{-2};
      \cord{2}{2}{2}{-2};
      \cord{3}{1}{3}{-2};
      \cord{4}{0}{4}{-2};
      \cord{5}{-1}{5}{-2};
      \dord{-8}{-1.8};
      \dord{-7}{-1.8};
      \dord{-6}{-1.8};
      \dord{-5}{-1.8};
      \dord{-3}{-1.8};
      \dord{0}{-1.8};
      \dord{2}{-1.8};
      \dord{3}{-1.8};
    \end{tikzpicture}
    \caption{Young diagram drawn in the Russian convention}\label{fig:Russian_conv}
  \end{figure}
\end{ex}
\begin{lem}\label{lem:partition}
  The following inequality holds:
  \begin{equation*}
    \# \binom{[n + 1]}{p}_{(k_i)} \leq \mathsf{p}(i),
  \end{equation*}
  where $ \mathsf{p}(i) $ is the partition function of the integer $ i $; that is, the number of ways to express $ i $ as a sum of positive integers. The equality holds if and only if $ i \leq \min(n - p + 1, p) $.
\end{lem}
\begin{proof}
  The size of the Young diagram $ \lambda(\sigma) $ corresponding to $ \sigma \in \binom{[n+1]}{p}_{(k_i)} $ is independent of the choice of $ \sigma $, and is equal to $ i $ by \eqref{eq:size_Young}.
  This implies the desired inequality.
  Moreover, equality holds if and only if $ \binom{[n + 1]}{p}_{(k_i)} $ contains the elements corresponding to the partitions of $ i $ with the minimum and maximum numbers of parts. These elements are represented by $ (0, 1, \ldots, p - 2, p - 1 + i) $ and $ (0, 1, \ldots, p - i - 1, p - i + 1, p - i + 2, \ldots, p - 1, p) $, respectively. The condition that they are contained in $ \binom{[n + 1]}{p}_{(k_i)} $ is equivalent to $ p - 1 + i \leq n$ and $p - i \geq 0 $, which is further equivalent to $ i \leq \min(n - p + 1, p) $.
\end{proof}
Since the explicit expression of $ \mathsf{p}(i) $ is complicated, Lemma~\ref{lem:partition} implies that computing $ \#\binom{[n + 1]}{p}_{(k_i)} $ is generally difficult. On the other hand, using ``Gauss's method'' (i.e., computing the sum of an arithmetic series), we can compute the following quantity.
\begin{lem}\label{lem:Gauss_comp}
  \begin{equation*}
    \sum_{s = 1}^{q}s \cdot \#\binom{[n + 1]}{p}_{(k_s)} = \frac{q}{2}\binom{n + 1}{p} = \frac{p(n - p + 1)}{2}\binom{n + 1}{p}.
  \end{equation*}
\end{lem}
\begin{proof}
  Since $ \#\binom{n + 1}{p}_{(k_s)} = \#\binom{n + 1}{p}_{(k_{q - s})} \, (s = 0, 1, \ldots, q)$, we have
  \begin{equation*}
    2\sum_{s = 0}^{q}s \cdot \#\binom{[n + 1]}{p}_{(k_s)} = q\sum_{s = 0}^{q}\#\binom{[n + 1]}{p}_{(k_s)} = q\binom{n + 1}{p}.
  \end{equation*}
\end{proof}
We will use this equality later (see Corollary~\ref{cor:deg_X^p}).
\subsection{\texorpdfstring{Stationary Indices of the $ p $-th Associated Curve}{Stationary Indices of the p-th Associated Curve}}\label{subsec:stat_index}
\ \par
In this subsection, we evaluate the stationary indices of $ \mathbf{X}^{(p)} $ by investigating the order of vanishing $ d(\sigma) \coloneqq d(i_0, i_1, \ldots, i_{p - 1}) $ of the Pl\"{u}cker coordinates $ (X^{(p)})_{\sigma} \coloneqq X_{i_0, i_1, \ldots, i_{p - 1}}^{(p)}$, where $\sigma = (i_0, i_1, \ldots, i_{p - 1}) \in \binom{[n + 1]}{p} $. We denote the stationary index of $ \mathbf{X}^{(p)} $ by $ v_{i}\{\mathbf{X}^{(p)}\} $ $ (i = 1, \ldots, \binom{n + 1}{p} - 1)$, and write $ d_p $ as $ v_0\{\mathbf{X}^{(p)}\} $.
Observe the following decomposition:
\begin{align}\label{eq:ord_decomp}
   & d(i_0, i_1, \ldots, i_{p - 1})
  = \delta_{i_0} + (\delta_{i_1} - 1) + \cdots + (\delta_{i_{p - 1}} - p + 1)                                                                                                  \notag    \\
   & = \delta_{0} + (\delta_{1} - 1) + \cdots + (\delta_{p - 1} - p + 1) + (\delta_{i_0} - \delta_0) + (\delta_{i_1} - \delta_1) + \cdots + (\delta_{i_{p - 1}} - \delta_{p - 1}) \notag \\
   & = v_0\{\mathbf{X}^{(p)}\} + (\delta_{i_0} - \delta_0) + (\delta_{i_1} - \delta_1) + \cdots + (\delta_{i_{p - 1}} - \delta_{p - 1}).
\end{align}
If $ 0 \leq i < j \leq q $, then the definition of $ \delta_{\ast} $ and \eqref{eq:ord_decomp} imply
\begin{equation*}
  \min_{\sigma \in \binom{[n + 1]}{p}_{(k_i)}}d(\sigma) < \min_{\tau \in \binom{[n + 1]}{p}_{(k_j)}}d(\tau).
\end{equation*}
From this, we obtain an upper bound for $ v_i\{\mathbf{X}^{(p)}\} $ by subtracting $ \sum_{k = 0}^{i - 1}v_k\{\mathbf{X}^{(p)}\} $ from the common order of vanishing of the Pl\"{u}cker coordinates $ (X^{(p)})_{\sigma} $. Here, $ \sigma $ ranges over all elements of the set $ V(i) \, (0 \leq i \leq p(n + p - 1))$, which is inductively defined as
\begin{equation}\label{eq:V(i)}
  \bigcup_{0 \leq j \leq i}\binom{[n + 1]}{p}_{(k_j)} \bigg\backslash \{\sigma \in \mathbb{Z}^p \mid \text{there exists} \, j \, (< i) \, \text{such that} \, d(\sigma) = \min_{\tau \in V(j)}d(\tau) \}.
\end{equation}
The initial set $ V(0) $ is defined as $ \binom{[n + 1]}{p}_{(k_0)} = \{(0, 1, \ldots, p - 1)\} $.
Namely, we establish the following inequality:
\begin{equation}\label{eq:stat_rank_p}
  v_i\{\mathbf{X}^{(p)}\} \mathbf\leq \min_{\sigma \in V(i)}d(\sigma) - \sum_{k = 0}^{i - 1}v_k\{\mathbf{X}^{(p)}\}.
\end{equation}
\begin{ex}\label{ex:stat_index}
  For simplicity, we assume that $ i \leq \min(n - p + 1, p) $.
  \begin{itemize}
    \item $ i = 1 $. $ \#\binom{[n + 1]}{p}_{(k_1)} = \mathsf{p}(1) = 1 $ and its unique element is $ (0, 1, \ldots, p - 2, p) $. Since
          $$ d(0, 1, \ldots, p - 2, p) = v_0\{\mathbf{X}^{(p)}\} + \delta_p - \delta_{p - 1} = v_0\{\mathbf{X}^{(p)}\} + v_{p}, $$
          we have $ v_1\{\mathbf{X}^{(p)}\} = v_p $.
    \item $ i = 2 $. $ \#\binom{[n + 1]}{p}_{(k_2)} = \mathsf{p}(2) = 2 $ and its elements are
          $ (0, 1, \ldots, p - 3, p - 1, p)$ and $(0, 1, \ldots, p - 2, p + 1) $. Thus we have
          \begin{align*}
             & d(0, 1, \ldots, p - 3, p - 1, p)                                                                                                         = v_0\{\mathbf{X}^{(p)}\} + v_{p - 1} + v_{p}, \\
             & d(0, 1, \ldots, p - 2, p + 1)                                                                                                         = v_0\{\mathbf{X}^{(p)}\} + v_{p} + v_{p + 1}.
          \end{align*}
          This implies that $ v_2\{\mathbf{X}^{(p)}\} = \min(v_{p - 1}, v_{p + 1}) $.
    \item $ i = 3 $. $ \#\binom{[n + 1]}{p}_{(k_3)} = \mathsf{p}(3) = 3 $ and its elements are
          $ (0, 1, \ldots, p - 4, p - 2, p - 1, p) $, $ (0, 1, \ldots, p - 3, p - 1, p + 1) $, and $ (0, 1, \ldots, p - 2, p + 2)$. Thus we have
          \begin{align*}
             & d(0, 1, \ldots, p - 4, p - 2, p - 1, p)
            = v_0\{\mathbf{X}^{(p)}\} + v_{p - 2} + v_{p - 1} + v_{p},                                                                                                                                                                           \\
             & d(0, 1, \ldots, p - 3, p - 1, p + 1)                                                                                                                                   = v_0\{\mathbf{X}^{(p)}\} + v_{p - 1} + v_{p} + v_{p + 1}, \\
             & d(0, 1, \ldots, p - 2, p + 2)
            = v_0\{\mathbf{X}^{(p)}\} + v_{p} + v_{p + 1} + v_{p + 2}.
          \end{align*}
          This implies that
          \begin{equation*}
            v_3\{\mathbf{X}^{(p)}\}
            \begin{cases*}
              = \min(v_{p - 2}, v_{p + 1} - v_{p - 1}) & if $ v_2\{\mathbf{X}^{(p)}\} = v_{p - 1} < v_{p + 1}, $ \\
              = \min(v_{p - 1} - v_{p + 1}, v_{p + 2}) & if $ v_2\{\mathbf{X}^{(p)}\} = v_{p + 1} < v_{p - 1}, $ \\
              \leq \min(v_{p - 2}, v_{p - 1} (= v_{p + 1}), v_{p + 2}) & if $ v_2\{\mathbf{X}^{(p)}\} = v_{p - 1} = v_{p + 1}.$
            \end{cases*}
          \end{equation*}
  \end{itemize}
\end{ex}
The following function plays a key role in estimating the upper bounds of the stationary indices of the $ p $-th associated curves.
\begin{defn}\label{def:phi}
  Let $ \mathbf{x} $ be a holomorphic curve, and let $ 1 \leq p \leq n $ be an integer. Let $ \lambda $ be a Young diagram. Rotate $ \lambda $ clockwise by $ 45 $ degrees, and place it on the coordinate plane $ \mathbb{R}^2 $ so that the $ x $-coordinate of the peak vertex of the rotated diagram is $ p $.
  Additionally, we apply a suitable scaling so that the horizontal distance between adjacent vertices of the squares in $ \lambda $ becomes $ 1 $. For each integer $ 1 \leq k \leq n $, let $ n_{\lambda}(k) $ denote the number of squares intersected by the vertical line $ x = k $, excluding intersections that occur only at the vertices. Let $ \phi_p $ be a map from the set of Young diagrams to the set of non-negative integers $ \mathbb{Z}_{\geq 0} $ defined by
  \begin{equation*}
    \phi_p(\lambda) \coloneqq \sum_{k = 1}^n n_{\lambda}(k)v_k,
  \end{equation*}
  where $ v_k $ $ (k = 1, 2, \ldots, n) $ is the stationary index of $ \mathbf{x} $.
\end{defn}

\begin{ex}
  Let $ n = 13 $ and $ p = 7 $. We compute $ \phi_{7}(\lambda) $ for the Young diagram $ \lambda $ corresponding to $ (0, 1, 4, 7, 9, 10, 12) \in \binom{[n + 1]}{p} = \binom{[14]}{7}$.
  \begin{figure}[htbp]
    \centering

    \begin{tikzpicture}
      \draw[->,>=stealth,semithick] (\cod{-7.5},-1.5)--(\cod{6.5},-1.5)node[above]{$x$};
      \draw[->,>=stealth,semithick] (\cod{-0.5},-1.5)--(\cod{-0.5},2.8)node[right]{$y$};
      \draw[domain=\cod{-0.5}:\cod{6.5}] plot(\x, -\x + \cod{5});
      \draw[domain=\cod{-7.5}:\cod{-0.5}] plot(\x, \x + \cod{6});
      \node[rotate=-45] (a) at (0, 0) {
        $\ydiagram{6, 5, 5, 4, 2}$
      };
      \cord{-7.5}{-1.5}{-7.5}{-1.5};
      \cord{-6.5}{-0.5}{-6.5}{-1.5};
      \cord{-5.5}{0.5}{-5.5}{-1.5};
      \cord{-4.5}{1.5}{-4.5}{-1.5};
      \cord{-3.5}{2.5}{-3.5}{-1.5};
      \cord{-2.5}{3.5}{-2.5}{-1.5};
      \cord{-1.5}{4.5}{-1.5}{-1.5};
      \cord{-0.5}{5.5}{-0.5}{-1.5};
      \cord{0.5}{4.5}{0.5}{-1.5};
      \cord{1.5}{3.5}{1.5}{-1.5};
      \cord{2.5}{2.5}{2.5}{-1.5};
      \cord{3.5}{1.5}{3.5}{-1.5};
      \cord{4.5}{0.5}{4.5}{-1.5};
      \cord{5.5}{-0.5}{5.5}{-1.5};
      \cord{6.5}{-1.5}{6.5}{-1.5};
      \node at (\cod{-7.5}, -1.75) {$0$};
      \node at (\cod{-6.5}, -1.75) {$1$};
      \node at (\cod{-5.5}, -1.75) {$2$};
      \node at (\cod{-4.5}, -1.75) {$3$};
      \node at (\cod{-3.5}, -1.75) {$4$};
      \node at (\cod{-2.5}, -1.75) {$5$};
      \node at (\cod{-1.5}, -1.75) {$6$};
      \node at (\cod{-0.5}, -1.75) {$7$};
      \node at (\cod{0.5}, -1.75) {$8$};
      \node at (\cod{1.5}, -1.75) {$9$};
      \node at (\cod{2.5}, -1.75) {$10$};
      \node at (\cod{3.5}, -1.75) {$11$};
      \node at (\cod{4.5}, -1.75) {$12$};
      \node at (\cod{5.5}, -1.75) {$13$};
      \node at (\cod{6.5}, -1.75) {$14$};
    \end{tikzpicture}

    \caption{\texorpdfstring{Young diagram on $ \mathbb{R}^2 $}{Young diagram on R^2}}\label{fig:Young_diag_on_R^2}
  \end{figure}
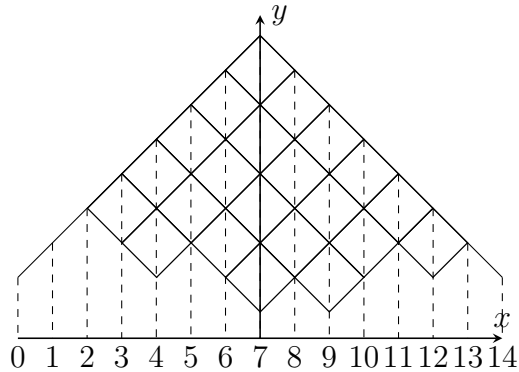

  From \textup{Figure~\ref{fig:Young_diag_on_R^2}}, we have
  \begin{align*}
    n_{\lambda}(3) & =  n_{\lambda}(11) =  n_{\lambda}(12) = 1, \\
    n_{\lambda}(4) & =  n_{\lambda}(5) =  n_{\lambda}(10) = 2,  \\
    n_{\lambda}(6) & =  n_{\lambda}(8) =  n_{\lambda}(9) = 3,   \\
    n_{\lambda}(7) & = 4.
  \end{align*}
  Therefore, we obtain
  \begin{equation*}
    \phi_7(\lambda) = v_3 + 2v_4 + 2v_5 + 3v_6 + 4v_7 + 3v_8 + 3v_9 + 2v_{10} + v_{11} + v_{12}.
  \end{equation*}
\end{ex}

\begin{prop}\label{prop:ord_ineq}
  Let $ 1 \leq i \leq p(n - p + 1) $ be an integer. Then we have
  \begin{equation*}
    v_i\{\mathbf{X}^{(p)}\} \leq \min_{\sigma \in V(i)}\phi_{p}(\lambda(\sigma)) - \sum_{k = 1}^{i - 1}v_k\{\mathbf{X}^{(p)}\},
  \end{equation*}
  where $ \lambda(\sigma) $ is the Young diagram corresponding to $ \sigma \in V(i) $.
\end{prop}
\begin{proof}
  By \eqref{eq:stat_rank_p}, it suffices to show that
  \begin{equation*}
    d(\sigma) = v_0\{\mathbf{X}^{(p)}\} + \phi_p(\lambda(\sigma))
  \end{equation*}
  for all $ \sigma \in \binom{[n + 1]}{p} $. To this end, we first examine the quantity $ \delta_{i_k} - \delta_k $, which appears in the equality \eqref{eq:ord_decomp}:
  \begin{align}\label{eq:delta_decomp}
    \delta_{i_k} - \delta_k & = (\delta_{k + 1} - \delta_{k}) + (\delta_{k + 2} - \delta_{k + 1})+ \cdots + (\delta_{i_k} - \delta_{i_{k - 1}}) \notag \\
                            & = v_{k + 1} + v_{k + 2} + \cdots + v_{i_{k}}.
  \end{align}
  To represent this quantity, we use a row of $ n $ boxes together with some balls.
  For each integer $ 0 \leq k \leq p - 1 $, place one ball into each of the $ (k + 1), (k + 2), \ldots, i_k $-th boxes.
  Repeat this procedure for all $ k $ from $  0 $ to $ p - 1 $, and for each integer $ 1 \leq m \leq n $, let $ b_{\sigma}(m) $ denote the number of balls in the $ m $-th box after all placements are complete.
  Then, using \eqref{eq:ord_decomp}, we derive
  \begin{equation*}
    d(\sigma) = v_0\{\mathbf{X}^{(p)}\} + \sum_{k = 1}^{n}b_{\sigma}(k)v_k.
  \end{equation*}
  As shown in Example~\ref{ex:Young_diag}, we represent the Young diagram $ \lambda(\sigma) $ derived from the Maya diagram $ \sigma $ as a figure in $ \mathbb{R}^2 $ using the Russian convention. Here, we arrange the Young diagram as described in Definition~\ref{def:phi} (see Figure~\ref{fig:Young_diag_on_R^2}).
  We claim that
  \begin{equation*}
    b_{\sigma}(k) = n_{\lambda(\sigma)}(k)
  \end{equation*}
  for all $ 1 \leq k \leq n $. We prove this claim by induction on the size $ s $ of the Young diagram $ \lambda(\sigma) $. First, this claim holds for $ s = 0 $. Indeed, for $ \sigma = (0, 1, \ldots, p - 1) $, we have $ n_{\lambda(\sigma)}(k) = 0 $ for all $ 1 \leq k \leq n $. On the other hand, since no balls are placed in any of the boxes by the above procedure, we have $ b_{\sigma}(k) = 0 $.
  Assume that the claim holds for $ s - 1 $ $ (s \geq 1 )$. Since $ s \geq 1 $, there exists a ball in $ \sigma $ with an empty box to its left. Suppose that the ball is placed in the $ (j + 1) $-th box from the left. Let $ \sigma' $ be the Maya diagram obtained by shifting this ball one position to the left. Then we have $ b_{\sigma'}(k) = b_{\sigma}(k) - 1 $ when
  $ k = j $, and $ b_{\sigma'}(k) = b_{\sigma}(k) $ for all other $ 1 \leq k \leq n $. Since the size of the corresponding Young diagram $ \lambda(\sigma') $ is $ s - 1 $, the induction hypothesis gives $ b_{\sigma'}(k) = n_{\lambda(\sigma')}(k) $ for all $ 1 \leq k \leq n $.
  By the correspondence between Maya diagrams and Young diagrams, it follows that $ n_{\lambda(\sigma')}(k) =  n_{\lambda(\sigma)}(k) - 1 $ if $ k = j $. For all other $ 1 \leq k \leq n $, we have $ n_{\lambda(\sigma')}(k) =  n_{\lambda(\sigma)}(k) $. Therefore, we have $ b_{\sigma}(k) = n_{\lambda(\sigma)}(k) $ for all $ 1 \leq k \leq n $, which completes the proof of the claim.
\end{proof}

\subsection{\texorpdfstring{Higher-Order Derivatives for the $ p $-th Associated Curve}{Higher-Order Derivatives for the p-th Associated Curve}}
\ \par
We compute the higher-order derivatives of the holomorphic curve $ \mathbf{X}^{(p)} $. Here, differentiation of $ \mathbf{X}^{(p)} $ refers to applying differentiation componentwise to its Pl\"{u}cker coordinates. We denote the $ i $-th derivative of $ \mathbf{X}^{(p)} $ by $ \dif{i} $.
These computations are carried out formally  using the Leibniz rule. When identifying $ \mathbf{x}^{(i_0)} \wedge \mathbf{x}^{(i_1)} \wedge \cdots \wedge \mathbf{x}^{(i_{p - 1})}$ with $ \sigma = (i_0, i_1, \ldots, i_{p - 1}) \in \binom{[\infty]}{p} $, the differentiation operator acts as
\begin{equation*}
  \sigma \mapsto \sum_{k : \, i_k - i_{k - 1} \geq 2}(i_0, i_1, \ldots, i_{k - 2}, i_{k - 1} + 1, i_{k}, \ldots, i_{p - 1}) + (i_0, i_1, \ldots, i_{p - 1} + 1),
\end{equation*}
where the right-hand side is a formal sum of elements in $ \binom{[\infty]}{p} $.
In terms of Maya diagrams, this operation corresponds to shifting a ball to the right when the adjacent right box is empty.
In terms of Young diagrams, this operation can be seen as the growth of Young diagrams. Lemma~\ref{lem:partition} implies that if $ \sigma \in \binom{[n + 1]}{p}_{(k_{i - 1})} $ and $ i \leq \min(n - p + 1, p) $ hold, then each term in the formal sum belongs to $ \binom{[n + 1]}{p}_{(k_{i})} $, and the sum consists of exactly $ \mathsf{p}(i) $ terms (counted with multiplicities).
To compute the coefficients, we introduce the \textbf{standard Young tableaux}. A standard Young tableau is a filling of a Young diagram $ \lambda $ with all natural numbers from $ 1 $ to $ |\lambda| $, such that the entries increase strictly from left to right in each row and from top to bottom in each column.
Let $ \mathrm{SYT}(\lambda) $ denote the set of standard Young tableaux of shape $ \lambda $. We define $ f_{\lambda} \coloneqq \# \mathrm{SYT}(\lambda) $ and set $ f_{\varnothing} \coloneqq 1 $.
\begin{lem}[\textit{see, for instance,} L. Gatto \textit{and} I. Scherbak \cite{Gatto-Scherbak}, p.~278]\label{lem:diff_formula}
  Let $ \mathbf{x} : \mathbb{C} \to \mathbb{C}^{n + 1} $ be a holomorphic curve, and let $ p \, (\leq n + 1) $ be a positive integer.
  Then we have
  \begin{equation*}
    \dif{i} = \sum_{\sigma \in \binom{[\infty]}{p}_{(k_i)}}f_{\lambda(\sigma)}\dif{i}_{\; \sigma},
  \end{equation*}
  where $ \dif{i}_{\;\sigma} $ is defined by $ \mathbf{x}^{(i_0)} \wedge \mathbf{x}^{(i_1)} \wedge \cdots \wedge \mathbf{x}^{(i_{p - 1})} $ for $ \sigma = (i_0, i_1, \ldots, i_{p - 1}) \in \binom{[\infty]}{p}_{(k_i)}$.
\end{lem}
\begin{proof}
  The proof proceeds by induction on $ i $.
  Let $ \lambda(\sigma) = (\lambda_0, \lambda_1, \ldots, \lambda_{p - 1}) $ be the Young diagram corresponding to $ \sigma \in \binom{[n + 1]}{p}_{(k_i)} $.
  Observe that the following recurrence relation holds for standard Young tableaux (see W. Fulton \cite{Fulton}, pp.~54--55) :
  \begin{equation}\label{eq:rec_rel}
    f_{\lambda(\sigma)} = \sum_{k : \, \lambda_k > 0} f_{\lambda(\sigma) - 1_k},
  \end{equation}
  where $ \lambda(\sigma) - 1_k $ is a partition of $ i - 1 $ defined as $ (\lambda_0, \lambda_1, \ldots, \lambda_{k - 1}, \lambda_k - 1, \lambda_{k + 1}, \ldots, \lambda_{p - 1}) $ if $ \lambda_k - 1 \geq \lambda_{k + 1} $ (otherwise, we set $ f_{\lambda(\sigma) - 1_k} = 0 $).
  The quantity $ f_{\lambda(\sigma)} $ is precisely the coefficient of $ \dif{i}_{\; \sigma} $, since each summand $ f_{\lambda(\sigma) - 1_k} $ on the right-hand side of \eqref{eq:rec_rel} appears as the coefficient of $ \dif{i - 1}_{\!\!\sigma'} $ for some $ \sigma' \in \binom{[n + 1]}{p}_{(i - 1)} $ whose derivative is $ \dif{i}_{\;\sigma} $, by the inductive hypothesis.
\end{proof}
\begin{rem}
  $ (1) $ \textup{(J. S. Frame, G. de B. Robinson, \textit{and} R. W. Thrall \cite{Frame-Robinson-Thrall}; \textit{see also} \cite{Fulton}, p.~53.)} Let $ \lambda $ be a Young diagram. It is well known that the number of standard Young tableaux of shape $ \lambda $ can be computed using the \textbf{hook length formula}:
  \begin{equation}\label{eq:hook_length_formula}
    f_{\lambda} = \frac{|\lambda|!}{\prod_{(i, j) \in \lambda}h(i, j)},
  \end{equation}
  where $ h(i, j) $ is the \textbf{hook length} of the $ (i, j) $-box in $ \lambda $, which is defined as the number of boxes lying to the right of the box in the same row or below it in the same column, including the box itself.

  \noindent $ (2) $ \textup{(\textit{see, for instance}, E. Smirnov \cite{Smirnov}, Theorem~2.31.)} $ f_{\lambda} $ also represents the degree of the Schubert variety $ X_{\widehat{\lambda}} \coloneqq \overline{\Omega_{\widehat{\lambda}}} \subseteq \mathrm{Gr}(p, n + 1) $ under the Pl\"ucker embedding. Here, $ \widehat{\lambda} $ denotes the complementary Young diagram of $ \lambda $ in the $ p \times (n - p + 1) $ rectangle $ [(n - p + 1)^p] $, and $ \Omega_{\lambda} $ is defined in \textup{Section~\ref{sec:gen_pec_rel}}, \textup{Subsection~\ref{subsec:geom_str}}
\end{rem}

\section{The Weyl peculiar relation}\label{sec:Weyl_pec_rel}

In this section, we present the proof of the Weyl peculiar relation in a form adapted to the context of this paper, following the Weyls' book \cite{Weyl_2}.
To this end, we begin by recalling fundamental notation and functions from classical Nevanlinna theory.
Fix a basis $ \mathbf{e}_0, \mathbf{e}_1, \ldots, \mathbf{e}_n $ of $ \mathbb{C}^{n + 1} $.
Let $ (\mathbb{C}^{n + 1})^{*} $ be the dual space of $ \mathbb{C}^{n + 1} $, with respect to the \emph{duality pairing} $ (\cdot, \cdot) $ which is defined by
\begin{equation*}
  (\mathbf{a}, \mathbf{x}) \coloneqq \sum_{k = 0}^n a_k x_k = a_0 x_0 + a_1 x_1 + \cdots + a_{n} x_{n}, \quad (\mathbf{a} \in (\mathbb{C}^{n + 1})^{*}, \mathbf{x} \in \mathbb{C}^{n + 1}).
\end{equation*}
More generally, for $ \mathbf{A}^{(p)} \in (\bigwedge^p \mathbb{C}^{n + 1})^{*} $ and $ \mathbf{X}^{(p)} \in \bigwedge^p \mathbb{C}^{n + 1} $, we define
\begin{equation*}
  (\mathbf{A}^{(p)}, \mathbf{X}^{(p)}) \coloneqq \sum_{\sigma \in \binom{[n + 1]}{p}}(A^{(p)})_{\sigma}(X^{(p)})_{\sigma
    }.
\end{equation*}
For all $ \mathbf{x}, \mathbf{y} \in \mathbb{C}^{n + 1} $, the Hermitian inner product $ \langle \mathbf{x}, \mathbf{y}\rangle $ is defined by
\begin{equation*}
  \langle \mathbf{x}, \mathbf{y}\rangle \coloneqq \sum_{k = 0}^n x_k \bar{y}_k = x_0 \bar{y}_0 + x_1 \bar{y}_1 + \cdots + x_n \bar{y}_n.
\end{equation*}
Similarly, for $ \mathbf{X}^{(p)}, \mathbf{Y}^{(p)} \in \bigwedge^p \mathbb{C}^{n + 1} $ we define
\begin{equation*}
  \langle \mathbf{X}^{(p)}, \mathbf{Y}^{(p)}\rangle \coloneqq \sum_{\sigma \in \binom{[n + 1]}{p}} (X^{(p)})_{\sigma} \overline{(Y^{(p)})}_{\sigma}.
\end{equation*}
Then we have $ \langle \mathbf{X}^{(p)}, \mathbf{Y}^{(p)}\rangle = (\mathbf{X}^{(p)}, \overline{\mathbf{Y}^{(p)}}) $.
In addition, the \emph{length} of $ \mathbf{x} $ (resp.~$ \mathbf{X}^{(p)} $) is given by $ |\mathbf{x}|^2 \coloneqq \langle\mathbf{x}, \mathbf{x}\rangle $ (resp.~$ |\mathbf{X}^{(p)}|^2 \coloneqq \langle \mathbf{X}^{(p)}, \mathbf{X}^{(p)} \rangle$).

Fix an orthonormal basis $ \mathbf{e}_0, \mathbf{e}_1, \ldots, \mathbf{e}_n $ of $ \mathbb{C}^{n + 1} $.
Then the \textbf{Hodge star operator} $ \star $ is a linear operator
\begin{equation*}
  \star : \bigwedge^p \mathbb{C}^{n + 1} \to \left(\bigwedge^{n - p + 1} \mathbb{C}^{n + 1} \right)^{*}
\end{equation*}
defined by the equation
\begin{equation*}
  \mathbf{A}^{(p)} \wedge \star\mathbf{X}^{(p)} = (\mathbf{A}^{(p)}, \mathbf{X}^{(p)})\mathbf{e}_0 \wedge \mathbf{e}_1 \wedge \cdots \wedge \mathbf{e}_n, \quad \mathbf{X}^{(p)} \in \bigwedge^p\mathbb{C}^{n + 1},
\end{equation*}
for all $ \mathbf{A}^{(p)} \in (\bigwedge^p \mathbb{C}^{n + 1})^{*}$.

\subsection{Fundamental Functions in the Weyl--Ahlfors Theory}
\ \par
\begin{defn}
  Let $ \mathbf{x} $ be a holomorphic curve. For each $ 1 \leq i \leq n $, outside zeros of $ \mathbf{X}^{(i)} $, we define
  \begin{equation*}
    S^{i}(z) \coloneqq 2\frac{|\mathbf{X}^{(i - 1)}|^2|\mathbf{X}^{(i + 1)}|^2}{|\mathbf{X}^{(i)}|^4}.
  \end{equation*}
  Fix a positive constant $ r_0 > 0 $ \textup{(\textit{sufficiently small})}. Then the \textbf{order function} $ T_i(r) $ of $ \mathbf{x} $ of rank $ i $ is defined by
  \begin{equation*}
    T_i(r) \coloneqq \frac{1}{2\pi}\int_{r_0}^r \frac{dt}{t}\int_{\Delta(t)}S^i(re^{\sqrt{-1}\theta})rdrd\theta \quad (r \geq r_0).
  \end{equation*}
  We sometimes denote it by $ T = T(r) $ if $ i = 1 $, and set
  \begin{equation*}
    T_0 \coloneqq 0, \quad T_{n + 1} \coloneqq 0.
  \end{equation*}
  In addition, we define
  \begin{equation*}
    \Omega_i(r) \coloneqq \frac{1}{4\pi}\int_0^{2\pi}\log \widetilde{S^i}(re^{\sqrt{-1}\theta})d\theta \quad (r \geq r_0),
  \end{equation*}
  where we denote $ \frac{S^i(z)}{2} $ by $ \widetilde{S}^i(z) $.
  These functions are also defined for the $ p $-th associated curve $ \mathbf{X}^{(p)} $, denoted by $ S^i\{\mathbf{X}^{(p)}\}, T_i\{\mathbf{X}^{(p)}\}$, and $ \Omega_i\{\mathbf{X}^{(p)}\}$, respectively.
\end{defn}
\begin{defn}
  Let $ w = (w_0 : w_1 : \cdots : w_{n}) $ be homogeneous coordinates on $ \mathbb{P}^n $, and let $ \mathbf{z}_i = (z_0, \ldots, z_{i - 1}, z_{i + 1}, \ldots, z_n) = \left(\frac{w_0}{w_i}, \ldots, \frac{w_{i - 1}}{w_i}, \frac{w_{i + 1}}{w_i}, \ldots, \frac{w_n}{w_i}\right) $ be local coordinates on $ U_i = \{w_i \neq 0\} $. The \textbf{Fubini-Study form} $ \omega_{\mathrm{FS}, n} $ is given by
  \begin{equation*}
    \omega_{\mathrm{FS}, n} \coloneqq \frac{\sqrt{-1}}{2\pi}\partial \bar{\partial} \log |w|^2 =  \frac{\sqrt{-1}}{2\pi}\partial \bar{\partial}\log(1 + |\mathbf{z}_i|^2)
  \end{equation*}
  on $ U_i $.
\end{defn}
The following lemma provides a geometric meaning of the order functions.
\begin{thm}[\textbf{Ahlfors--Shimizu characteristic}; T. Shimizu \cite{Shimizu}, pp.~126--127, Theorem III; \cite{Ahlfors}, pp.~10--12; \cite{Wu}, p.~102]\label{eq:A-S_char}
  \begin{equation*}
    T_i(r) = \int_{r_0}^r \frac{dt}{t}\int_{\Delta(t)}(\mathbf{X}^{(i)})^{\ast} \omega_{\mathrm{FS}, \binom{n + 1}{i} - 1} = \frac{\sqrt{-1}}{2\pi}\int_{r_0}^r \frac{dt}{t} \int_{\Delta(t)} \partial \bar{\partial}\log |\mathbf{X}^{(i)}(z)|^2.
  \end{equation*}
\end{thm}
In many cases, $ T_i $ is defined by this equation.
This representation indicates that the order function $ T_i $ measures the average area of $ \mathbf{X}^{(i)}(\Delta(r)) $ with respect to the Fubini--Study metric $ \omega_{\mathrm{FS}, \binom{n + 1}{i} - 1} $.
\begin{proof}
  Here, we prove only the case where $ i = 1 $. It suffices to show that
  \begin{equation*}
    \mathbf{x}^{\ast}\omega_{\mathrm{FS}, n} = \frac{\sqrt{-1}}{2\pi}\frac{|\mathbf{x}(z) \wedge \mathbf{x}'(z)|^2}{|\mathbf{x}(z)|^4}dz \wedge d\bar{z},
  \end{equation*}
  where $ \mathbf{x}'(z) \coloneqq \mathbf{x}^{(1)}(z) $.
  The computation of the left-hand side yields
  \begin{equation*}
    \mathbf{x}^{\ast}\omega_{\mathrm{FS}, n} = \frac{\sqrt{-1}}{2\pi}\partial \frac{\langle \mathbf{x}, \mathbf{x}'\rangle}{|\mathbf{x}|^2}d\bar{z} = \frac{\sqrt{-1}}{2\pi}\frac{|\mathbf{x}|^2|\mathbf{x}'|^2 - \langle \mathbf{x}, \mathbf{x}'\rangle \langle \mathbf{x}', \mathbf{x}\rangle}{|\mathbf{x}|^4}dz \wedge d\bar{z}.
  \end{equation*}
  The numerator of the right-hand side is calculated as follows:
  \begin{equation*}
    \begin{split}
      & |\mathbf{x}|^2|\mathbf{x}'|^2 - \langle \mathbf{x}, \mathbf{x}'\rangle \langle \mathbf{x}', \mathbf{x}\rangle                                                                             \\
      & = \left(\sum_{i = 0}^n |x_i|^2\right)\left(\sum_{j = 0}^n |x_j|^2\right) - \left(\sum_{i = 0}^n x_i \bar{x}_i'\right)\left(\sum_{j = 0}^n x_j' \bar{x}_j\right)                           \\
      & = \left(\sum_{i = 0}^n |x_i|^2|x_i'|^2 + \sum_{i \neq j}|x_i|^2|x_j'|^2\right) - \left(\sum_{i = 0}^n x_i \bar{x}_i' x_i' \bar{x}_i + \sum_{i \neq j}x_i \bar{x}_i' x_j' \bar{x}_j\right) \\
      & = \sum_{i \neq j}x_ix_j'(\bar{x}_i \bar{x}_j' - \bar{x}_i' \bar{x}_j) = \sum_{i < j}(x_i x_j' - x_j x_i')(\bar{x}_i \bar{x}_j' - \bar{x}_j' \bar{x_i})                                     \\
      & = \sum_{i < j}|x_i x_j' - x_j x_i'|^2 = |\mathbf{x} \wedge \mathbf{x}'|^2.
    \end{split}
  \end{equation*}
  Hence, the desired equality follows.
\end{proof}
\begin{ex}
  The Fubini--Study form on $ \mathbb{P}^1 $ is given by $ \omega_{\mathrm{FS}, 1} = \frac{\sqrt{-1}}{2\pi}\frac{dz \wedge d\bar{z}}{(1 + |z|^2)^2} $. By \textup{Theorem~\ref{eq:A-S_char}}, we obtain the following expression:
  \begin{equation*}
    T(r) = \frac{1}{\pi}\int_{r_0}^r \frac{dt}{t}\int_{\Delta(t)}\frac{|\mathbf{x} \wedge \mathbf{x}'|^2}{|\mathbf{x}|^4}udud\theta = \frac{\sqrt{-1}}{2\pi}\int_{r_0}^{r}\frac{dt}{t}\int_{\Delta(t)}\frac{|x'|^2}{(1 + |x|^2)^2}dz \wedge d\bar{z},
  \end{equation*}
  where $ x(z) \coloneqq \frac{x_1(z)}{x_2(z)} $ is a meromorphic function.
\end{ex}
\begin{defn}\label{def:N-V-Tbar}
  For a holomorphic curve $ \mathbf{x} $, we define
  \begin{equation*}
    N_i(r) \coloneqq \int_{r_0}^r \sum_{|z| < t}d_i(z)\frac{dt}{t}.
  \end{equation*}
  For the definition of $ d_i(z) $; see \textup{Lemma~\ref{lem:d_p}}. Let $ V_i(r) $ be the second-order difference of $ N_i(r) $ for $ i $:
  \begin{equation*}
    V_i(r) \coloneqq N_{i - 1}(r) - 2N_i(r) + N_{i + 1}(r) = \int_{r_0}^r \sum_{|z| < t}(v_i(z) - 1)\frac{dt}{t} \geq 0,
  \end{equation*}
  where we apply \textup{Lemma~\ref{lem:d-seq}} to the second equality.

  Using Ahlfors's notation, we define
  \begin{equation*}
    \overline{T}_i(r) \coloneqq T_i(r) + N_i(r).
  \end{equation*}
  In particular, we have $ \overline{T}_0 = 0 $, $ \overline{T}_1 = T_1 $, and $ \overline{T}_{n + 1} = N_{n + 1} $ .
  For the $ p $-th associated curve $ \mathbf{X}^{(p)} $, we denote the corresponding quantities by $ N_i\{\mathbf{X}^{(p)}\} $, $ V_i\{\mathbf{X}^{(p)}\} $, and $ \overline{T}_i\{\mathbf{X}^{(p)}\} $, respectively.
\end{defn}
The following two theorems play an important role in the Weyl--Ahlfors theory.
\begin{thm}[\textbf{Pl\"ucker formula}; \cite{Weyl_1}, p.~526; \cite{Weyl_2}, p.~123]\label{thm:Plucker_formula}
  \begin{equation*}
    V_i(r) + (T_{i - 1}(r) - 2T_{i}(r) + T_{i + 1}(r)) = \Omega_i(r) - \Omega_{i}(r_0) \quad (i = 1, 2, \ldots, n).
  \end{equation*}
  In Ahlfors's notation, we have
  \begin{equation*}
    \overline{T_{i - 1}}(r) - 2\overline{T_i}(r) + \overline{T_{i + 1}}(r) = \Omega_i(r) - \Omega_i(r_0) \quad (i = 1, 2, \ldots, n).
  \end{equation*}
\end{thm}
\begin{thm}[\cite{Weyl_1}, p.~531; \cite{Weyl_2}, p.~156]\label{thm:omega_small}
  For any $ \kappa > 1$, the following inequality holds:
  \begin{equation*}
    \Omega_i(r) \leq \frac{\kappa}{2} \log T_i(r) - \log r \, //.
  \end{equation*}
  See \textup{Theorem~B} in the introduction for the definition of the symbol $ // $.
\end{thm}
\begin{rem}\label{rem:T-seq}
  $ (1) $ \textup{Theorem~\ref{thm:Plucker_formula}} is an analogue of the Pl\"ucker formula for algebraic curves. Let $ M $ be a compact Riemann surface of genus $ g_M $, and let $ \mathbf{x} : M \to \mathbb{P}^n $ be a non-degenerate algebraic curve. We define the $ i $-th degree $ \nu_i $ of $ \mathbf{x} $, and the $ i $-th total ramification index $ \sigma_i $ by the following formulas:
  \begin{align*}
    \nu_i    & \coloneqq \deg(\mathbf{X}^{(i)}) = \int_{M}(\mathbf{X}^{(i)})^{*}\omega_{\mathrm{FS}, \binom{n + 1}{i} - 1}, \\
    \sigma_i & \coloneqq \sum_{p \in M}(v_i(p) - 1).
  \end{align*}
  Then the \textup{(\textit{global})} Pl\"ucker formula for algebraic curves states that
  \begin{equation}\label{eq:Plucker_formula_alg_curve}
    \sigma_i + (\nu_{i - 1} - 2\nu_i + \nu_{i + 1}) = 2g_M - 2.
  \end{equation}
  By the definition of $ \Omega_i(r) $, it is reasonable to expect that it describes the behavior of $ \mathbf{x} $ near infinity in the domain $ \mathbb{C} $.
  In particular, it is natural to predict that $ \Omega_i(r) $ becomes negative---corresponding to the value $ -2 $  in the Pl\"ucker formula for algebraic curves \eqref{eq:Plucker_formula_alg_curve} of genus $ 0 $---when $ r $ is sufficiently large. \textup{Theorem~\ref{thm:omega_small}} justifies this prediction modulo $ O(\log T_i) $.
  For more details on the Pl\"ucker formula for algebraic curves; see \textup{Chapter~2}, \textup{Section~4}, ``Pl\"ucker formulas \textup{(\textit{The General Pl\"ucker Formulas} I \& II)}'' in \cite{Griffiths-Harris}.

  \noindent$ (2) $ \textup{Theorem~\ref{thm:omega_small}} is proved using the \textbf{calculus lemma} \textup{(\cite{Ahlfors}, p.~15, (24); \cite{Weyl_2}, pp.~154--155, Lemma~6.~A)}, which is similar to Chebyshev's inequality. The appearance of exceptional sets in the Weyl--Ahlfors theory can all be traced back to this lemma.

  \noindent$ (3) $ By the two theorems above and \textup{Lemma~\ref{lem:d-seq}}, we see that the sequence $ \{T_i\}_{i = 0}^{n + 1} $ is concave modulo $O(\log \max_i T_i)$. In addition, this implies that each $ T_i $ $ (i = 1, \ldots, n)$ is of the same magnitude, in the sense that $ T_i(r) = O(T(r)) \, // $. More precisely, for all $ 1 \leq k < m \leq n $ and any $ \epsilon > 0 $, we have
  \begin{align}
    kT_m           & < mT_k + \epsilon T \, //, \label{eq:fund_ineq_1}           \\
    (n + 1 - m)T_k & < (n + 1 - k)T_m + \epsilon T \, //. \label{eq:fund_ineq_2}
  \end{align}
  These inequalities can be found in \cite{Ahlfors}, \textup{p.~14}; \cite{Wu}, \textup{p.~134}, \textup{Second Corollary}; \cite{Cowen-Griffiths}, \textup{p.~121}; and in \textup{S. Kobayashi} \cite{Kobayashi}, \textup{p.~423}. From these observations, in fact, we see that $ \{T_i\}_{i = 0}^{n + 1} $ is concave modulo $O(\log T) $. For convenience, we state this explicitly, as it will be used repeatedly:
  \begin{equation}\label{eq:concave_rel}
    T_{i - 1}(r) - 2T_i(r) + T_{i + 1}(r)                                  \leq \frac{\kappa}{2}\log T(r) - \log r \, //.
  \end{equation}
  Moreover, the equation $ T_p = T_1\{\mathbf{X}^{(p)}\} $ \textup{(\textit{see} \eqref{eq:T_p-omega_p})} implies that the order functions of the associated curve $ \mathbf{X}^{(p)} $ are also of the same magnitude.
  Therefore, from the perspective of studying the growth of holomorphic curves, all functions in the set $ \{T_i\}_{i = 1}^n \cup \{T_i\{\mathbf{X}^{(p)}\}\}_{i = 1}^{\binom{n + 1}{p} - 1} $ can be regarded as conveying essentially the same information.
  On the other hand, our approach in this paper is to consider $ \{T_i\}_{i = 0}^{n + 1} $ as a combinatorial sequence \textup(similar comments appear in \cite{Ahlfors}, \textup{p.~14}\textup). Note that the inequalities stated in this remark are also valid for $ \{\overline{T}_i\}_{i = 0}^{n + 1} $.
\end{rem}
\subsection{The Weyl Peculiar Relation and Its Proof}\label{subsec:pec_rel}
\ \par
Hereafter, we assume that the holomorphic curve $ \mathbf{X}^{(p)} $ is non-degenerate.

As seen in Example \ref{ex:stat_index}, we have
\begin{equation}\label{eq:v_p}
  v_1\{\mathbf{X}^{(p)}\} = v_{p}.
\end{equation}
We observe that a similar phenomenon also appears in Nevanlinna theory.
To compute $ S^{1}\{\mathbf{X}^{(p)}\} $, we introduce a specific orthonormal basis such that
\begin{align}\label{eq:unitary_coord}
   & \mathbf{x} = (x_0, 0, \ldots, 0),                                                                \notag        \\
   & \mathbf{x}^{(1)} = (x_0^{(1)}, x_1^{(1)}, 0, \ldots, 0),                                                \notag \\
   & \vdots                                                                                                         \\
   & \mathbf{x}^{(n)} = (x_0^{(n)}, x_1^{(n)}, \ldots, x_{n - 1}^{(n)}, x_n^{(n)})\notag
\end{align}
holds at $ z_0 $.
This moving frame is called the \textbf{Frenet frame}.

Then the only non-vanishing Pl\"{u}cker coordinate $ (X^{(p)})_{(i_0, i_1, \ldots, i_{p - 1})} $ of $ \mathbf{X}^{(p)} $ at $ z_0 $ occurs when $ (i_0, i_1, \ldots, i_{p - 1}) = \mathbf{1} \coloneqq (0, 1, \ldots, p - 1) $:
\begin{equation*}
  (X^{(p)})_{\mathbf{1}} = x_0 x_1^{(1)} \cdots x_{p - 1}^{(p - 1)}.
\end{equation*}
On the other hand,  the non-vanishing Pl\"{u}cker coordinates $ (\pdif{1})_{(i_0, i_1, \ldots, i_{p - 1})} $ of $ \dif{1} $ occur when $ (i_0, i_1, \ldots, i_{p - 1}) $ is either $ \mathbf{1} $ or $ \mathbf{2} \coloneqq (0, 1, \ldots, p - 2, p)$. Thus, we have
\begin{equation*}
  (\pdif{1})_{\mathbf{1}} = x_0 x_1^{(1)} \cdots x_{p - 2}^{(p - 2)} x_{p - 1}^{(p)}, \quad
  (\pdif{1})_{\mathbf{2}} = x_0 x_1^{(1)} \cdots x_{p - 2}^{(p - 2)} x_p^{(p)}.
\end{equation*}
These computations imply
\begin{equation*}
  \begin{split}
    |\mathbf{X}^{(p)} \wedge \dif{1}|
    & = |x_0 x_1^{(1)} \cdots x_{p - 1}^{(p - 1)} \cdot x_0 x_1^{(1)} \cdots x_{p - 2}^{(p - 2)} x_p^{(p)}|                          \\
    & = |x_0 x_1^{(1)} \cdots x_{p - 2}^{(p - 2)}||x_0 x_1^{(1)} \cdots x_{p}^{(p)}| = |\mathbf{X}^{(p - 1)}||\mathbf{X}^{(p + 1)}|.
  \end{split}
\end{equation*}
Therefore,
\begin{equation*}
  S^{1}\{\mathbf{X}^{(p)}\}
  = 2\frac{|\mathbf{X}^{(p)} \wedge \dif{1}|^2}{|\mathbf{X}^{(p)}|^4} = 2\frac{|\mathbf{X}^{(p - 1)}|^2|\mathbf{X}^{(p + 1)}|^2}{|\mathbf{X}^{(p)}|^4} = S^{p}.
\end{equation*}
Consequently, we have proved the following:
\begin{equation}\label{eq:T_p-omega_p}
  T_1\{\mathbf{X}^{(p)}\} = T_{p}, \quad \Omega_1\{\mathbf{X}^{(p)}\} = \Omega_{p}.
\end{equation}
Equalities \eqref{eq:v_p} and \eqref{eq:T_p-omega_p} are parallel. On the basis of this observation, H. and F. J. Weyl discovered that similar relations hold for $ T_2 $ and $ \Omega_2 $. These are the \textbf{Weyl peculiar relations} as mentioned in the introduction.
\begin{thm}[Weyl peculiar relation; \cite{Weyl_2}, pp.~160--162.]\label{thm:pec_rel}

  For each $ 1 \leq p \leq n $, we have
  \begin{equation*}
    T_2\{\mathbf{X}^{(p)}\} = T_{p - 1} + T_{p + 1}, \quad \Omega_2\{\mathbf{X}^{(p)}\} \geq \max(\Omega_{p - 1}, \Omega_{p + 1}).
  \end{equation*}
  The second inequality is parallel to
  \begin{equation*}
    v_2\{\mathbf{X}^{(p)}\} = \min(v_{p - 1}, v_{p + 1}).
  \end{equation*}
  which appears in \textup{Example~\ref{ex:stat_index}} for $ i = 2 $.
\end{thm}

\begin{proof}
  Since $ T_0 = 0$, the statement is trivial when $ p = 1 $ (in which case $ \Omega_{0} $ is not considered). We define the \textbf{dual curve} $ \bm{\xi} $ of $ \mathbf{x} $ in the dual space $ (\mathbb{C}^{n + 1})^{\ast} \simeq \mathbb{C}^{n + 1} $ by
  $$ \star\bm{\xi} \coloneqq \mathbf{x} \wedge \mathbf{x}^{(1)} \wedge \cdots \wedge \mathbf{x}^{(n - 1)} = \mathbf{X}^{(n)},$$
  where $ \star $ denotes the Hodge star operator on $ (\mathbb{C}^{n + 1})^{*} $.
  The following identity holds:
  \begin{equation*}
    S^{p} = S^{(n + 1) - p}\{\bm{\xi}\}.
  \end{equation*}
  To verify this, we use the following formula (see \cite{Weyl_2}, p.~49):
  \begin{equation*}
    \star (\bm{\xi} \wedge \bm{\xi}^{(1)} \wedge \cdots \wedge \bm{\xi}^{(p - 1)}) = W^{p - 1}\mathbf{x} \wedge \mathbf{x}^{(1)} \wedge \cdots \wedge \mathbf{x}^{(n - p)},
  \end{equation*}
  where $ W $ is the Wronskian of $ \mathbf{x} $. From this, we obtain
  \begin{equation*}
    \begin{split}
      S^p & = 2\frac{|\mathbf{X}^{(p - 1)}|^2|\mathbf{X}^{(p + 1)}|^2}{|\mathbf{X}^{(p)}|^4} = 2\frac{|\mathbf{x} \wedge \mathbf{x}^{(1)} \wedge \cdots \wedge \mathbf{x}^{(p - 2)}|^2|\mathbf{x} \wedge \mathbf{x}^{(1)} \wedge \cdots \wedge \mathbf{x}^{(p)}|^2}{|\mathbf{x} \wedge \mathbf{x}^{(1)} \wedge \cdots \wedge \mathbf{x}^{(p - 1)}|^4} \\
      & = 2\frac{|\bm{\xi} \wedge \bm{\xi}^{(1)} \wedge \cdots \wedge \bm{\xi}^{((n + 1) - p - 2)}|^2|\bm{\xi} \wedge \bm{\xi}^{(1)} \wedge \cdots \wedge \bm{\xi}^{((n + 1) - p)}|^2}{|\bm{\xi} \wedge \bm{\xi}^{(1)} \wedge \cdots \wedge \bm{\xi}^{((n + 1) - p - 1)}|^4} = S^{(n + 1) - p}\{\bm{\xi}\}.
    \end{split}
  \end{equation*}
  In particular, we have $ S^{n - 1} = S^{2}\{\bm{\xi}\} $. Hence, we derive
  \begin{equation*}
    T_2\{\mathbf{X}^{(n)}\} = T_2\{\bm{\xi}\} = T_{n - 1} = T_{n - 1} + T_{n + 1}, \quad \Omega_2\{\mathbf{X}^{(n)}\} = \Omega_2\{\bm{\xi}\} = \Omega_{n - 1}.
  \end{equation*}
  This proves our claim for $ p = n $ (in which case $ \Omega_{n + 1} $ is not considered).

  We proceed by following the Weyls' argument for $ 2 \leq p \leq n - 1 $ (this range is equivalent to $ i = 2 \leq \min(n - p + 1, p)$). It suffices to show that
  \begin{equation}\label{eq:S^2}
    S^2\{\mathbf{X}^{(p)}\} = S^{p - 1} + S^{p + 1}.
  \end{equation}
  Indeed, from this, the first equality in Theorem~\ref{thm:pec_rel} follows directly, and the second inequality can be deduced from
  \begin{equation*}
    \log(S^{p - 1} + S^{p + 1}) \geq \max(\log S^{p - 1}, \log S^{p + 1}).
  \end{equation*}
  We prove \eqref{eq:S^2} by direct computation. Each term of $\dif{2} $ in Lemma~\ref{lem:diff_formula} is denoted by $ \dif{2}_{\;\mathbf{3}}$ and $\dif{2}_{\;\mathbf{4}}$, respectively:
  \begin{equation*}
    \dif{2} = \dif{2}_{\;\mathbf{3}} + \dif{2}_{\;\mathbf{4}},
  \end{equation*}
  where $ \mathbf{3} \coloneqq (0, 1, \ldots, p - 3, p - 1, p) $ and $ \mathbf{4} \coloneqq (0, 1, \ldots, p - 2, p + 1) $.
  The non-vanishing coordinates $ (\pdif{2})_{(i_0, i_1, \ldots, i_{p - 1})} $ of $ \dif{2} $ occur when $ (i_0, i_1, \ldots, i_{p - 1}) $ is equal to $ \mathbf{1} $, $ \mathbf{2} $, $ \mathbf{3} $, or $ \mathbf{4} $. More specifically, the Pl\"{u}cker coordinate $ (\pdif{2}_{\;\mathbf{3}})_{(i_0, i_1, \ldots, i_{p - 1})} $ vanishes unless $ (i_0, i_1, \ldots, i_{p - 1}) $ is equal to $ \mathbf{1}$, $\mathbf{2}$, or $ \mathbf{3} $.
  Similarly, $ (\pdif{2}_{\;\mathbf{4}})_{(i_0, i_1, \ldots, i_{p - 1})} $ vanishes unless $ (i_0, i_1, \ldots, i_{p - 1}) $ is equal to $ \mathbf{1}, \mathbf{2} $, or $ \mathbf{4} $. Using the Frenet frame, we compute $ (\pdif{2}_{\;\mathbf{3}})_{\mathbf{3}} $ and $ (\pdif{2}_{\;\mathbf{4}})_{\mathbf{4}} $:
  \begin{equation*}
    (\pdif{2}_{\;\mathbf{3}})_{\mathbf{3}} = x_0x_1^{(1)} \cdots x_{p - 3}^{(p - 3)} x_{p - 1}^{(p - 1)}x_{p}^{(p)},\quad (\pdif{2}_{\;\mathbf{4}})_{\mathbf{4}} = x_0 x_1^{(1)} \cdots x_{p - 2}^{(p - 2)} x_{p + 1}^{(p + 1)}.
  \end{equation*}
  Thus, when we regard $ \mathbf{X}^{(p)} \wedge \dif{1} \wedge \dif{2} $ as a holomorphic curve in $ \mathbb{P}(\bigwedge^{3}\mathbb{C}^{\binom{n + 1}{p}}) $ via the Pl\"{u}cker embedding, the only non-vanishing Pl\"{u}cker coordinates are those with indices $ (i_0, i_1, i_2) = (\mathbf{1}, \mathbf{2}, \mathbf{3}) $ and $ (\mathbf{1}, \mathbf{2}, \mathbf{4}) $.
  Hence, we have
  \begin{align}\label{eq:wedge_norm}
    |\mathbf{X}^{(p)} \wedge \dif{1} \wedge \dif{2}|^2 & = |(\mathbf{X}^{(p)} \wedge \dif{1} \wedge \dif{2})_{(\mathbf{1}, \mathbf{2}, \mathbf{3})}|^2 + |(\mathbf{X}^{(p)} \wedge \dif{1} \wedge \dif{2})_{(\mathbf{1}, \mathbf{2}, \mathbf{4})}|^2   \notag \\
                                                       & = |(X^{(p)})_{\mathbf{1}} (\pdif{1})_{\mathbf{2}}|^2(|(\pdif{2}_{\;\mathbf{3}})_{\mathbf{3}}|^2 + |(\pdif{2}_{\;\mathbf{4}})_{\mathbf{4}}|^2)    \notag                                              \\
                                                       & = |\mathbf{X}^{(p)} \wedge \dif{1}|^2(|(\pdif{2}_{\;\mathbf{3}})_{\mathbf{3}}|^2 + |(\pdif{2}_{\;\mathbf{4}})_{\mathbf{4}}|^2).
  \end{align}
  From these calculations, we obtain
  \begin{equation*}
    \begin{split}
      S^2\{\mathbf{X}^{(p)}\}
      & = 2\frac{|\mathbf{X}^{(p)}|^2|\mathbf{X}^{(p)} \wedge \dif{1} \wedge \dif{2}|^2}{|\mathbf{X}^{(p)} \wedge \dif{1}|^4} = 2\frac{|\mathbf{X}^{(p)}|^2(|(\pdif{2}_{\;\mathbf{3}})_{\mathbf{3}}|^2 + |(\pdif{2}_{\;\mathbf{4}})_{\mathbf{4}}|^2)}{|\mathbf{X}^{(p)} \wedge \dif{1}|^2}       \\
      & = 2\frac{|(\pdif{2}_{\;\mathbf{3}})_{\mathbf{3}}|^2 + |(\pdif{2}_{\;\mathbf{4}})_{\mathbf{4}}|^2}{|\mathbf{X}^{(p)}|^2} : \frac{|\mathbf{X}^{(p)} \wedge \dif{1}|^2}{|\mathbf{X}^{(p)}|^4}                                                                                               \\
      & = 2\left(\frac{|x_p^{(p)}|^2}{|x_{p - 2}^{(p - 2)}|^2} + \frac{|x_{p + 1}^{(p + 1)}|^2}{|x_{p - 1}^{(p - 1)}|^2}\right) : \frac{|x_p^{(p)}|^2}{|x_{p - 1}^{(p - 1)}|^2} = 2\frac{|x_{p - 1}^{(p - 1)}|^2}{|x_{p - 2}^{(p - 2)}|^2} + 2\frac{|x_{p + 1}^{(p + 1)}|^2}{|x_p^{(p)}|^2}.
    \end{split}
  \end{equation*}
  This proves our claim, since $ S^{p - 1} $ and $ S^{p + 1} $ are given by
  \begin{equation*}
    S^{p - 1} = 2\frac{|\mathbf{X}^{(p - 2)}|^2 |\mathbf{X}^{(p)}|^2}{|\mathbf{X}^{(p - 1)}|^4} = 2\frac{|x_{p - 1}^{(p - 1)}|^2}{|x_{p - 2}^{(p - 2)}|^2},\quad  S^{p + 1} = 2\frac{|\mathbf{X}^{(p)}|^2|\mathbf{X}^{(p + 2)}|^2}{|\mathbf{X}^{(p + 1)}|^2} = 2\frac{|x_{p + 1}^{(p + 1)}|^2}{|x_p^{(p)}|^2}.
  \end{equation*}
\end{proof}
\begin{cor}
  If we set $ T_2\{\mathbf{X}^{(0)}\} = T_2\{\mathbf{X}^{(n + 1)}\} = 0 $, then the sequence $ \{T_2\{\mathbf{X}^{(p)}\}\}_{p = 0}^{n + 1} $ is concave modulo $ O(\log T) $ .
\end{cor}
\begin{proof}
  For $ 2 \leq p \leq n - 1 $, we have
  \begin{equation*}
    \begin{split}
      T_2\{\mathbf{X}^{(p - 1)}\} -2T_2\{\mathbf{X}^{(p)}\} + T_2\{\mathbf{X}^{(p + 1)}\} & = (T_{p - 2} + T_p) - 2(T_{p - 1} + T_{p + 1}) + (T_{p} + T_{p + 2}) \\
      & = (T_{p - 2} - 2T_{p - 1} + T_p) + (T_p - 2T_{p + 1} + T_{p + 2}).
    \end{split}
  \end{equation*}
  Thus, from \eqref{eq:concave_rel}, our claim follows. The cases $ p = 1 $ and $ p = n $ can be proved in a similar manner.
\end{proof}
\section{Generalization of the Weyl Peculiar Relation}\label{sec:gen_pec_rel}

In this section, we generalize the Weyl peculiar relation for $ 3 \leq i \leq p(n - p + 1) $  (Theorem~\ref{thm:gen_pec_rel}). Our strategy is to develop an analogy with the method used in the proof of the Weyl peculiar relation presented in the preceding section (Section~\ref{sec:Weyl_pec_rel}, Subsection~\ref{subsec:pec_rel}).
We also note that similarities also appear with the method for evaluating stationary indices (Section~\ref{sec:assoc_curve}, Subsection~\ref{subsec:stat_index}).
\subsection{Proof of the Generalized Weyl Peculiar Relation}
\ \par
Recall that for each $ \sigma= (i_0, i_1, \ldots, i_{p - 1}) \in \binom{[\infty]}{p}_{(k_i)} $, the quantity $ \dif{i}_{\,\sigma} $ is defined as $\mathbf{x}^{(i_0)} \wedge \mathbf{x}^{(i_1)} \wedge \cdots \wedge \mathbf{x}^{(i_{p - 1})}$; see Lemma~\ref{lem:diff_formula}.

For $ (i_0, i_1, \ldots, i_{p - 1}) , (i_0', i_1', \ldots, i_{p - 1}') \in \binom{[\infty]}{p}$, we write
\begin{equation*}
  (i_0, i_1, \ldots, i_{p - 1}) \prec (i_0', i_1', \ldots, i_{p - 1}'),
\end{equation*}
if $ i_k < i_k' $ for some $ k $. (We remark that the symbol $ \prec $ does not represent an order.)
\begin{lem}\label{lem:coord_zero}
  Let $ \sigma \in \binom{[\infty]}{p}_{(k_i)} $ and $ \tau \in \binom{[n + 1]}{p} $.
  If $ \sigma \prec \tau $, then
  $(\pdif{i}_{\;\sigma})_{\tau} = 0$.
\end{lem}

\begin{proof}
  Let $ \sigma = (j_0, j_1, \ldots, j_{p - 1})$, $0 \leq j_0 < j_1 < \cdots < j_{p - 1} $ and $ \tau = (i_0, i_1, \ldots, i_{p - 1})$, $0 \leq i_0 < i_1 < \cdots < i_{p - 1} \leq n $. It suffices to show that there exists an integer $ 0 \leq j_s < n $ such that $ j_s < i_s $. Indeed, from \eqref{eq:unitary_coord}, it follows that $ x_{i_s}^{(j_s)} = 0 $, and hence $ x_{i_k}^{(j_l)} = 0 $ for every $ s \leq k \leq p - 1$ and $0 \leq l \leq s $. Therefore,
  \begin{equation*}
    (\pdif{i}_{\;\sigma})_{\tau} =
    \begin{vmatrix}
      x_{i_0}^{(j_0)} & x_{i_1}^{(j_0)} & \cdots & x_{i_s}^{(j_0)} & x_{i_{s + 1}}^{(j_0)} & \cdots & x_{i_{p - 1}}^{(j_0)} \\
      x_{i_0}^{(j_1)} & x_{i_1}^{(j_1)} & \cdots & x_{i_s}^{(j_1)} & x_{i_{s + 1}}^{(j_1)} & \cdots & x_{i_{p - 1}}^{(j_1)} \\
      \vdots          & \vdots          & \ddots & \vdots          & \vdots                & \ddots & \vdots                \\
      x_{i_0}^{(j_s)} & x_{i_1}^{(j_s)} & \cdots & x_{i_s}^{(j_s)} & x_{i_{s + 1}}^{(j_s)} & \cdots & x_{i_{p - 1}}^{(j_s)} \\
      \ast            & \ast            & \cdots & \ast            & \ast                  & \cdots & \ast                  \\
      \vdots          & \vdots          & \ddots & \vdots          & \vdots                & \ddots & \vdots                \\
      \ast            & \ast            & \cdots & \ast            & \ast                  & \cdots & \ast
    \end{vmatrix}
    =
    \begin{vmatrix}
      x_{i_0}^{(j_0)} & x_{i_1}^{(j_0)} & \cdots & 0      \\
      x_{i_0}^{(j_1)} & x_{i_1}^{(j_1)} & \cdots & 0      \\
      \vdots          & \vdots          & \ddots & \vdots \\
      x_{i_0}^{(j_s)} & x_{i_1}^{(j_s)} & \cdots & 0
    \end{vmatrix} \times \ast
    = 0.
  \end{equation*}
  If $ j_{p - 1} \leq n $, by definition of $ \prec $, there exists $ 0 \leq s \leq p - 1 $ such that $ j_s < i_s \leq n $. Thus, we may assume that $ j_{p - 1} > n $. If $ j_0 > n $, this contradicts the assumption $ \sigma \prec \tau $. Hence, we can find an integer $ 0 \leq s \leq p - 2 $ such that $ 0 \leq j_s \leq n $ and $ j_{s + 1} > n$.
  Suppose that $ j_k \geq i_k $ holds for all $ 0 \leq k \leq s $. Since $ \sigma \prec \tau $, there exists $ s + 1 \leq s' \leq p - 1 $ such that $ j_{s'} < i_{s'}$. This leads to a contradiction, as it implies $ n < j_{s + 1} \leq j_{s'} < i_{s'} \leq n $. Hence, we can always find some integer $ 0 \leq j_s < n $ satisfying $ j_s < i_s $.
\end{proof}
The following lemma is an analogue of \eqref{eq:wedge_norm}.
\begin{lem}\label{lem:norm_compare}
  Let $ n $ and $ p \, (\leq n)$ be positive integers. Fix $ i \leq p(n - p + 1) $. Then
  \begin{align*}
    |\mathbf{X}^{(p)} \wedge \dif{1} \wedge \cdots \wedge \dif{i}|^2 \geq |\mathbf{X}^{(p)} \wedge \dif{1} \wedge \cdots \wedge \dif{i - 1}|^2 \sum_{\sigma \in \binom{[n + 1]}{p}_{(k_i)}}f_{\lambda(\sigma)}^2|(\pdif{i}_{\;\sigma})_{\sigma}|^2.
  \end{align*}
\end{lem}
\begin{proof}
  Note that $ \binom{[n + 1]}{p}_{(k_i)} $ is not an empty set, which follows from our assumption $ i \leq q = p(n - p + 1) $; see \eqref{eq:dim_Grassmann}.
  By definition of the set $ \binom{[\infty]}{p}_{(k_i)} $, the sum of the components of any element $ (i_0, i_1, \ldots, i_{p - 1}) \in \binom{[\infty]}{p}_{(k_i)}$ equals $ k_i $. Hence, for every $ \sigma \in \binom{[\infty]}{p}_{(k_i)}$ and $ \tau \in \binom{[n + 1]}{p}_{(k_i)}$, we have $ \sigma \prec \tau $ if and only if $ \sigma \neq \tau $. Therefore, by Lemma~\ref{lem:diff_formula} and Lemma~\ref{lem:coord_zero}, the equality
  \begin{align*}
    (\pdif{i})_{\sigma} = \sum_{\tau \in \binom{[\infty]}{p}_{(k_i)}}f_{\lambda(\tau)}(\pdif{i}_{\;\tau})_{\sigma} = f_{\lambda(\sigma)}(\pdif{i}_{\;\sigma})_{\sigma}
  \end{align*}
  holds for all $ \sigma \in \binom{[n + 1]}{p}_{(k_i)} $, and we have $ (\pdif{i})_{\tau} = 0 $ for all $ \tau \in \bigcup_{l > i}\binom{[n + 1]}{p}_{(k_l)} $.
  Using this formula, we obtain the following:
  \begin{equation*}
    \begin{split}
      & |\mathbf{X}^{(p)} \wedge \dif{1} \wedge \cdots \wedge \dif{i}|^2 = \sum_{(\sigma_0, \sigma_1, \ldots, \sigma_i)}|(\mathbf{X}^{(p)} \wedge \dif{1} \wedge \cdots \wedge \dif{i})_{(\sigma_0, \sigma_1, \ldots, \sigma_i)}|^2 \\
      & \geq \sum_{\substack{(\sigma_0, \sigma_1, \ldots, \sigma_i)                                                                                                                                                                        \\ \sigma_i \in \binom{[n + 1]}{p}_{(k_i)} \\ \sigma_s \notin \binom{[n + 1]}{p}_{(k_i)} \, (s \neq i)}}|(\mathbf{X}^{(p)} \wedge \dif{1} \wedge \cdots \wedge \dif{i})_{(\sigma_0, \sigma_1, \ldots, \sigma_i)}|^2 \\
      & = |\mathbf{X}^{(p)} \wedge \dif{1} \wedge \cdots \wedge \dif{i - 1}|^2\sum_{\sigma \in \binom{[n + 1]}{p}_{(k_i)}}f_{\lambda(\sigma)}^2|(\pdif{i}_{\;\sigma})_{\sigma}|^2.
    \end{split}
  \end{equation*}
  Here, the first sum runs over all  lexicographically ordered sequences $ (\sigma_0, \sigma_1, \ldots, \sigma_i)$, where $ \sigma_s = (i_0^{(s)}, i_1^{(s)}, \ldots, i_{p - 1}^{(s)}) \in \bigcup_{0 \leq l \leq i}\binom{[n + 1]}{p}_{(k_l)} $ for $ 0 \leq s \leq i $.
  The second sum runs over all lexicographically such sequences $ (\sigma_0, \sigma_1, \ldots, \sigma_i) $ with $ \sigma_i \in \binom{[n + 1]}{p}_{(k_i)} $ and $ \sigma_0, \sigma_1, \ldots, \sigma_{i - 1} \in \bigcup_{0 \leq l \leq i - 1}\binom{[n + 1]}{p}_{(k_l)} $.
\end{proof}
Now, we provide a generalization of the formula \eqref{eq:S^2}.
\begin{prop}\label{prop:S_ineq}
  \begin{equation*}
    \prod_{k = 1}^i S^k\{\mathbf{X}^{(p)}\} \geq \sum_{\sigma \in \binom{[n + 1]}{p}_{(k_i)}} f_{\lambda(\sigma)}^2 \prod_{k = 1}^n(S^k)^{n_{\lambda(\sigma)}(k)}.
  \end{equation*}
\end{prop}

\begin{proof}
  By Lemma~\ref{lem:norm_compare}, we have
  \begin{align*}
    \widetilde{S^i}\{\mathbf{X}^{(p)}\}
     & = \frac{|\mathbf{X}^{(p)} \wedge \dif{1} \wedge \cdots \wedge \dif{i - 2}|^2|\mathbf{X}^{(p)} \wedge \dif{1} \wedge \cdots \wedge \dif{i}|^2}{|\mathbf{X}^{(p)} \wedge \dif{1} \wedge \cdots \wedge \dif{i - 1}|^4}                                                                               \\
     & \geq \frac{|\mathbf{X}^{(p)} \wedge \dif{1} \wedge \cdots \wedge \dif{i - 2}|^2}{|\mathbf{X}^{(p)} \wedge \dif{1} \wedge \cdots \wedge \dif{i - 1}|^2}\sum_{\sigma \in \binom{[n + 1]}{p}_{(k_i)}}f_{\lambda(\sigma)}^2|(\pdif{i}_{\;\sigma})_{\sigma}|^2                                         \\
     & = \frac{|\mathbf{X}^{(p)} \wedge \dif{1} \wedge \cdots \wedge \dif{i - 3}|^2}{|\mathbf{X}^{(p)} \wedge \dif{1} \wedge \cdots \wedge \dif{i - 2}|^2 \widetilde{S^{i - 1}}\{\mathbf{X}^{(p)}\}}\sum_{\sigma \in \binom{[n + 1]}{p}_{(k_i)}}f_{\lambda(\sigma)}^2|(\pdif{i}_{\;\sigma})_{\sigma}|^2. \\
  \end{align*}
  Repeating a similar computation, we obtain the following inequality:
  \begin{equation*}
    \widetilde{S^i}\{\mathbf{X}^{(p)}\} \geq \sum_{\sigma \in \binom{[n + 1]}{p}_{(k_i)}}f_{\lambda(\sigma)}^2\frac{|(\pdif{i}_{\;\sigma})_{\sigma}|^2}{|\mathbf{X}^{(p)}|^2\prod_{k = 1}^{i - 1}\widetilde{S^k}\{\mathbf{X}^{(p)}\}}.
  \end{equation*}
  Using the coordinates given in \eqref{eq:unitary_coord}, we can compute the right-hand side of this inequality. First, for every $ 1 \leq k \leq n $, we have
  \begin{equation*}
    S^k = 2\frac{|\mathbf{X}^{(k - 1)}|^2|\mathbf{X}^{(k + 1)}|^2}{|\mathbf{X}^{(k)}|^4} = 2\frac{|x_0 x_1^{(1)}\cdots x_{k - 2}^{(k - 2)}|^2|x_0 x_1^{(1)} \cdots x_{k}^{(k)}|^2}{|x_0 x_1 \cdots x_{k - 1}^{(k - 1)}|^4} = 2\frac{|x_k^{(k)}|^2}{|x_{k - 1}^{(k - 1)}|^2}.
  \end{equation*}
  We observe that there is a decomposition similar to \eqref{eq:delta_decomp}:
  $$\frac{|x_{i_k}^{(i_k)}|^2}{|x_k^{(k)}|^2} = \frac{|x_{k + 1}^{(k + 1)}|^2}{|x_k^{(k)}|^2}\frac{|x_{k + 2}^{(k + 2)}|^2}{|x_{k + 1}^{(k + 1)}|^2} \cdots \frac{|x_{i_k}^{(i_k)}|^2}{|x_{i_{k - 1}}^{(i_{k - 1})}|^2} = \widetilde{S^{k + 1}}\widetilde{S^{k + 2}} \cdots \widetilde{S^{i_k}}.$$
  Thus, it follows that
  \begin{align*}
    \frac{|(\pdif{i}_{\;\sigma})_{\sigma}|^2}{|\mathbf{X}^{(p)}|^2} = \frac{|x_{i_0}^{(i_0)} x_{i_1}^{(i_1)} \cdots x_{i_{p - 1}}^{(i_{p - 1})}|^2}{|x_0 x_1^{(1)} \cdots x_{p - 1}^{(p - 1)}|^2} = \prod_{k = 0}^{p - 1}\widetilde{S^{k + 1}}\widetilde{S^{k + 2}} \cdots \widetilde{S^{i_k}}.
  \end{align*}
  Furthermore, we have
  \begin{equation*}
    \frac{1}{2^i}\prod_{k = 1}^n(S^k)^{n_{\lambda(\sigma)}(k)} = \prod_{k = 1}^n(\widetilde{S^k})^{n_{\lambda(\sigma)}(k)} = \prod_{k = 1}^n(\widetilde{S^k})^{b_{\sigma}(k)} = \prod_{k = 0}^{p - 1}\widetilde{S^{k + 1}} \widetilde{S^{k + 2}} \cdots \widetilde{S^{i_k}},
  \end{equation*}
  where $ b_{\sigma}(k) $ denotes the same quantity as in the proof of Proposition~\ref{prop:ord_ineq}. Combining these, we obtain the desired inequality.
\end{proof}

\begin{thm}[Generalized Weyl Peculiar Relation \textit{for} $ \Omega_k\{\mathbf{X}^{(p)}\} $]\label{thm:gen_pec_rel_omega}
  For every integer $ 1 \leq i \leq p(n - p + 1) $, the following inequality holds:
  \begin{equation*}
    \sum_{k = 1}^i \Omega_k\{\mathbf{X}^{(p)}\} \geq \max_{\sigma \in \binom{[n + 1]}{p}_{(k_i)}}\left(\sum_{k = 1}^n n_{\lambda(\sigma)}(k)\Omega_k\right).
  \end{equation*}
  This inequality is parallel to \textup{Proposition~\ref{prop:ord_ineq}}:
  $$\sum_{k = 1}^i v_i\{\mathbf{X}^{(p)}\} \leq \min_{\sigma \in V(i)}\phi_p(\lambda(\sigma)) \leq \min_{\sigma \in \binom{[n + 1]}{p}_{(k_i)}}\left(\sum_{k = 1}^n n_{\lambda(\sigma)}(k)v_k\right).$$
\end{thm}

\begin{proof}
  By Proposition~\ref{prop:S_ineq}, we have
  \begin{equation*}
    \begin{split}
      \sum_{k = 1}^i \Omega_k\{\mathbf{X}^{(p)}\} & = \sum_{k = 1}^i \frac{1}{4\pi}\int_0^{2\pi}\log \widetilde{S^k}\{\mathbf{X}^{(p)}\}d\theta                                                                                                   \\
      & = \frac{1}{4\pi}\int_0^{2\pi}\log \left(\prod_{k = 1}^i \widetilde{S^k}\{\mathbf{X}^{(p)}\}\right)d\theta                                                                                     \\
      & \geq \frac{1}{4\pi}\int_0^{2\pi}\log \left(\sum_{\sigma \in \binom{[n + 1]}{p}_{(k_i)}}f_{\lambda(\sigma)}^2 \prod_{k = 1}^n(\widetilde{S^k})^{n_{\lambda(\sigma)}(k)}\right)d\theta \\
      & \geq \max_{\sigma \in \binom{[n + 1]}{p}_{(k_i)}}\left(\sum_{k = 1}^n n_{\lambda(\sigma)}(k)\frac{1}{4\pi}\int_0^{2\pi}\log \widetilde{S^k}d\theta\right)                             \\
      & = \max_{\sigma \in \binom{[n + 1]}{p}_{(k_i)}}\left(\sum_{k = 1}^n n_{\lambda(\sigma)}(k)\Omega_k\right).
    \end{split}
  \end{equation*}
\end{proof}
The following is an analogue of Lemma~\ref{lem:d-seq}.
\begin{lem}\label{lem:V_ineq}
  For every $\sigma \in \binom{[n + 1]}{p}_{(k_i)}$, we have
  \begin{equation*}
    \sum_{k = 1}^n n_{\lambda(\sigma)}(k)V_k - \sum_{k = 1}^{i}V_k\{\mathbf{X}^{(p)}\} \geq 0.
  \end{equation*}
\end{lem}
\begin{proof}
  It suffices to show that for every $ z \in \Delta(r) $,
  \begin{equation*}
    \sum_{k = 1}^n n_{\lambda(\sigma)}(k)(v_k(z) - 1) - \sum_{k = 1}^{i}(v_k\{\mathbf{X}^{(p)}\}(z) - 1) \geq 0.
  \end{equation*}
  The left-hand side is bounded below by
  \begin{equation}\label{eq:min}
    \min_{\sigma \in \binom{[n + 1]}{p}_{(k_i)}}\left(\sum_{k  = 1}^n n_{\lambda(\sigma)}(k)v_k(z)\right) - \sum_{k = 1}^i v_k\{\mathbf{X}^{(p)}\}(z),
  \end{equation}
  where we use the fact that $ \sum_{k = 1}^n n_{\lambda(\sigma)}(k) = i $ $ (\sigma \in \binom{[n + 1]}{p}_{(k_i)})$.
  Recall the definitions of $ V(i) $ (see \eqref{eq:V(i)}) and $ \phi_p(\lambda) $ (Definition~\ref{def:phi}). By Proposition~\ref{prop:ord_ineq}, we have
  \begin{equation*}
    \sum_{k = 1}^{i}v_i\{\mathbf{X}^{(p)}\} \leq \min_{\sigma \in V(i)}\phi_p(\lambda(\sigma)).
  \end{equation*}
  Hence, \eqref{eq:min} is bounded below by
  \begin{align*}
    \min_{\sigma \in \binom{[n + 1]}{p}_{(k_i)}}\phi_p(\lambda(\sigma))(z) - \min_{\sigma \in V(i)}\phi_p(\lambda(\sigma))(z).
  \end{align*}
  Since $ V(i) \supseteq  \binom{[n + 1]}{p}_{(k_{i})} $, this quantity is greater than or equal to $ 0 $.
\end{proof}
The following result provides an explicit connection between $ i\overline{T}_p $ and $ \overline{T}_i\{\mathbf{X}^{(p)}\} $.
\begin{lem}\label{lem:iT_p-T_i{X^p}_rel}
  For every $ 1 \leq i \leq \binom{n + 1}{p} $, the following equality holds:
  \begin{equation*}
    \sum_{s = 1}^{i - 1}\sum_{k = 1}^s \Omega_k\{\mathbf{X}^{(p)}\} + i\overline{T}_p = \overline{T}_i\{\mathbf{X}^{(p)}\} + O(1).
  \end{equation*}
\end{lem}
\begin{proof}
  This statemant directly follows from Theorem~\ref{thm:Plucker_formula}. Indeed, we have
  \begin{equation*}
    \begin{split}
      \sum_{s = 1}^{i - 1}\sum_{k = 1}^s \Omega_k\{\mathbf{X}^{(p)}\} & = \sum_{s = 1}^{i - 1}\sum_{k = 1}^s (\overline{T}_{k - 1}\{\mathbf{X}^{(p)}\} - 2\overline{T}_k\{\mathbf{X}^{(p)}\} + \overline{T}_{k + 1}\{\mathbf{X}^{(p)}\}) + O(1) \\
      & = \sum_{s = 1}^{i - 1}(-\overline{T}_1\{\mathbf{X}^{(p)}\} - \overline{T}_s\{\mathbf{X}^{(p)}\} + \overline{T}_{s + 1}\{\mathbf{X}^{(p)}\}) + O(1)                      \\
      & = -i\overline{T}_p + \overline{T}_i\{\mathbf{X}^{(p)}\} + O(1),
    \end{split}
  \end{equation*}
  where we use the formula \eqref{eq:T_p-omega_p} (and \eqref{eq:iN_p-N_i{X^p}_rel}).
\end{proof}
We are now ready to prove the generalized Weyl peculiar relation.
\begin{thm}[Theorem A, Generalized Weyl Peculiar Relation \textit{for} $ T_i\{\mathbf{X}^{(p)}\} $]\label{thm:gen_pec_rel}
  Let $ n $ and $ p \ (\leq n) $ be positive integers, and let $ i $ be an integer satisfying $ 1 \leq i \leq p(n - p + 1) $. Assume that $ \mathbf{X}^{(p)} $ is non-degenerate as a holomorphic curve. Then the following inequality holds:
  \begin{equation}\label{eq:gen_pec_ineq}
    \sum_{s = 1}^{i - 1} \max_{\sigma \in \binom{[n + 1]}{p}_{(k_s)}}\left(\sum_{k = 1}^n n_{\lambda(\sigma)}(k)(T_{k - 1} - 2T_k + T_{k + 1})\right)
    + i T_p
    \leq T_i\{\mathbf{X}^{(p)}\} + O(1).
  \end{equation}
\end{thm}
\begin{proof}
  Lemma~\ref{lem:iT_p-T_i{X^p}_rel} and Theorem~\ref{thm:gen_pec_rel_omega} imply
  \begin{equation*}
    \sum_{s = 1}^{i - 1}\max_{\sigma \in \binom{[n + 1]}{p}_{(k_s)}}\left(\sum_{k = 1}^n n_{\lambda(\sigma)}(k)\Omega_k\right) + i\overline{T}_p \leq \overline{T}_i\{\mathbf{X}^{(p)}\} + O(1).
  \end{equation*}
  The Pl\"{u}cker formula (Theorem~\ref{thm:Plucker_formula}) implies
  \begin{equation}\label{eq:combined_ineq}
    \begin{split}
      \sum_{s = 1}^{i - 1}\max_{\sigma \in \binom{[n + 1]}{p}_{(k_s)}}\left(\sum_{k = 1}^n n_{\lambda(\sigma)}(k)(T_{k - 1} -2T_k + T_{k + 1} + V_k)\right) + iT_p & + iN_p - N_i\{\mathbf{X}^{(p)}\}  \\
      & \leq T_i\{\mathbf{X}^{(p)}\} + O(1).
    \end{split}
  \end{equation}
  An argument similar to that in Lemma~\ref{lem:iT_p-T_i{X^p}_rel} shows that
  \begin{equation}\label{eq:iN_p-N_i{X^p}_rel}
    \sum_{s = 1}^{i - 1}\sum_{k = 1}^s V_k\{\mathbf{X}^{(p)}\} + iN_p = N_i\{\mathbf{X}^{(p)}\}.
  \end{equation}
  Hence, for some $ \sigma_s \in \binom{[n + 1]}{p}_{(k_s)} $ $ (s = 1, 2, \ldots, i - 1)$, the left-hand side of \eqref{eq:combined_ineq} is bounded below by
  \begin{equation*}
    \begin{split}
      \sum_{s = 1}^{i - 1} & \sum_{k = 1}^n n_{\lambda(\sigma_s)}(k)(T_{k - 1} -2T_k + T_{k + 1})                                                                               \\
      & + iT_p + \sum_{s = 1}^{i - 1}\left(\sum_{k = 1}^n n_{\lambda(\sigma_s)}(k)V_k\right) - \sum_{s = 1}^{i - 1}\sum_{k = 1}^s V_k\{\mathbf{X}^{(p)}\}.
    \end{split}
  \end{equation*}
  Moreover, Lemma~\ref{lem:V_ineq} implies that
  \begin{equation*}
    \sum_{s = 1}^{i - 1}\left(\sum_{k = 1}^n n_{\lambda(\sigma_s)}(k)V_k\right) - \sum_{s = 1}^{i - 1}\sum_{k = 1}^s V_k\{\mathbf{X}^{(p)}\} \geq 0.
  \end{equation*}
  Therefore, we obtain
  \begin{equation*}
    \sum_{s = 1}^{i - 1}\max_{\sigma \in \binom{[n + 1]}{p}_{(k_s)}}\left(\sum_{k = 1}^n n_{\lambda(\sigma)}(k)(T_{k - 1} -2T_k + T_{k + 1})\right) + iT_p \leq T_i\{\mathbf{X}^{(p)}\} + O(1).
  \end{equation*}
\end{proof}

\subsection{Remarks on the Structure of Our Results}\label{subsec:geom_str}
\ \par
\noindent
\underline{\textbf{1.}} Recall the original relations \eqref{eq:T_p-omega_p} and Theorem~\ref{thm:pec_rel}. From the proof of the generalized Weyl peculiar relation, we see that $ T_p $, $ T_{p - 1} $, and $ T_{p + 1} $ correspond to Young diagrams of size $ 1 $ or $ 2 $ (see Figure~\ref{fig:Young_diag_size_1_2}).
\begin{figure}[htbp]
  \centering
  \begin{equation*}
    \Yvcentermath1 T_p \leftrightarrow  \yng(1) ,
    \;  T_{p - 1} \leftrightarrow \yng(1^2), \; T_{p + 1} \leftrightarrow \yng(2)
  \end{equation*}
  \caption{Young diagrams of size 1 and 2}\label{fig:Young_diag_size_1_2}
\end{figure}
A similar correspondence can also be observed for $ \Omega_p $, $ \Omega_{p - 1} $, and $ \Omega_{p + 1} $.

Moreover, as established earlier, we also obtained the following relations:
\begin{equation*}
  \Yvcentermath1
  v_1\{\mathbf{X}^{(p)}\} = \phi_p(\yng(1)), \quad v_1\{\mathbf{X}^{(p)}\} + v_2\{\mathbf{X}^{(p)}\} = \min\left(\phi_p\left(\yng(1^2)\right), \phi_p\left(\yng(2)\right)\right).
\end{equation*}

These phenomena reflect the geometric and combinatorial structure of $ \mathrm{Gr}(p, n + 1) $. Fix a basis $ \mathbf{e}_0, \mathbf{e}_1, \ldots, \mathbf{e}_{n} $ of $ \mathbb{C}^{n + 1} $. Let $ \lambda = (\lambda_0, \lambda_1, \ldots, \lambda_{p - 1}) \subseteq [(n - p + 1)^p] $ be a Young diagram. Then the \textbf{Schubert cell} $ \Omega_{\lambda} $ is defined by
\begin{equation*}
  \Omega_{\lambda} \coloneqq \{U \in \mathrm{Gr}(p, n + 1) \mid \dim(U \cap V_j) = i \Leftrightarrow n + 1 - p + i - \lambda_{i - 1} \leq j \leq n + 1 - p + i - \lambda_{i}\},
\end{equation*}
where $ V_j \coloneqq \mathrm{Span}_{\mathbb{C}}{(\mathbf{e}_0, \mathbf{e}_1, \ldots, \mathbf{e}_{j - 1})} $ $ (j = 1, 2, \ldots, n + 1)$. The closure $ X_{\lambda} \coloneqq \overline{\Omega_{\lambda}}$ is called a \textbf{Schubert variety}. It is known that $ X_{\lambda} $ can be expressed as follows:
\begin{equation*}
  X_{\lambda} = \{U \in \mathrm{Gr}(p, n + 1) \mid \dim(U \cap V_{n + 1 - p + i - \lambda_{i - 1}}) \geq i \, (1 \leq i \leq p)\}.
\end{equation*}
Since $ X_{\lambda} $ is a subvariety of $ \mathrm{Gr}(p, n + 1) $, it defines an integral (co)homology class
\begin{equation*}
  \sigma_{\lambda} \coloneqq [X_{\lambda}] \in H_{2(p(n - p + 1) - |\lambda|)}(\mathrm{Gr}(p,  n + 1), \mathbb{Z}) \simeq H^{2|\lambda|}(\mathrm{Gr}(p, n + 1), \mathbb{Z}),
\end{equation*}
via \emph{Poincar\'{e} duality}.
It is independent of the choice of basis (or flag) of $ \mathbb{C}^{n + 1} $, and is called a \textbf{Schubert class} (also called a \emph{Schubert cycle}). Moreover, we have
\begin{equation*}
  H^{*}(\mathrm{Gr}(p, n + 1), \mathbb{Z}) = \bigoplus_{\lambda \subseteq [(n - p + 1)^p]}\mathbb{Z} \cdot \sigma_{\lambda}.
\end{equation*}
The cup product structure on $ H^{*}(\mathrm{Gr}(p, n + 1), \mathbb{Z}) $ is well understood. For our purposes, we need the following special case.
\begin{thm}[\textbf{Pieri's formula}; \textit{see for instance}, \cite{Griffiths-Harris}, p.~203; \cite{Smirnov}, p.~194, Theorem~2.27]\label{thm:Pieri_formula}
  Let $ \lambda \subseteq [(n - p + 1)^p]$ be a Young diagram, and $ k \leq n - p + 1 $ be a non-negative integer. Then we obtain
  \begin{equation*}
    \sigma_{\lambda} \cdot \sigma_{[k]} = \sum_{\substack{\tau \subseteq [(n - p + 1)^p], \\ \tau \in \lambda \otimes [k]}}\sigma_{\tau},
  \end{equation*}
  where $ \lambda \otimes [k] $ denotes the set of Young diagams obtained from $ \lambda $ by adding $ k $ boxes, no two of which are in the same column.
\end{thm}

We now take the Frenet frame $ \mathbf{e}_0 = \mathbf{e}_0(z), \mathbf{e}_1 = \mathbf{e}_1(z), \ldots, \mathbf{e}_{n} = \mathbf{e}_n(z) $ corresponding to the holomorphic curve $ \mathbf{x} $. For each $ z_0 \in \mathbb{C} $, the expression \eqref{eq:unitary_coord} holds.
Hence, for each $ \sigma \in \binom{[n + 1]}{p}_{(k_i)} $, the subspace $ \dif{i}_{\;\sigma} = \dif{i}_{(i_0, i_1, \ldots, i_{p - 1})} \in \mathrm{Gr}(p, n + 1)$ can be represented as
\setcounter{MaxMatrixCols}{15}
\begin{equation*}
  \begin{pmatrix}
    x_{0}^{(i_0)}       & \cdots                      & x_{i_0}^{(i_0)}       & 0                         & \cdots & 0                     & 0                         & \cdots & 0      & 0      & \cdots & 0      \\
    x_{0}^{(i_1)}       & \cdots                      & x_{i_0}^{(i_1)}       & x_{i_0 + 1}^{(i_1)}       & \cdots & x_{i_1}^{(i_1)}       & 0                         & \cdots & 0      & 0      & \cdots & 0      \\
    \vdots              & \ddots                      & \vdots                & \vdots                    & \vdots & \vdots                & \vdots                    & \ddots & \vdots & \vdots & \ddots & \vdots \\
    x_{0}^{(i_{p - 1})} & \cdots                      & x_{i_0}^{(i_{p - 1})} & x_{i_0 + 1}^{(i_{p - 1})} & \cdots & x_{i_1}^{(i_{p - 1})} & x_{i_1 + 1}^{(i_{p - 1})} &
    \cdots              & x_{i_{p - 1}}^{(i_{p - 1})} & 0                     & \cdots                    & 0
  \end{pmatrix}
\end{equation*}
at $ z_0 $. That is, the subspace $ \dif{i}_{\;\sigma} $ is spanned by the row vectors of this matrix. This implies that
\begin{equation*}
  \dif{i}_{\;\sigma}(z_0) \in \Omega_{\widehat{\lambda(\sigma)}}(z_0),
\end{equation*}
where $ \Omega_{\widehat{\lambda(\sigma)}}(z_0) $ is the Schubert cell corresponding to $ \widehat{\lambda(\sigma)} $ in the basis $ \mathbf{e}_0(z_0), \mathbf{e}_1(z_0), \ldots, \mathbf{e}_n(z_0) $. Hence, Lemma~\ref{lem:diff_formula} gives a natural decomposition of $ \dif{i} $ into decomposable vectors $ \dif{i}_{\,\sigma} $, each of which lies in the Schubert cell $ \Omega_{\widehat{\lambda(\sigma)}}(z) $ corresponding to the Young diagram $ \widehat{\lambda(\sigma)} $ of size $ p(n - p + 1) - i $.
Moreover, each $ \dif{i}_{\;\sigma} $ lies in the subspace $ \mathrm{Span}_{\mathbb{C}}(\mathbf{e}_{\tau}, \mathbf{e}_{\sigma} \mid \tau \in \bigcup_{j < i}\binom{[n + 1]}{p}_{(k_j)}) $ by Lemma~\ref{lem:coord_zero}. Here $ \mathbf{e}_{\tau} $ is defined as $ \mathbf{e}_{j_0} \wedge \mathbf{e}_{j_1} \wedge \cdots \wedge \mathbf{e}_{j_{p - 1}} $ when $ \tau = (j_0, j_1, \ldots, j_{p - 1})$. These geometric structures seem to form the background of the Weyl peculiar relation.

Next, we investigate the case of algebraic curves. Let $ M $ be a compact Riemann surface and let $ \mathbf{x} : M \to \mathbb{P}^n $ be a non-degenrate algebraic curve.
A geometric interpretation of the original Weyl peculiar relation for algebraic curves
\begin{equation*}
  \nu_2\{\mathbf{X}^{(p)}\} = \nu_{p - 1} + \nu_{p + 1} \quad (\nu_k = \deg(\mathbf{X}^{(k)}), \nu_2\{\mathbf{X}^{(p)}\} \coloneqq \deg (\mathbf{X}^{(p)} \wedge \dif{1}))
\end{equation*}
can be found in \cite{Griffiths-Harris}, pp.~272--273.  We now examine this interpretation in light of their exposition.  We define the \emph{tangential ruled surface} \textup(also called the \emph{tangent developable}\textup) $ T(\mathbf{X}^{(p)}(M)) $ of $ \mathbf{X}^{(p)} $ by
\begin{equation*}
  T(\mathbf{X}^{(p)}(M)) \coloneqq \bigcup_{z \in M}T_z(\mathbf{X}^{(p)}(M)) \subseteq \mathrm{Gr}(p, n + 1) \subseteq \mathbb{P}^{\binom{n + 1}{p} - 1},
\end{equation*}
where the tangent line $ T_z(\mathbf{X}^{(p)}(M)) $ at $ z $ is given by
\begin{equation*}
  T_z(\mathbf{X}^{(p)}(M)) = \{(\mathbf{x} \wedge \mathbf{x}^{(1)} \wedge \cdots \wedge \mathbf{x}^{(p - 2)} \wedge (c_0 \mathbf{x}^{(p - 1)} + c_1 \mathbf{x}^{(p)}))(z)\}_{(c_0 \, : \, c_1) \in \mathbb{P}^1}.
\end{equation*}
By the definition of the degree of a map $ \deg_{\mathrm{cov}} $, we have
\begin{equation*}
  \begin{split}
    \nu_2\{\mathbf{X}^{(p)}\} & = \deg(\mathbf{X}^{(p)} \wedge \dif{1}) = ((\mathbf{X}^{(p)} \wedge \dif{1})_*[M] \cdot \sigma_1')_{\mathrm{Gr}(2, \binom{n + 1}{p})} \\
    & = \deg_{\mathrm{cov}}(\mathbf{X}^{(p)})\#(T(\mathbf{X}^{(p)}(M)) \cap \Gamma_{\binom{n + 1}{p} - 3})_{\mathbb{P}^{\binom{n + 1}{p} - 1}} \\
    & = \deg_{\mathrm{cov}}(\mathbf{X}^{(p)})\deg T(\mathbf{X}^{(p)}(M)),
  \end{split}
\end{equation*}
where $ \Gamma_{\binom{n + 1}{p} - 3} $ denotes a generic $ (\binom{n + 1}{p} - 3) $-plane in $ \mathbb{P}^{\binom{n + 1}{p} - 1} $, and $ \sigma_1' $ is the Schubert cycle on $ \mathrm{Gr}(2, \binom{n + 1}{p}) $ determined by a hyperplane section. Hence, we obtain
\begin{equation*}
  \deg_{\mathrm{cov}}(\mathbf{X}^{(p)})\deg T(\mathbf{X}^{(p)}(M)) = \nu_{p - 1} + \nu_{p + 1}.
\end{equation*}
We observe that this equality can be derived using Schubert calculus. Let $ \Gamma_j \coloneqq \mathbb{P}(V_{j + 1}) $. Then the subvarieties $ \sigma_{1, 1}(\Gamma_{n - p + 1}) $ and $ \sigma_{2}(\Gamma_{n - p - 1}) $ in $ \mathrm{Gr}(p, n + 1) $ are given by
\begin{align*}
  \sigma_{1, 1}(\Gamma_{n - p + 1})
                                 & = \{\Lambda \in \mathrm{Gr}(p, n + 1) \mid \dim(\mathbb{P}(\Lambda) \cap \Gamma_{n - p + 1}) \geq 1\}     \\
                                 & = \{\Lambda \in \mathrm{Gr}(p, n + 1) \mid \dim(\Lambda \cap V_{n - p + 2}) \geq 2 \},                    \\
  \sigma_{2}(\Gamma_{n - p - 1}) & = \{\Lambda \in \mathrm{Gr}(p, n + 1) \mid \mathbb{P}(\Lambda) \cap \Gamma_{n - p - 1} \neq \varnothing\} \\
                                 & = \{\Lambda \in \mathrm{Gr}(p, n + 1) \mid \dim(\Lambda \cap V_{n - p}) \geq 1\}.
\end{align*}
Pieri's formula (Theorem \ref{thm:Pieri_formula}) implies
\begin{equation*}
  \sigma_1^{2} = \sigma_{1, 1} + \sigma_2 = [\sigma_{1, 1}(\Gamma_{n - p + 1})] + [\sigma_2(\Gamma_{n - p - 1})].
\end{equation*}
We establish the following equivalences:
\begin{align*}
  T_z(\mathbf{X}^{(p)}(M)) \cap \sigma_{1, 1}(\Gamma_{n - p + 1}) \neq \varnothing & \Leftrightarrow \dim(\mathbf{X}^{(p - 1)}(z) \cap V_{n - p + 2}) \geq 1, \\
  T_z(\mathbf{X}^{(p)}(M)) \cap \sigma_{2}(\Gamma_{n - p - 1}) \neq \varnothing    & \Leftrightarrow \dim(\mathbf{X}^{(p + 1)}(z) \cap V_{n - p}) \geq 1.
\end{align*}
To prove the first equivalence, suppose there exists a nonzero vector $ 0 \neq \Lambda_1 \in \mathbf{X}^{(p - 1)}(z) \cap V_{n - p + 2} $. Then $ \Lambda_1 $ can be written as
\begin{equation*}
  \Lambda_1 = (a_0 \mathbf{x} + a_1 \mathbf{x}^{(1)} + \cdots + a_{p - 2} \mathbf{x}^{(p - 2)})(z) \in V_{n - p + 2} \quad (a_k \in \mathbb{C}, \, k = 0, 1, \ldots, p - 2).
\end{equation*}
Fix  $ (c_0 : c_1) \in \mathbb{P}^1 $. Then there exists $ b_k \in \mathbb{C} $ $ (k = 0, 1, \ldots, p - 2) $ such that
\begin{equation*}
  \Lambda_2 \coloneqq (b_0 \mathbf{x} + b_1 \mathbf{x}^{(1)} + \cdots + b_{p - 2} \mathbf{x}^{(p - 2)} + c_0 \mathbf{x}^{(p - 1)} + c_1 \mathbf{x}^{(p)})(z) \in V_{n - p + 2}.
\end{equation*}
Thus, we have
\begin{equation*}
  \mathrm{Span}_{\mathbb{C}}(\Lambda_1, \Lambda_2) \subseteq \mathrm{Span}_{\mathbb{C}}(\mathbf{X}^{(p - 1)}(z), (c_0 \mathbf{x}^{(p - 1)} + c_1 \mathbf{x}^{(p)})(z)) \subseteq V_{n - p + 2},
\end{equation*}
and hence $ \mathbf{X}^{(p - 1)}(z) \wedge (c_0\mathbf{x}^{(p - 1)} + c_1\mathbf{x}^{(p)})(z) \in T_z(\mathbf{X}^{(p)}(M)) \cap \sigma_{1, 1}(\Gamma_{n - p + 1})$. The converse implication can also be proved by reversing the above argument.
The proof of the latter statement is straightforward.
Therefore, we obtain the desired equality as follows:
\begin{align*}
   & \deg_{\mathrm{cov}}(\mathbf{X}^{(p)})\deg T(\mathbf{X}^{(p)})                                                                                                                                                               \\
   & = \deg_{\mathrm{cov}}(\mathbf{X}^{(p)})([T(\mathbf{X}^{(p)})(M)] \cdot \sigma_1^2)_{\mathrm{Gr}(p, n + 1)}                                                                                                                  \\
   & = \deg_{\mathrm{cov}}(\mathbf{X}^{(p)})([T(\mathbf{X}^{(p)})(M)] \cdot ([\sigma_{1, 1}(\Gamma_{n - p + 1})] + [\sigma_2(\Gamma_{n - p - 1})]))_{\mathrm{Gr}(p, n + 1)}                                                      \\
   & = \deg_{\mathrm{cov}}(\mathbf{X}^{(p - 1)})\#(\mathbf{X}^{(p - 1)}(M) \cap \Gamma_{n - p + 1})_{\mathbb{P}^n} + \deg_{\mathrm{cov}}(\mathbf{X}^{(p + 1)})\#(\mathbf{X}^{(p + 1)}(M) \cap \Gamma_{n - p - 1})_{\mathbb{P}^n} \\
   & = \nu_{p - 1} + \nu_{p + 1},
\end{align*}
where $ \Gamma_{n - p + 1} $ and $ \Gamma_{n - p - 1} $ are generic subspaces of $ \mathbb{P}^n $, of dimensions $ n - p + 1 $ and $ n - p - 1 $, respectively.

For $ i > 2 $, in place of $ T^{(2)}(\mathbf{X}^{(p)}(M)) \coloneqq T(\mathbf{X}^{(p)}(M)) $, it would be natural to consider the \emph{$ i $-th osculating planes} $ T^{(i)}(\mathbf{X}^{(p)}(M)) \subseteq \mathbb{P}^{\binom{n + 1}{p} - 1} $.
However, since $ T^{(i)}(\mathbf{X}^{(p)}(M)) \nsubseteq \mathrm{Gr}(p, n + 1) $, we cannot count points lying outside $ \mathrm{Gr}(p, n + 1) $, and so the above argument (Schubert calculus) does not seem to be directly applicable.
\begin{prob}
  Is a geometric interpretation similar to the one above possible for the generalized Weyl peculiar relation? How might this be related to the fact that it is given as an inequality rather than an equality?
\end{prob}

\noindent
\underline{\textbf{2.}} From \eqref{eq:fund_ineq_1}, \eqref{eq:fund_ineq_2} and \eqref{eq:T_p-omega_p}, for all $ \epsilon > 0$, we have
\begin{align}
  T_i\{\mathbf{X}^{(p)}\}                                & < (i + \epsilon)T_1\{\mathbf{X}^{(p)}\} = (i + \epsilon)T_p \, //, \label{eq:fund_ineq_3}                                                          \\
  \left(1 + \frac{1 - i}{\binom{n + 1}{p} - 1}\right)T_p & = \left(1 + \frac{1 - i}{\binom{n + 1}{p} - 1}\right)T_1\{\mathbf{X}^{(p)}\} < (1 + \epsilon)T_i\{\mathbf{X}^{(p)}\} \, //. \label{eq:fund_ineq_4}
\end{align}
For example, if $ i = 2 $, \eqref{eq:fund_ineq_3} becomes
\begin{equation*}
  T_{p - 1} + T_{p + 1} = T_2\{\mathbf{X}^{(p)}\} < (2 + \epsilon)T_p \, //.
\end{equation*}
This inequality is exactly Corollary~4.21 in \cite{Cowen-Griffiths}, p.~120.

Now, we compare \eqref{eq:fund_ineq_4} with \eqref{eq:gen_pec_ineq}. The left-hand side of \eqref{eq:fund_ineq_4} is clearly positive and less than $ T_p $ ($ i > 1 $).
In contrast, it is not even evident that the left-hand side of Theorem~\ref{thm:gen_pec_rel} is positive, so we need to analyze it more carefully. If we drop the term $ iT_p $ from the left-hand side of \eqref{eq:gen_pec_ineq}, our inequality becomes meaningless by \eqref{eq:concave_rel} in Remark~\ref{rem:T-seq}.
Therefore, in view of \eqref{eq:fund_ineq_3} and Theorem~\ref{thm:gen_pec_rel} (and Lemma~\ref{lem:iT_p-T_i{X^p}_rel}), the comparison between $ iT_p $ and $ T_i\{\mathbf{X}^{(p)}\} $ appears to offer the proper perspective from which to interpret this theorem. For example, the original Weyl peculiar relation can be viewed as taking the form
\begin{equation*}
  (T_{p - 1} - 2T_p + T_{p + 1}) + 2T_p = T_2\{\mathbf{X}^{(p)}\}.
\end{equation*}

In fact, we can show that the left-hand side of \eqref{eq:gen_pec_ineq}, which may seem difficult to estimate at first glance, is actually positive (Theorem~\ref{thm:pos_ineq}; see Section~\ref{sec:weighted_balanced_sum_formula}). To verify this, we use an alternative expression for the left-hand side of \eqref{eq:gen_pec_ineq}:
\begin{equation*}
  \sum_{s = 0}^{i - 1}\max_{\sigma \in \binom{[n + 1]}{p}_{(k_s)}}\left(\sum_{k = 1}^n n_{\lambda(\sigma)}(k)(T_{k - 1} - 2T_k + T_{k + 1}) + T_p\right).
\end{equation*}
Under this viewpoint, the main objects of our study are the sums consisting of several second-order differences $ T_{k - 1} -2T_k + T_{k + 1} $ together with a single term $ T_p $. While these quantities can take both positive and negative values individually, we will prove that their average is zero.

\section{The balanced sum formula}\label{sec:balanced_sum_formula}

In this section, we study the sequence $ \{T_k\}_{k = 0}^{n + 1} $ from a combinatorial perspective.

\subsection{The Balanced Sum Formula}
\ \par
The following theorem indicates that the total quantity maintains a balance between the ``negative'' terms $ T_{k - 1} -2T_k + T_{k + 1} $ and the positive term $ T_p $.
\begin{thm}[\textbf{Balanced Sum Formula}]\label{thm:balanced_sum_formula}
  \begin{equation}\label{eq:balanced_sum_formula}
    \sum_{s = 0}^{p(n - p + 1)}\sum_{\sigma \in \binom{[n + 1]}{p}_{(k_s)}}\left(\sum_{k = 1}^n n_{\lambda(\sigma)}(k)(T_{k - 1} - 2T_k + T_{k + 1}) + T_p\right) = 0.
  \end{equation}
\end{thm}
\begin{proof}
  First, we compute $ \sum_{k = 1}^n n_{\lambda(\sigma)}(k)(T_{k - 1} - 2T_k + T_{k + 1}) $ for each $ \sigma \in \binom{[n + 1]}{p} $. Let $ \lambda(\sigma) $ be the Young diagram corresponding to $ \sigma $. We call the number of boxes on the main diagonal of $ \lambda(\sigma) $ the \emph{thickness} of $ \lambda(\sigma) $. In other words, this is the number of balls placed in boxes labeled greater than $ p $ in the Maya diagram $ \sigma $.
  Let $ t = t(\lambda(\sigma)) $ be the thickness of $ \lambda(\sigma) $. Then the Young diagram $ \lambda(\sigma) $ decomposes into $ t $ hooks. Denote them by $ \{h_j(\sigma)\}_{j = 1}^t $. Then we obtain
  \begin{equation*}
    \begin{split}
      \sum_{k = 1}^n n_{\lambda(\sigma)}(k)(T_{k - 1} - 2T_k + T_{k + 1}) & = \sum_{j = 1}^t \sum_{k = 1}^n n_{h_j(\sigma)}(k)(T_{k - 1} - 2T_k + T_{k + 1}) \\
      & = \sum_{j = 1}^t(T_{b_{j} - 1} - T_{b_j} - T_{a_j} + T_{a_{j} + 1}),
    \end{split}
  \end{equation*}
  where $ a_j $ (resp.~$ b_j $) is defined as the maximum (resp.~minimum) number $ c $ satisfying $n_{h_j(\sigma)}(c) \neq 0$. From the perspective of Maya diagrams, this quantity equals
  \begin{equation*}
    \sum_{j \in \sigma_{\leq p - 1}^c} (T_j - T_{j + 1}) + \sum_{k \in \sigma_{\geq p}} (T_{k + 1} - T_k),
  \end{equation*}
  where $ \sigma_{\leq p - 1}^{c} $ and $ \sigma_{\geq p} $ are defined by
  \begin{equation*}
    \sigma_{\leq {p - 1}}^c \coloneqq \{j \in \{0, 1, \ldots, p - 1\} \mid j \notin \sigma\}, \, \sigma_{\geq p} \coloneqq \{k \in \{p, p + 1, \ldots, n\} \mid k \in \sigma\}.
  \end{equation*}
  Here, $ k \in \sigma $ means that the Maya diagram $ \sigma $ has a ball in the box numbered $ k + 1 $. Next, we compute
  \begin{equation}\label{eq:T_sum}
    \sum_{s = 0}^{p(n - p + 1)}\sum_{\sigma \in \binom{[n + 1]}{p}_{(k_s)}}\sum_{j \in \sigma_{\leq p - 1}^c}T_j.
  \end{equation}
  We then place the Young diagram $ \lambda(\sigma) $ above the Maya diagram $ \sigma $ on $ \mathbb{R}^2 $, using the Russian convention, as illustrated in Figure~\ref{fig:Russian_conv} and Figure~\ref{fig:Young_diag_on_R^2}. We express \eqref{eq:T_sum} in the form $ \sum_{j = 0}^{p - 1} c_jT_j$.
  Let $ c_j(\sigma) $ (resp. $ c_j'(\sigma) $) denote the number of edges of squares in the Young diagram $ \lambda(\sigma) $ with a negative (resp. positive) slope that are located above the empty box numbered $ j + 1 $ in the Maya diagram $ \sigma $, where $j \in \sigma_{\leq p - 1}^c $ (resp. $ j \in \sigma_{\geq p} $). Then $ c_j $ is the total sum of $ c_j(\sigma) $ over all $ \sigma \in \binom{[n + 1]}{p}$ such that $ j \in \sigma_{\leq p - 1}^c $.
  Recall that a Young diagram fitting in the $ p \times (n - p + 1) $ rectangle can be identified with a minimal stepwise path from the bottom-left corner to the top-right corner of the $ p \times (n - p + 1) $ grid. Hence, to determine $ c_j $, it suffices to count the number of minimal paths that pass through one of the edges located above the empty box numbered $ j + 1 \, (j \in \sigma_{\leq p - 1}^c)$ in the diagram on $ \mathbb{R}^2 $.
  For simplicity, we consider only the case $ p \leq n - p + 1 $. For each $ 0 \leq j \leq p - 1 $, it can be readily verified that the number of such minimal paths is given by
  \begin{equation*}
    \sum_{s = 1}^j \binom{j - 1}{s - 1}\binom{n - j + 1}{s + p - j} = \sum_{s = 0}^{j - 1}\binom{j - 1}{s}\binom{n - j + 1}{n - p - s} = \binom{n}{n - p} = \binom{n}{p},
  \end{equation*}
  where we apply Vandermonde's convolution. As a result, we obtain $ c_j = \binom{n}{p} $ for all $ 0 \leq j \leq p - 1 $. Hence, \eqref{eq:T_sum} equals $ \binom{n}{p}\sum_{j = 1}^{p - 1} T_j $. In a similar way, if we define $ c_j' $ as the total sum of $ c_j'(\sigma) $, where $ \sigma $ runs over all $ \sigma \in \binom{[n + 1]}{p}$ such that $ j \in \sigma_{\geq p} $, then we can show that $ c_j' = \binom{n}{n + 1 - p} = \binom{n}{p - 1}$ for all $ p \leq j \leq n $.  Thus we have
  \begin{align}\label{eq:binom_coeff}
    \sum_{s = 0}^{p(n - p + 1)}\sum_{\sigma \in \binom{[n + 1]}{p}_{(k_s)}} & \left( \sum_{j \in \sigma_{\leq p - 1}^c}(T_j - T_{j + 1}) + \sum_{k \in \sigma_{\geq p}}(T_{k + 1} - T_k)\right) \notag \\
                                                                            & = -\binom{n}{p}T_p -\binom{n}{p - 1}T_p = -\binom{n + 1}{p}T_p.
  \end{align}
  Since
  \begin{equation*}
    \sum_{s = 0}^{p(n - p + 1)}\sum_{\sigma \in \binom{[n + 1]}{p}_{(k_s)}}T_p = \binom{n + 1}{p}T_p,
  \end{equation*}
  we obtain the desired equality.
\end{proof}
\begin{rem}
  In fact, we can show that $ c_j = \binom{n}{p} $ and $ c_j' = \binom{n}{p - 1} $ for all $ 0 \leq j \leq n $. This result also follows from the proof of \textup{Proposition~\ref{prop:Young_diag_count_formula}}.
\end{rem}
\subsection{Some Remarks on the Balanced Sum Formula}
\ \par
In this subsection, we further investigate the structure of Theorem~\ref{thm:balanced_sum_formula}.

\noindent
\underline{\textbf{1.}} From the proof of the balanced sum formula (Theorem~\ref{thm:balanced_sum_formula}), we derive the following formula.
\begin{cor}\label{cor:sec_diff_sum_id}
  \begin{align*}
    \sum_{s = 1}^{p(n - p + 1)}\sum_{\sigma \in \binom{[n + 1]}{p}_{(k_s)}}\left(\sum_{k = 1}^n n_{\lambda(\sigma)}(k) (\overline{T}_{k - 1} - 2\overline{T}_k + \overline{T}_{k + 1})\right) &                              \\
    = -\binom{n + 1}{p}\overline{T}_p                                                                                                                                                         & + \binom{n}{p - 1}N_{n + 1},
  \end{align*}
  where $ \overline{T}_i $ and $ N_{n + 1} $ are defined in \textup{Definition~\ref{def:N-V-Tbar}}.
\end{cor}
\begin{proof}
  Since $ \overline{T}_{0} = 0 $ and $ \overline{T}_{n + 1} = N_{n + 1} $, this relation immediately follows from the proof of \eqref{eq:balanced_sum_formula}.
\end{proof}

\noindent
\underline{\textbf{2.}}
Let $ \{a_k\}_{k = 0}^{n + 1} $ be a sequence with $ a_0 = a_{n + 1} = 0 $ and $ a_k \in \mathbb{R} $ $ (1 \leq k \leq n)$. Theorem~\ref{thm:balanced_sum_formula} is a reflection of the following identity:
\begin{align}\label{eq:ak_id}
  \sum_{k = 1}^{p}\binom{n}{p}k(a_{k - 1} - 2a_k + a_{k + 1}) + \sum_{k = p + 1}^n \binom{n}{p - 1}(n + 1 - k) (a_{k - 1} & - 2a_k + a_{k + 1})    \notag \\
                                                                                                                          & = -\binom{n + 1}{p}a_p.
\end{align}
Although this identity \eqref{eq:ak_id} can be readily verified by a direct computation, we instead derive it by explicitly computing the quantities that appear on the left-hand side of the equality \eqref{eq:balanced_sum_formula} in Theorem~\ref{thm:balanced_sum_formula}.
\begin{proof}[Proof of \eqref{eq:ak_id}]
  \ \par
  Fix $ k \in \{1, 2, \ldots, n\} $, and let $ \alpha_k \coloneqq \sum_{s = 0}^{p(n - p + 1)}\sum_{\sigma \in \binom{[n + 1]}{p}_{(k_s)}}n_{\lambda(\sigma)}(k) $. Then
  \begin{equation*}
    \sum_{s = 0}^{p(n - p + 1)}\sum_{\sigma \in \binom{[n + 1]}{p}_{(k_s)}}\left(\sum_{k = 1}^n n_{\lambda(\sigma)}(k)(a_{k - 1} - 2a_{k} + a_{k + 1})\right) = \sum_{k = 1}^n \alpha_k (a_{k - 1} - 2a_k + a_{k + 1}).
  \end{equation*}
  We claim that $ \alpha_k = \binom{n}{p}k $ for $ 1 \leq k \leq p $, and $ \alpha_k = \binom{n}{p - 1}(n + 1 - k) $ for $ p \leq k \leq n $. Here, we prove the former by induction on $ k $; the latter follows from a similar argument.
  As shown in the proof of Theorem~\ref{thm:balanced_sum_formula}, we have $ \alpha_1 = \binom{n}{p} $. Assume that the claim holds for some $ k $ $(\leq p - 1)$. Then, from the proof of Theorem~\ref{thm:balanced_sum_formula}, it follows that
  \begin{align*}
    \alpha_{k + 1}
     & = \sum_{s = 0}^{p(n - p + 1)}\left(\sum_{\sigma \in \binom{[n + 1]}{p}_{(k_s)}}n_{\lambda(\sigma)}(k) +\sum_{\sigma \in \binom{[n + 1]}{p}_{(k_s)}, \, k \in \sigma_{\leq p - 1}^c}c_k(\sigma) \right) \\
     & = \alpha_k + c_k = \binom{n}{p}k + \binom{n}{p} = \binom{n}{p}(k + 1),
  \end{align*}
  where $ c_k(\sigma) $ and $ c_k $ are defined in the proof of Theorem~\ref{thm:balanced_sum_formula}.
\end{proof}
\noindent
\underline{\textbf{3.}}
Theorem~\ref{thm:balanced_sum_formula} holds for any sequence $ \{a_k\}_{k = 0}^{n + 1} $ with $ a_0 = a_{n + 1} = 0 $ and $ a_k \in \mathbb{R} $ $ (1 \leq k \leq n) $, not necessarily restricted to $ \{T_k\}_{k = 0}^{n + 1} $.
Let $ M $ be a compact Riemann surface of genus $ g_M $, and let $ \mathbf{x} : M \to \mathbb{P}^n $ be a non-degenerate algebraic curve. The classical \textbf{Brill--Segre formula} (see, for instance, \cite{Griffiths-Harris}, pp.~270--271; \cite{Ru}, p.~18) states that
\begin{equation*}
  \sum_{k = 1}^{n}(n - k + 1)\sigma_{k} = n(n + 1)(g_M - 1) + (n + 1)\deg(\mathbf{x}).
\end{equation*}
The application of Theorem~\ref{thm:balanced_sum_formula} to the sequence $ \{\nu_k = \deg(\mathbf{X}^{(k)})\}_{k = 0}^{n + 1} $ yields a generalization of the Brill--Segre formula.
Indeed, when $ p = 1 $, \eqref{eq:balanced_sum_formula} for the sequence $ \{\nu_k\}_{k = 0}^{n + 1} $ becomes
\begin{equation}\label{eq:balanced_sum_formula_for_nu_k}
  \sum_{s = 0}^n \sum_{\sigma \in \binom{[n + 1]}{1}_{(k_s)}}\left(\sum_{k = 1}^n n_{\lambda(\sigma)}(k)(\nu_{k - 1} - 2\nu_k + \nu_{k + 1}) + \nu_1\right) = 0,
\end{equation}
where $ \binom{[n + 1]}{1}_{(k_s)} $ is the set $\{[k] \mid 0 \leq k \leq s\}$. Hence, we have $ n_{[k]}(l) = 1 $ $ (1 \leq l \leq k)$ and $ n_{[k]}(l) = 0 $ $ (k + 1 \leq l \leq n)$. Thus, the left-hand side of \eqref{eq:balanced_sum_formula_for_nu_k} equals
\begin{equation*}
  \begin{split}
    & \sum_{s = 1}^n \sum_{k = 1}^s (\nu_{k - 1} - 2\nu_k + \nu_{k + 1}) + (n + 1)\nu_1
    = \sum_{s = 1}^n \sum_{k = 1}^s(2g_M - 2 - \sigma_k)  + (n + 1)\deg(\mathbf{x})     \\
    & = n(n + 1)(g_M - 1) - \sum_{k = 1}^n (n - k + 1)\sigma_k + (n + 1)\deg(\mathbf{x}),
  \end{split}
\end{equation*}
where we apply the Pl\"ucker formula for algebraic curves (see \eqref{eq:Plucker_formula_alg_curve}).

We can also generalize the argument presented in Proposition on pp.~270--271 of \cite{Griffiths-Harris}.
\begin{cor}\label{cor:deg_X^p}
  If a non-degenerate algebraic curve $ \mathbf{x} : M \to \mathbb{P}^{n} $ is totally unramified \textup{(\textit{i.e. $ \sigma_i  = 0$ for all $ i $})}, then $ \mathbf{x} $ is the rational normal curve, and we have
  \begin{equation*}
    \deg(\mathbf{X}^{(p)}) = p(n - p + 1) = \dim \mathrm{Gr}(p, n + 1).
  \end{equation*}
\end{cor}
\begin{proof}
  Since $ \mathbf{x} $ is totally unramified, we have
  \begin{equation*}
    (2g_M - 2)\sum_{s = 1}^{q}s \cdot \#\binom{[n + 1]}{p}_{(k_s)} + \binom{n + 1}{p}\nu_p = 0.
  \end{equation*}
  We thus have $ g_M = 0 $. Moreover, by Lemma~\ref{lem:Gauss_comp}, we obtain
  \begin{equation*}
    \nu_p = q = p(n - p + 1).
  \end{equation*}
\end{proof}
\begin{rem}
  \textup{Corollary~\ref{cor:deg_X^p}} can be found in \textup{R. Piene} \cite{Piene}, \textup{Section~3}, \textup{Example~1}, \textup{pp.~482--483}.
  In a similar manner, by applying the argument in \textup{\underline{\textbf{2.}}} and the Brill--Segre formula, one can compute $ \deg(\mathbf{X}^{(p)}) $ for a general curve $ \mathbf{x} : M \to \mathbb{P}^n $:
  \begin{equation*}
    \begin{split}
      \deg(\mathbf{X}^{(p)})
      & = \frac{n + 1 - p}{n + 1}\sum_{k = 1}^p k\sigma_k + \frac{p}{n + 1}\sum_{k = p + 1}^n (n + 1 - k)\sigma_k - p(n - p + 1)(g_M - 1) \\
      & = p(\deg(\mathbf{x}) + (p - 1)(g_M - 1)) - \sum_{k = 1}^{p - 1}(p - k)\sigma_k.
    \end{split}
  \end{equation*}
  This formula is stated in \cite{Piene}, \textup{Theorem~3.2}, \textup{p.~481}.
\end{rem}

\section{Combinatorial Analysis of Standard Young Tableaux}\label{sec:weighted_balanced_sum_formula}
\subsection{A Combinatorial Proposition on Standard Young Tableaux}
\ \par
To prove Theorem~\ref{thm:pos_ineq}, we study combinatorial properties of the standard Young tableaux in more detail.
\begin{defn}
  Let $ \tau $ be a Young diagram, and let $ \mathbb{Y}_{\tau} $ be the set of Young diagrams contained in $ \tau $.
  The poset $ \mathbb{Y}_{\tau} = (\mathbb{Y}_{\tau}, \subseteq) $ is called the \textbf{finite Young lattice} of $ \tau $.
  The two operations $ \lor $ \textup{(\textit{join})} and $ \land $ \textup{(\textit{meet})} are defined respectively as the union and intersection of two Young diagrams.
  For each $ \lambda \subseteq \tau $, we denote $ |\lambda| $ by $ \rk(\lambda) $. For two Young diagrams $ \lambda $, $ \lambda' $ with $ \lambda \subseteq \lambda' \subseteq \tau $, a map
  \begin{equation*}
    \lambda \to \lambda' : \{\rk(\lambda), \rk(\lambda) + 1, \ldots, \rk(\lambda')\} \to \mathbb{Y}_{\tau}, \quad s \mapsto (\lambda \to \lambda')(s)
  \end{equation*}
  is called a \textbf{saturated chain} \textup{(\textit{hereafter simply \textbf{chain}})} from $ \lambda $ to $ \lambda' $ if it satisfies $ (\lambda \to \lambda')(\rk(\lambda)) = \lambda $, $ (\lambda \to \lambda')(\rk(\lambda')) = \lambda' $, $ (\lambda \to \lambda')(s) \subseteq (\lambda \to \lambda')(s + 1) $ and $ |(\lambda \to \lambda')(s)| = s $ for all
  $ s \in \{\rk(\lambda), \rk(\lambda) + 1, \ldots, \rk(\lambda') - 1\} $.
\end{defn}
\begin{ex} Let $ p = 2 $ and $ n = 4 $. \textup{Figure~\ref{fig:Young_lattice}} is the finite Young lattice of $ \tau = [(n - p + 1)^p] = [3^2] = (3, 3) $.
  \begin{figure}[htbp]
    \begin{equation*}
      \Yvcentermath1
      \YRussian
      \xymatrix@=10pt{
      & & & {\gyoungs(0.65,;;,;;,;;)}\ar@{-}[dl]\\
      & & {\gyoungs(0.65,:;,;;,;;)}\ar@{-}[dl]\ar@{-}[d] & \\
      & {\gyoungs(0.65,;;,;;)}\ar@{-}[d] & {\gyoungs(0.65,:;,:;,;;)}\ar@{-}[dl]\ar@{-}[dr] & \\
      & {\gyoungs(0.65,:;,;;)}\ar@{-}[dl]\ar@{-}[dr] & & {\gyoungs(0.65,;,;,;)}\ar@{-}[dl] \\
      {\gyoungs(0.65,;;)}\ar@{-}[dr] & & {\gyoungs(0.65,;,;)}\ar@{-}[dl] & \\
      & {\yngs(0.65,1)}\ar@{-}[d] & & \\
      & \varnothing & &
      }
    \end{equation*}
    \caption{\texorpdfstring{The Young lattice $ \mathbb{Y}_{[3^2]} $}{The Young lattice Y_{[3^2]}}}
    \label{fig:Young_lattice}
  \end{figure}
\end{ex}
\begin{prop}\label{prop:Young_diag_count_formula}
  Let $ p $ be an integer with $ 1 \leq p \leq n $.
  For each integer $ 0 \leq j \leq n $, we have
  \begin{align*}
    \sum_{j \in \sigma_{\leq n}^c}f_{\lambda(\sigma)}f_{\widehat{\lambda(\sigma)}} & = f_{[(n - p + 1)^p, 1]} = \frac{(n - p + 1)(p(n - p + 1) + 1)!}{n + 1}\prod_{k = 0}^{n - p}\frac{k!}{(p + k)!}             \\
                                                                                   & = \frac{(n - p + 1)(p(n - p + 1) + 1)}{n + 1} f_{[(n - p + 1)^p]},                                                          \\
    \sum_{j \in \sigma_{\geq 0}}f_{\lambda(\sigma)}f_{\widehat{\lambda(\sigma)}}   & = f_{[n - p + 2, (n - p + 1)^{p - 1}]} = \frac{p(p(n - p + 1) + 1)!}{n + 1}\prod_{k = 0}^{p - 1}\frac{k!}{(n - p + 1 + k)!} \\
                                                                                   & = \frac{p(p(n - p + 1) + 1)}{n + 1}f_{[(n - p + 1)^p]}.
  \end{align*}
  For the meaning of the summations, see the proof of \textup{Theorem~\ref{thm:balanced_sum_formula}}.
\end{prop}
\begin{proof}
  Here, we prove the former in the case where $ p < n - p + 1 $. The latter and the other cases can be proved in a similar way.

  First, we prove that $ \sum_{j \in \sigma_{\leq n}^c}f_{\lambda(\sigma)}f_{\widehat{\lambda(\sigma)}} $ is independent of the choice of an integer $ 0 \leq j \leq n $. A box in a Young diagram $ \lambda $ is called a \emph{bottom box} (resp. \emph{rightmost box}) if there is no box directly below (resp. to the right of) it, and an \emph{outside corner} if there is no box in $ \lambda $ directly below it or to the right of it.

  The sets $ \mathrm{B}(i, j), \mathrm{R}(i, j) $, $ \mathrm{BR}((i, j), (i', j')) $, $ \mathrm{BB}((i, j), (i', j')) $, and $ \mathrm{C}(i, j) $ are defined respectively by
  \begin{align*}
    \mathrm{B}(i, j)              & \coloneqq \{\lambda \subseteq   [(n - p + 1)^p] \mid \lambda \, \text{has the} \, (i, j)\text{-box as a bottom box.}\},          \\
    \mathrm{R}(i, j)              & \coloneqq \{\lambda \subseteq   [(n - p + 1)^p] \mid \lambda \, \text{has the} \, (i, j)\text{-box as a rightmost box.}\},       \\
    \mathrm{BR}((i, j), (i', j')) & \coloneqq \{\lambda \subseteq [(n - p + 1)^p] \mid  \lambda \, \text{has the} \, (i, j)\text{-box as a bottom box}               \\
                                  & \hspace{11em} \text{and the} \, (i', j')\text{-box as a rightmost box.}\},                                                       \\
    \mathrm{BB}((i, j), (i', j')) & \coloneqq \{\lambda \subseteq [(n - p + 1)^p] \mid \lambda \, \text{has the} \, (i, j)\text{-box and the} \, (i', j')\text{-box} \\
                                  & \hspace{19em} \text{as bottom boxes}.\},                                                                                         \\
    \mathrm{C}(i, j)              & \coloneqq \{\lambda \subseteq [(n - p + 1)^p] \mid \lambda \, \text{has the} \, (i, j)\text{-box as an outside corner.}\},
  \end{align*}
  Formally, we set $ \mathrm{B}(0, 1) \coloneqq \{\varnothing\} $, $ \mathrm{B}(0, j) \coloneqq \mathrm{B}(0, j - 1) \cup \mathrm{R}(1, j - 1) $.
  In addition, we define
  \begin{align*}
    \mathrm{b}(i, j)              & \coloneqq \sum_{\lambda \in \mathrm{B}(i, j)}\#\{c = (\varnothing \to [(n - p + 1)^p]) \mid c \, \text{passes through} \, \lambda\},                \\
    \mathrm{br}((i, j), (i', j')) & \coloneqq \sum_{\lambda \in \mathrm{BR}((i, j), (i', j'))} \#\{c = (\varnothing \to [(n - p + 1)^p]) \mid c \, \text{passes through} \, \lambda \}, \\
    \mathrm{bb}((i, j), (i', j')) & \coloneqq \sum_{\lambda \in \mathrm{BB}((i, j), (i', j'))} \#\{c = (\varnothing \to [(n - p + 1)^p]) \mid c \, \text{passes through} \, \lambda \}, \\
    \mathrm{c}(i, j)              & \coloneqq \sum_{\lambda \in \mathrm{C}(i, j)} \#\{c = (\varnothing \to [(n - p + 1)^p]) \mid c \, \text{passes through} \, \lambda \}.
  \end{align*}
  These numbers satisfy the following relations:
  \begin{enumerate}
    \item $\mathrm{b}(i, j) = \mathrm{bb}((i, j), (i, j + 1)) + \mathrm{c}(i, j),$
    \item $ \mathrm{b}(i, j) = \mathrm{bb}((i, j - 1), (i, j)) + \mathrm{br}((i, j), (i + 1, j - 1)),$
    \item $\mathrm{c}(i, j) = \mathrm{br}((i - 1, j), (i, j - 1)).$
  \end{enumerate}
  For simplicity, we consider only the case $ p + 1 \leq j \leq n - p + 1 $.
  Since $ \mathrm{bb}((p, j), (p, j + 1)) = \mathrm{b}(p, j + 1)$, these equalities imply
  \begin{align*}
     & \sum_{j - 1 \in \sigma_{\leq n}^c}f_{\lambda(\sigma)}f_{\widehat{\lambda(\sigma)}}
    = \sum_{k = 0}^{p}\mathrm{b}(p - k,  j - k)                                                                                         \\
     & =  \sum_{k = 0}^{p - 1}(\mathrm{bb}((p - k, j - k), (p - k, j - k + 1)) + \mathrm{c}(p - k, j - k)) + \mathrm{b}(0, j - p)       \\
     & = \sum_{k = 0}^{p - 1}(\mathrm{bb}((p - k, j - k), (p - k, j - k + 1))                                                           \\
     & \qquad \quad + \mathrm{br}((p - k - 1, j - k), (p - k, j - k - 1))) + \mathrm{b}(0, j - p)                                       \\
     & = \sum_{k = 0}^{p}\mathrm{b}(p - k, j - k + 1) = \sum_{j \in \sigma_{\leq n}^c}f_{\lambda(\sigma)}f_{\widehat{\lambda(\sigma)}},
  \end{align*}
  A similar equality holds for all $ 0 \leq j \leq n $. Thus, by applying this relation inductively, we obtain
  \begin{equation*}
    \sum_{j \leq \sigma_{\leq n}^c}f_{\lambda(\sigma)}f_{\widehat{\lambda(\sigma)}} = \sum_{0 \leq \sigma_{\leq n}^c}f_{\lambda(\sigma)}f_{\widehat{\lambda(\sigma)}} \quad (0 \leq j \leq n).
  \end{equation*}
  Hence, it is enough to show that $ f_{[(n - p + 1)^p, 1]} = \sum_{0 \leq \sigma_{\leq n}^c}f_{\lambda(\sigma)}f_{\widehat{\lambda(\sigma)}} $. Since all chains $ c = (\varnothing \to [(n - p + 1)^p, 1]) $ can be classified by the step at which they contain the $ (p + 1, 1) $-box, we have
  \begin{equation*}
    \begin{split}
      & f_{[(n - p + 1)^p, 1]} = \#\{\varnothing \to [(n - p + 1)^p, 1]\}                                                                                                                    \\
      & = \sum_{k = 1}^{p(n - p + 1) + 1}\#\{c = (\varnothing \to [(n - p + 1)^p, 1]) \mid [1^{p + 1}] \subseteq c(k), [1^{p + 1}] \nsubseteq c(k - 1)\}                                           \\
      & = \sum_{[1^p] \subseteq \lambda \subseteq [(n - p + 1)^p]} f_{\lambda} f_{\widehat{\lambda}}
      = \sum_{0 \in \sigma_{\leq n}^c} f_{\lambda(\sigma)}f_{\widehat{\lambda(\sigma)}},
    \end{split}
  \end{equation*}
  where the second sum runs over all Young diagrams that  contain $ [1^p] $ and are contained in $ [(n - p + 1)^p] $. Figure~\ref{fig:extd_Young_lattice} illustrates the argument above.
  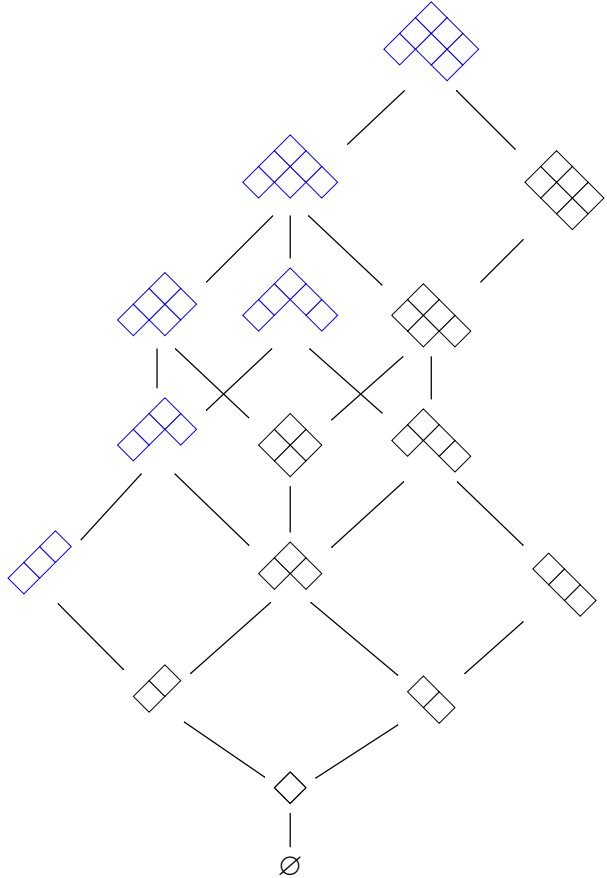
\begin{figure}[htbp]
    \begin{equation*}
      \Yvcentermath1
      \YRussian
      \xymatrix@=10pt{
      & & & \Ylinecolour{blue}{\gyoungs(0.65,:;;,:;;,;;;)}\ar@{-}[dl]\ar@{-}[dr] & \\
      & & \Ylinecolour{blue}{\gyoungs(0.65,::;,:;;,;;;)}\ar@{-}[dl]\ar@{-}[d]\ar@{-}[dr] & & {\gyoungs(0.65,;;,;;,;;)}\ar@{-}[dl]\\
      & \Ylinecolour{blue}{\gyoungs(0.65,:;;,;;;)}\ar@{-}[d]\ar@{-}[dr] & \Ylinecolour{blue}{\gyoungs(0.65,::;,::;,;;;)}\ar@{-}[dl]\ar@{-}[dr] & {\gyoungs(0.65,:;,;;,;;)}\ar@{-}[dl]\ar@{-}[d] & \\
      & \Ylinecolour{blue}{\gyoungs(0.65,::;,;;;)}\ar@{-}[dl]\ar@{-}[dr] & {\gyoungs(0.65,;;,;;)}\ar@{-}[d] & {\gyoungs(0.65,:;,:;,;;)}\ar@{-}[dl]\ar@{-}[dr] & \\
      \Ylinecolour{blue}{\gyoungs(0.65,;;;)}\ar@{-}[dr] & & {\gyoungs(0.65,:;,;;)}\ar@{-}[dl]\ar@{-}[dr] & & {\gyoungs(0.65,;,;,;)}\ar@{-}[dl] \\
      & {\gyoungs(0.65,;;)}\ar@{-}[dr] & & {\gyoungs(0.65,;,;)}\ar@{-}[dl] & \\
      & & {\yngs(0.65,1)}\ar@{-}[d] & & \\
      & & \varnothing & &
      }
    \end{equation*}
    \caption{\texorpdfstring{The Young lattice $ \mathbb{Y}_{[3^2, 1]} $}{The Young lattice Y_{[3^2, 1]}}}
    \label{fig:extd_Young_lattice}
  \end{figure}

  Moreover, the hook length formula \eqref{eq:hook_length_formula} implies
  \begin{equation*}
    \begin{split}
      f_{[(n - p + 1)^p, 1]} & = \frac{(p(n - p + 1) + 1)!}{\frac{(n + 1)!}{(n - p + 1)!}\left(\frac{(n - 1)!}{(n - p - 1)!} \cdots \frac{(p + 1)!}{1!}\frac{p!}{0!}\right)}
      = \frac{n - p + 1}{n + 1}\frac{(p(n - p + 1) + 1)!}{\frac{n!}{(n - p)!} \frac{(n - 1)!}{(n - p - 1)!} \cdots \frac{(p + 1)!}{1!} \frac{p!}{0!}}                        \\
      & = \frac{(n - p + 1)(p(n - p + 1) + 1)}{n + 1}f_{[(n - p + 1)^p]}.
    \end{split}
  \end{equation*}
\end{proof}
\begin{ex}
  Let $ n = 4 $ and $ p = 2 $. Then we obtain:
  \begin{align*}
     & \sum_{0 \in \sigma_{\leq 4}^c}f_{\lambda(\sigma)}f_{\widehat{\lambda(\sigma)}} = f_{\yngs(0.4,1^2)}f_{\yngs(0.4,2^2)} + f_{\yngs(0.4,2,1)}f_{\yngs(0.4,2,1)} + f_{\yngs(0.4,2^2)}f_{\yngs(0.4,1^2)} + f_{\yngs(0.4,3,1)}f_{\yngs(0.4,2)} + f_{\yngs(0.4,3,2)}f_{\yngs(0.4,1)} + f_{\yngs(0.4,3^2)}f_{\varnothing} = 21, \\
     & \sum_{1 \in \sigma_{\leq 4}^c}f_{\lambda(\sigma)}f_{\widehat{\lambda(\sigma)}} = f_{\yngs(0.4,1)}f_{\yngs(0.4,3,2)} + f_{\yngs(0.4,2)}f_{\yngs(0.4,3,1)} + f_{\yngs(0.4,3)}f_{\yngs(0.4,3)} + f_{\yngs(0.4,2^2)}f_{\yngs(0.4,1^2)} + f_{\yngs(0.4,3,2)}f_{\yngs(0.4,1)} + f_{\yngs(0.4,3^2)}f_{\varnothing} = 21,       \\
     & \sum_{2 \in \sigma_{\leq 4}^c}f_{\lambda(\sigma)}f_{\widehat{\lambda(\sigma)}} = f_{\varnothing}f_{\yngs(0.4,3^2)} + f_{\yngs(0.4,2)}f_{\yngs(0.4,3,1)} + f_{\yngs(0.4,2,1)}f_{\yngs(0.4,2,1)} + f_{\yngs(0.4,3)}f_{\yngs(0.4,3)} + f_{\yngs(0.4,3,1)}f_{\yngs(0.4,2)} + f_{\yngs(0.4,3^2)}f_{\varnothing} = 21,        \\
     & \sum_{2 \in \sigma_{\geq 0}}f_{\lambda(\sigma)}f_{\widehat{\lambda(\sigma)}} = f_{\yngs(0.4,1)}f_{\yngs(0.4,3,2)} + f_{\yngs(0.4,1^2)}f_{\yngs(0.4,2^2)} + f_{\yngs(0.4,2^2)}f_{\yngs(0.4,1,1)} + f_{\yngs(0.4,3,2)}f_{\yngs(0.4,1)} = 14,                                                                              \\
     & \sum_{3 \in \sigma_{\geq 0}}f_{\lambda(\sigma)}f_{\widehat{\lambda(\sigma)}} = f_{\yngs(0.4,2)}f_{\yngs(0.4,3,1)} + f_{\yngs(0.4,2,1)}f_{\yngs(0.4,2,1)} + f_{\yngs(0.4,2^2)}f_{\yngs(0.4,1,1)} + f_{\yngs(0.4,3^2)}f_{\varnothing} = 14,                                                                               \\
     & \sum_{4 \in \sigma_{\geq 0}}f_{\lambda(\sigma)}f_{\widehat{\lambda(\sigma)}} = f_{\yngs(0.4,3)}f_{\yngs(0.4,3)} + f_{\yngs(0.4,3,1)}f_{\yngs(0.4,2)} + f_{\yngs(0.4,3,2)}f_{\yngs(0.4,1)} + f_{\yngs(0.4,3^2)}f_{\varnothing} = 14.
  \end{align*}
  Furthermore, we obtain
  \begin{equation*}
    f_{\yngs(0.4,3^2,1)} = 21, \quad f_{\yngs(0.4,4,3)} = 14.
  \end{equation*}
\end{ex}
\begin{ex}
  It is well known that $ f_{[n, n]} $ equals the \textbf{Catalan number}
  \begin{equation*}
    C_n \coloneqq \frac{1}{n + 1}\binom{2n}{n} = \frac{(2n)!}{n!(n + 1)!}.
  \end{equation*}
  \textup{(\textit{see} OEIS A000108.)}
  Thus, if $ p = 2 $, the numbers computed using the second formula in \textup{Proposition \ref{prop:Young_diag_count_formula}} coincide with $ C_n $:
  \begin{equation*}
    \sum_{j \in \sigma_{\geq 0}}f_{\lambda(\sigma)}f_{\widehat{\lambda(\sigma)}} = f_{[n, n - 1]} = f_{[n, n]} = C_n.
  \end{equation*}
\end{ex}
\subsection{Proofs of the Weighted Balanced Sum Formula and Theorem~B}
\ \par
The main purpose of this section is to prove the following identity.
\begin{thm}[\textbf{Weighted Balanced Sum Formula}]\label{thm:weighted_balanced_sum_formula}
  \begin{align*}
     & \sum_{c \in \{\varnothing \to [(n - p + 1)^p]\}}\sum_{s = 0}^{p(n - p + 1)}\left(\sum_{k = 1}^n n_{c(s)}(k)(T_{k - 1} - 2T_k + T_{k + 1}) + T_p\right)                                                               \\
     & = \sum_{s = 0}^{p(n - p + 1)}\sum_{\sigma \in \binom{[n + 1]}{p}_{(k_s)}}f_{\lambda(\sigma)}f_{\widehat{\lambda(\sigma)}}\left(\sum_{k = 1}^n n_{\lambda(\sigma)}(k)(T_{k - 1} - 2T_k + T_{k + 1}) + T_p\right) = 0.
  \end{align*}
\end{thm}
\begin{proof}
  For each Young diagram $ \lambda \subseteq [(n - p + 1)^p] $, the product $ f_{\lambda}f_{\widehat{\lambda}} $ represents the number of chains $ c $ from $ \varnothing $ to $ [(n - p + 1)^p] $ that  pass through $ \lambda $ in the Young lattice $ \mathbb{Y}_{[(n - p + 1)^p]} $. This establishes the first equality.
  By Proposition~\ref{prop:Young_diag_count_formula}, as in equation~\eqref{eq:binom_coeff}, we obtain
  \begin{equation*}
    \begin{split}
      & \sum_{s = 0}^{p(n - p + 1)} \sum_{\sigma \in \binom{[n + 1]}{p}_{(k_s)}}f_{\lambda(\sigma)}f_{\widehat{\lambda(\sigma)}} \left(\sum_{j \in \sigma_{\leq p - 1}^c}(T_j - T_{j + 1}) + \sum_{k \in \sigma_{\geq p}}(T_{k + 1} - T_k)\right) \\
      & = -\frac{(n - p + 1)(p(n - p + 1) + 1)}{n + 1}f_{[(n - p + 1)^p]}T_p - \frac{p(p(n - p + 1) + 1)}{n + 1}f_{[(n - p + 1)^p]}T_p                                                                                                            \\
      & = -(p(n - p + 1) + 1)f_{[(n - p + 1)^p]}T_p.
    \end{split}
  \end{equation*}
  Since
  \begin{equation*}
    \sum_{s = 0}^{p(n - p + 1)}\sum_{\sigma \in \binom{[n + 1]}{p}_{(k_s)}}f_{\lambda(\sigma)}f_{\widehat{\lambda(\sigma)}}T_p = \sum_{s = 0}^{p(n - p + 1)}f_{[(n - p + 1)^p]}T_p = (p(n - p + 1) + 1)f_{[(n - p + 1)^p]}T_p,
  \end{equation*}
  we obtain the second equality of Theorem~\ref{thm:weighted_balanced_sum_formula}.
\end{proof}
We now give the proof of Theorem~\ref{thm:pos_ineq}.
\begin{thm}[Theorem B]\label{thm:pos_ineq}
  Let $ 0 < \epsilon < 1 $. For each integer $ 1 \leq i \leq p(n - p + 1) $, the following inequality holds:
  \begin{equation*}
    (1 - \epsilon)\min(T_p, T_{n - p + 1}) < \sum_{s = 1}^{i - 1}\max_{\sigma \in \binom{[n + 1]}{p}_{(k_s)}}\left(\sum_{k = 1}^n n_{\lambda(\sigma)}(k)(T_{k - 1} - 2T_k + T_{k + 1})\right) + iT_p \, //.
  \end{equation*}
\end{thm}
\begin{proof}
  Let $ \sigma $ be the unique element of the set $ \binom{[n + 1]}{p}_{(k_q)} $, where $ q = p(n - p + 1) $ (see \eqref{eq:dim_Grassmann}). Then the Young diagram $ \lambda(\sigma) $ is represented by the $ p \times (n - p + 1) $ rectangle.
  Hence we have
  \begin{equation*}
    \begin{split}
      & \max_{\sigma \in \binom{[n + 1]}{p}_{(k_q)}}\left(\sum_{k = 1}^n n_{\lambda(\sigma)}(k)(T_{k - 1} - 2T_k + T_{k + 1}) + T_p\right) \\
      & = \sum_{k = 1}^n n_{[(n - p + 1)^p]}(k)(T_{k - 1} - 2T_k + T_{k + 1}) + T_p                                                        \\
      & = (- T_p - T_{n - p + 1}) + T_p = -T_{n - p + 1}.
    \end{split}
  \end{equation*}
  Thus, Theorem~\ref{thm:weighted_balanced_sum_formula} implies
  \begin{equation*}
    T_{n - p + 1} \leq \sum_{s = 0}^{p(n - p + 1) - 1}\max_{\sigma \in \binom{[n + 1]}{p}_{(k_s)}}\left(\sum_{k = 1}^n n_{\lambda(\sigma)}(k)(T_{k - 1} - 2T_k + T_{k + 1}) + T_p\right).
  \end{equation*}

  Fix $ \epsilon > 0 $.
  Suppose that there exists an integer $ 2 \leq i \leq p(n - p + 1) - 1 $ such that the inequality
  \begin{equation*}
    \sum_{s = 0}^{i - 1}\max_{\sigma \in \binom{[n + 1]}{p}_{(k_s)}}\left(\sum_{k = 1}^n n_{\lambda(\sigma)}(k)(T_{k - 1} - 2T_k + T_{k + 1}) + T_p\right) \leq (1 - \epsilon)\min(T_p, T_{n - p + 1})
  \end{equation*}
  holds on a subset $ S \subseteq \mathbb{R} $ of infinite measure. Let $ i' $ be the minimal integer satisfying this condition. Then the following inequality holds on a subset $ S' \subseteq S $ of infinite measure:
  \begin{equation*}
    \max_{\sigma \in \binom{[n + 1]}{p}_{(k_{i' - 1})}}\left(\sum_{k = 1}^n n_{\lambda(\sigma)}(k)(T_{k - 1} - 2T_k + T_{k + 1}) + T_p\right) < 0.
  \end{equation*}
  Indeed, if this relation does not hold, then we obtain
  \begin{equation*}
    \begin{split}
      (1 - \epsilon) & \min(T_p, T_{n - p + 1}) - \sum_{s = 0}^{i' - 2}\max_{\sigma \in \binom{[n + 1]}{p}_{(k_s)}}\left(\sum_{k = 1}^n n_{\lambda(\sigma)}(k)(T_{k - 1} - 2T_k + T_{k + 1}) + T_p\right) \\
      & \geq \max_{\sigma \in \binom{[n + 1]}{p}_{(k_{i' - 1})}}\left(\sum_{k = 1}^n n_{\lambda(\sigma)}(k)(T_{k - 1} - 2T_k + T_{k + 1}) + T_p\right) \geq 0 \, //_{S},
    \end{split}
  \end{equation*}
  where the symbol $ //_{S} $ indicates that the inequality holds for all $ r \in S $ outside a subset of $ S $ of finite measure.
  This contradicts the minimality of $ i' $. Thus, by Remark~\ref{rem:T-seq}, it follows that
  \begin{equation*}
    \max_{\sigma \in \binom{[n + 1]}{p}_{(k_{s})}}\left(\sum_{k = 1}^n n_{\lambda(\sigma)}(k)(T_{k - 1} - 2T_k + T_{k + 1}) + T_p\right) < C \log T \, //_{S'}
  \end{equation*}
  for $ s \geq i' - 1 $ and $ C > 1 $.
  Hence we have
  \begin{equation*}
    \begin{split}
      T_{n - p + 1} & \leq \sum_{s = 0}^{p(n - p + 1) - 1}\max_{\sigma \in \binom{[n + 1]}{p}_{(k_s)}}\left(\sum_{k = 1}^n n_{\lambda(\sigma)}(k)(T_{k - 1} - 2T_k + T_{k + 1}) + T_p\right) \\
      & \leq (1 - \epsilon)\min(T_p, T_{n - p + 1}) + C' \log T \, //_{S'},
    \end{split}
  \end{equation*}
  which leads to a contradiction for some sufficiently large $ r \in S' \subseteq \mathbb{R} $.
\end{proof}
\subsection{Further Combinatorial Property of Standard Young Tableaux}
\ \par
The remainder of this section is devoted to purely combinatorial arguments related to Proposition~\ref{prop:Young_diag_count_formula}. Readers primarily interested in Nevanlinna theory may skip this part.
\begin{defn}
  Fix an integer $ 1 \leq p \leq n + 1 $. Let $ \lambda $ be a Young diagram, and let $ T $ be a standard Young tableau of shape $ \lambda $. We place $ T $ on $ \mathbb{R}^2 $ in the same manner as in \textup{Definition~\ref{def:phi}}. For each integer $ 1 \leq k \leq n $, define $ n_{T}(k) $ as the sum of the entries in the boxes of $ T $ that intersect the line $ x = k $.
\end{defn}
The following statement is equivalent to Proposition~\ref{prop:Young_diag_count_formula}.
\begin{cor}
  Let $ 1 \leq k \leq n $. Then the following formula holds:
  \begin{equation*}
    \begin{split}
      \phi_{[(n - p + 1)^p]}(k) & \coloneqq \sum_{T \in \mathrm{SYT}([(n - p + 1)^p])}n_{T}(k)         \\
      & = \sum_{\substack{1 \leq i \leq p, 1 \leq j \leq n - p + 1  \\ \text{s.t.} \, x = k \, \text{intersects the} \, (i, j)\text{-box on} \, \mathbb{R}^2}}\sum_{[j^{i - 1}, j - 1]
        \subseteq \tau \subseteq [(n - p + 1)^p] - (i, j)}(|\tau| + 1)f_{\tau}f_{\widehat{\tau}'} \\
      & =
      \left\{
      \begin{array}{ll}
        \frac{p(p(n - p + 1) + 1)}{n + 1}f_{[(n - p + 1)^p]}k                     & (1 \leq k \leq n - p + 1), \\
        \frac{(n - p + 1)(p(n - p + 1) + 1)}{n + 1}f_{[(n - p + 1)^p]}(n + 1 - k) & (n - p + 1 \leq k \leq n),
      \end{array}
      \right.
    \end{split}
  \end{equation*}
  where $ [(n - p + 1)^p] - (i, j) $ denotes the diagram obtained from $ [(n - p + 1)^p] $ by removing the $ (i, j) $-box, which is not necessarily a Young diagram, and $ \widehat{\tau}' $ is the complementary Young diagram of $ \tau $ in $ [(n - p + 1)^p] - (i, j) $.
\end{cor}
\begin{proof}
  The first equality can be proved in almost the same way as the computation of $ f_{[(n - p + 1)^p, 1]} $ in the proof of Proposition~\ref{prop:Young_diag_count_formula}. Note that, in this context, the number $ |\tau| + 1 = \rk(\tau) + 1 $ corresponds to the $ (i, j) $-entry of a standard Young tableau of shape $ [(n - p + 1)^p] $.
  For the proof of the second equality, we consider the case where $ 1 \leq k \leq n - p + 1 $ and $ p \leq n - p + 1 $. The other cases can be handled in a similar manner.
  Fix a standard Young tableau $ T $ of shape $ \tau \subseteq \lambda \coloneqq [(n - p + 1)^p] $ and a standard Young tableau $ \widehat{T} $ of shape $ \widehat{\tau} $ which is complementary to $ \tau $ in the $ p \times (n - p + 1) $ rectangle $ \lambda $.
  Then the pair $ (T, \widehat{T}) $ can be identified with a standard Young tableau $ T' $ of shape $ \lambda $, in which a stepwise path is drawn that forms part of the contour of $ \tau $. Here, $ T' $ is obtained by gluing $ T $ to $ \widehat{T}_0 $ along this path, where the diagram $ \widehat{T}_0 $ is obtained by replacing numbers $ 1, 2, \ldots, |\widehat{\tau}| $ in $ \widehat{T} $ with $ p(n - p + 1), p(n - p + 1) - 1, \ldots, p(n - p + 1) - |\widehat{\tau}| + 1 $.
  Hence, for each $ 0 \leq j \leq n $, we obtain
  \begin{equation*}
    \begin{split}
      & \sum_{j \in \sigma_{\leq n}^c}f_{\lambda(\sigma)}f_{\widehat{\lambda(\sigma)}}                                                                                                                       \\
      & = \sum_{T' \in \mathrm{SYT}(\lambda)}\#\left\{(T', \sigma) \middle| \, \sigma \in \binom{[n + 1]}{p} \, \text{such that} \, T'|_{\lambda(\sigma)} \in \mathrm{SYT}(\lambda(\sigma)) \, \text{and} \, j \in \sigma_{\leq n}^c\right\},
    \end{split}
  \end{equation*}
  where $ T'|_{\lambda(\sigma)} $ denotes the diagram obtained by restricting $ T' $ to the shape $ \lambda(\sigma) $.
  Fix a standard Young tableau $ T' $ of shape $ \lambda $. We now count the number of $ \sigma \in \binom{[n + 1]}{p} $ such that $ T'|_{\lambda(\sigma)} \in \mathrm{SYT}(\lambda(\sigma)) $ and $ j \in \sigma_{\leq n}^c $, and denote this number by $ N(T') $. We claim that
  \begin{equation*}
    N(T') = p(n - p + 1) + 1 - (n_{T'}(j + 1) - n_{T'}(j)).
  \end{equation*}
  If this can be shown, then by Proposition~\ref{prop:Young_diag_count_formula}, we obtain
  \begin{equation*}
    \begin{split}
      \frac{(n - p + 1)(p(n - p + 1) + 1)}{n + 1}f_{\lambda}
      & = \sum_{T' \in \mathrm{SYT}(\lambda)}N(T')                                           \\
      & = (p(n - p + 1) + 1)f_{\lambda} - (\phi_{\lambda}(j + 1) - \phi_{\lambda}(j))
    \end{split}
  \end{equation*}
  for $ 0 \leq j \leq n - p $ (where we set $ \phi_{\lambda}(0) \coloneqq 0 $), and in particular,
  \begin{equation*}
    \phi_{\lambda}(1) = \frac{p(p(n - p + 1) + 1)}{n + 1}f_{\lambda}.
  \end{equation*}
  This implies that $ \phi_{\lambda}(k) = \frac{p(p(n - p + 1) + 1)}{n + 1}f_{\lambda}k $ $ (1 \leq k \leq n - p + 1)$.
  Let $ B_{i, j} $ denote the $ (i, j) $-entry of $ T' $. Define the set $ P(T') $ by
  \begin{equation*}
    P(T') \coloneqq \left\{\sigma \in \binom{[n + 1]}{p} \middle| \, T'|_{\lambda(\sigma)} \in \mathrm{SYT}(\lambda(\sigma))\right\}.
  \end{equation*}
  Clearly, we have $ \#P(T') = p(n - p + 1) + 1 $. Note that $ \sigma \in P(T') $ does not satisfy $ j \in \sigma_{\leq n}^c $ if and only if the number $ |\lambda(\sigma)| $ belongs to the set $ P'(T') $, defined by
  \begin{equation*}
    \begin{split}
      & P'(T') \coloneqq \\
      & \left\{
      \begin{array}{ll}
        \coprod_{k = 1}^{j + 1}\{B_{p - j + k - 1, k - 1}, B_{p - j + k - 1, k - 1} + 1,\ldots ,B_{p - j + k - 1, k} - 1\} & (0 \leq j \leq p - 1),     \\
        \coprod_{k = 1}^{p}\{B_{k, j - p + k}, B_{k, j - p + k} + 1, \ldots, B_{k, j - p + k + 1} - 1\}                    & (p - 1 \leq j \leq n - p),
      \end{array}
      \right.
    \end{split}
  \end{equation*}
  where we set $ B_{\ast, 0} \coloneqq 0 $.
  Hence, we obtain
  \begin{align*}
     & N(T') = \# P(T') - \#P'(T')                       \\
     & = \left\{
    \begin{array}{ll}
      p(n - p + 1) + 1 - \sum_{k = 1}^{j + 1}(B_{p - j + k - 1, k} - B_{p - j + k - 1, k - 1}) & (0 \leq j \leq p - 1),     \\
      p(n - p + 1) + 1 - \sum_{k = 1}^p (B_{k, j - p + k + 1} - B_{k, j - p + k})              & (p - 1 \leq k \leq n - p),
    \end{array}
    \right.
    \\
     & = p(n - p + 1) + 1 - (n_{T'}(j + 1) - n_{T'}(j)).
  \end{align*}

\end{proof}
\begin{ex}
  Let $ n = 4 $ and $ p = 2 $. We compute $ \phi_{\lambda}(k) $ $ (1 \leq k \leq 4)$ for the standard Young tableaux of shape $ \lambda = [(n - p + 1)^p] = [3^2] = (3, 3) $ \textup{(Figure~\ref{fig:tableaux_[3^2]})}.
  \begin{figure}[htbp]
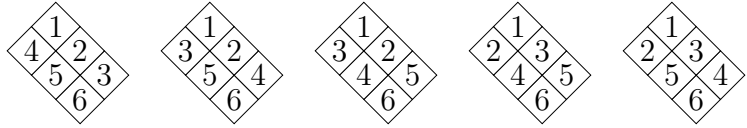

    \centering
    \begin{equation*}
      \Yvcentermath1
      \YRussian
      \young(63,52,41)
      \quad
      \young(64,52,31)
      \quad
      \young(65,42,31)
      \quad
      \young(65,43,21)
      \quad
      \young(64,53,21)
    \end{equation*}
    \caption{\texorpdfstring{Standard Young tableaux of shape $ [3^2] $}{Standard Young tableaux of shape [3^2]}}
    \label{fig:tableaux_[3^2]}
  \end{figure}

  Then we obtain
  \begin{align*}
    \phi_{[3^2]}(1) & = 4 + 3 + 3 + 2 + 2 = 14 = \frac{14}{5}f_{[3^2]},                                         \\
    \phi_{[3^2]}(2) & = 6 + 6 + 5 + 5 + 6 = 28 = \frac{14}{5}f_{[3^2]} \cdot 2,                                 \\
    \phi_{[3^2]}(3) & = 8 + 8 + 8 + 9 + 9 = 42 = \frac{14}{5}f_{[3^2]} \cdot 3 = \frac{21}{5}f_{[3^2]} \cdot 2, \\
    \phi_{[3^2]}(4) & = 3 + 4 + 5 + 5 + 4 = 21 = \frac{21}{5} f_{[3^2]}.
  \end{align*}
  In another expression, for example, $ \phi_{[3^2]}(2) $ is expressed as
  \begin{equation*}
    f_{\varnothing}f_{\yngs(0.4,3,2)} + 4f_{\yngs(0.4,2,1)}f_{\yngs(0.4,1^2)} + 5f_{\yngs(0.4,3,1)}f_{\yngs(0.4,1)} = 5 + 8 + 15 = 28.
  \end{equation*}
\end{ex}
\section{Geometric Equalities and Inequalities Based on the Weyl Peculiar Relation}\label{sec:geom_ineq}

In this section, we provide an elementary geometric interpretation of the generalized Weyl peculiar relation by applying our theorem to exponential curves.
\subsection{Exponential Curves}
\begin{defn}
  Let $ w_i = a_i - \sqrt{-1}b_i $ $ (i = 0, 1, \ldots, n)$ be mutually distinct complex numbers, where $ a_i $ and $ b_i $ are real numbers. Define a map $ \mathbf{e} $ by
  \begin{equation*}
    \mathbf{e} : \mathbb{C} \to \mathbb{P}^n, \quad z \mapsto (e^{w_0z} : e^{w_1z} : \cdots : e^{w_nz}).
  \end{equation*}
  This map is called an \textbf{exponential curve}. We denote the $ p $-th associated curve of $ \mathbf{e} $ by $ \mathbf{E}^{(p)} $.
\end{defn}
\begin{rem}
  The map $ \mathbf{e} $ is non-degenerate. Indeed, since the Wronskian of $ \mathbf{e} $ is given by the Vandermonde determinant multiplied by a non-zero entire function, it is not identically zero due to the assumption that the $ w_i $ $ (i = 0, 1, \ldots, n)$ are mutually distinct.

  In this section, we assume that $ \mathbf{E}^{(p)} $ is non-degenerate. If $ \mathbf{E}^{(p)} $ is degenerate, we can obtain a non-degenerate curve by slightly perturbing the $ w_i $. Indeed,
  \begin{equation*}
    \prod_{\substack{(i_0, i_1, \ldots, i_{p - 1}), \, (i_0', i_1', \ldots, i_{p - 1}') \in \binom{[n + 1]}{p} \\ (i_0, i_1, \ldots, i_{p - 1}) \neq (i_0', i_1', \ldots, i_{p - 1}')}}\sum_{k = 0}^{p - 1}(a_{i_k} - a_{i_k'}) = 0
  \end{equation*}
  is a closed condition on $ \mathbb{R}^{n + 1} $.
\end{rem}

\subsection{Application to Geometric Inequalities}
\ \par
The order function of exponential curves can be explicitly computed.
\begin{prop}[\cite{Ahlfors}, pp.~28--31; \cite{Weyl_2}, pp.~94--106; B. Shiffman \cite{Shiffman}, pp.~630--631]\label{prop:exp_order}
  \begin{equation*}
    T\{\mathbf{E}^{(p)}\}(r) = \frac{L_{p}}{2\pi}r + O(1),
  \end{equation*}
  where $ L_{p} $ denotes the perimeter of the convex hull of the finite set
  \begin{equation*}
    V_{p} \coloneqq \left\{\sum_{k = 0}^{p - 1}(a_{i_k}, b_{i_k})\right\}_{(i_0, i_1, \ldots, i_{p - 1}) \in \binom{[n + 1]}{p}} \subseteq \mathbb{R}^2.
  \end{equation*}
\end{prop}
Let $ L_i^{(p)} $ denotes the perimeter of the convex hull of the finite set
\begin{equation*}
  V_i^{(p)} \coloneqq \left\{\sum_{\sigma \in I}\sum_{k = 0}^{p - 1}(a_{i_k}, b_{i_k})\right\}_{I = \left\{\sigma = (i_0, i_1, \ldots, i_{p - 1}) \in \binom{[n + 1]}{p} \right\} \, \textit{s.t.} \#I = i} \subseteq \mathbb{R}^2.
\end{equation*}
Naturally, $ L_i^{(p)} $ appears as a coefficient in the order function $ T_i\{\mathbf{E}^{(p)}\} $. By applying Theorem~\ref{thm:pec_rel} to the exponential curve $ \mathbf{e} $, we obtain the following formula.
\begin{cor}\label{cor:perim}
  \begin{equation*}
    L_{p - 1} + L_{p + 1} = L_2^{(p)}.
  \end{equation*}
\end{cor}
We present an alternative proof of this corollary. Recall a fundamental operation in convex geometry.
Let $ A $ and $ B $ be subsets of $ \mathbb{R}^n $. The \textbf{Minkowski sum} of $ A $ and $ B $ is defined by
\begin{equation*}
  A + B \coloneqq \{a + b \in \mathbb{R}^n \mid a \in A, \, b \in B\}.
\end{equation*}
The following is a fundamental property of the Minkowski sum.
\begin{lem}\label{lem:Minkowski_sum}
  Let $ A $, $ B \subseteq \mathbb{R}^2 $ be convex bodies. Then we have
  \begin{equation*}
    \mathrm{perimeter}(A + B) = \mathrm{perimeter}(A) + \mathrm{perimeter}(B),
  \end{equation*}
  where $ \mathrm{perimeter}(C) $ denotes the perimeter of a convex body $ C \subseteq \mathbb{R}^2 $.
\end{lem}

\begin{proof}[Another proof of \textup{Corollary~\ref{cor:perim}}]
  \ \par
  It is enough to show that
  \begin{align*}
    V_{p - 1} + V_{p + 1} = V_2^{(p)}.
  \end{align*}
  Indeed, Corollary~\ref{cor:perim} follows directly from this equality and Lemma~\ref{lem:Minkowski_sum}. For $ (i_0, i_1, \ldots, i_{p - 1}) $, $ (i_0', i_1', \ldots, i_{p - 1}') \in \binom{[n + 1]}{p} $, we write
  \begin{equation*}
    (i_0, i_1, \ldots, i_{p - 1}) < (i_0', i_1', \ldots, i_{p - 1}')
  \end{equation*}
  if there exists $ s \in \{0, 1, \ldots, p - 1\} $ such that $ i_k = i_k' $ $ (k = 0, 1, \ldots, s - 1)$ and  $ i_s < i_s' $. (This is the lexicographical order.) Let $ c_k \coloneqq (a_k, b_k) $ for $ 0 \leq k \leq n $. Then $ V_l $ $ (l = p - 1, p + 1)$ and $ V_2^{(p)} $ are given by
  \begin{align*}
    V_l       & = \left\{\sum_{k = 0}^{l - 1}c_{i_k}\right\}_{(i_0, i_1, \ldots, i_{l - 1}) \in \binom{[n + 1]}{l}} \quad (l = p - 1, p + 1),                                                    \\
    V_2^{(p)} & = \left\{\sum_{k = 0}^{p - 1}c_{i_k} + \sum_{k = 0}^{p - 1}c_{i_k'}\right\}_{\substack{(i_0, i_1, \ldots, i_{p - 1}), \, (i_0', i_1', \ldots, i_{p - 1}') \in \binom{[n + 1]}{p} \\ (i_0, i_1, \ldots, i_{p - 1}) < (i_0', i_1', \ldots, i_{p - 1}')}}.
  \end{align*}
  Let $ \sum_{k = 0}^{p - 1}c_{i_k} + \sum_{{k = 0}}^{p - 1}c_{i_k'} $ $ ((i_0, i_1, \ldots, i_{p - 1}) < (i_0', i_1', \ldots, i_{p - 1}')) $ be an element of $ V_2^{(p)} $. Then, by the definition of the lexicographical order, we obtain the minimal integer $ s $ such that $ i_s < i_s' $. Since $ i_{s - 1}' = i_{s - 1} < i_s $, the sequence $ (i_0', i_1', \ldots, i_{s - 1}', i_s, i_s', \ldots, i_{p - 1}') $ belongs to $ \binom{[n + 1]}{p + 1} $. Thus we obtain
  \begin{equation*}
    \sum_{k = 0}^{p - 1} c_{i_k} + \sum_{k = 0}^{p - 1} c_{i_k'} = \sum_{k = 0, \, k \neq s}^{p - 1} c_{i_k} + \left(c_{i_s} + \sum_{k = 0}^{p - 1}c_{i_k'}\right) \in V_{p - 1} + V_{p + 1}.
  \end{equation*}
  The reverse inclusion is also straightforward.
\end{proof}
\begin{rem}
  A similar interpretation can be applied to the Kronecker product of two exponential curves. Let $ \mathbf{x}_1 = (x_{1, i})_{i = 0, 1, \ldots, n} $ and $ \mathbf{x}_2 = (x_{2, j})_{j = 0, 1, \ldots, m} $ be reduced representations of holomorphic curves. Then $ (x_{1, i}x_{2, j})_{i = 0, 1, \ldots, n, j = 0, 1, \ldots, m} $ defines a holomorphic curve in $ (n + 1)(m + 1) $-space. This curve is called the \textbf{Kronecker product} of $ \mathbf{x}_1 $ and $ \mathbf{x}_2 $, denoted by $ \mathbf{x}_1 \times \mathbf{x}_2 $. Then we obtain the following equality \textup{(\cite{Weyl_2}, p.~107)}:
  \begin{equation}\label{eq:Kronecker_prod}
    T\{\mathbf{x}_1 \times \mathbf{x}_2\} = T\{\mathbf{x}_1\} + T\{\mathbf{x}_2\}.
  \end{equation}
  If $ \mathbf{x}_1 $ and $ \mathbf{x}_2 $ are exponential curves, then so is their Kronecker product $ \mathbf{x}_1 \times \mathbf{x}_2 $. Thus, the identity \eqref{eq:Kronecker_prod} for exponential curves is simply a reformulation of \textup{Lemma~\ref{lem:Minkowski_sum}}, applied to the convex hulls of two finite sets in $\mathbb{R}^2$.
\end{rem}
\begin{rem}
  The inequality \eqref{eq:concave_rel} implies that the sequence $\{T_i\}_{i = 0}^{n + 1}$ is concave with no $O(\log T)$ error term when the holomorphic curve in question is an exponential curve. In other words, the second-order difference $ T_{k - 1} - 2T_k + T_{k + 1}$ is always non-positive for every $ 1 \leq k \leq n $. Consequently, the same holds for $ L_{k - 1} - 2L_k + L_{k + 1} $. This fact can also be verified  directly; see \cite{Weyl_2}, \textup{pp.~124--125}.
\end{rem}
Furthermore, one can observe the following notable property for exponential curves.
\begin{lem}\label{lem:sym_ord}
  For $ 0 \leq p \leq n + 1 $, we have
  \begin{align*}
    L_{n - p + 1} = L_p.
  \end{align*}
  Namely, the sequence $ \{L_i\}_{i = 0}^{n + 1} $ is symmetric.
  In particular, by \textup{Proposition~\ref{prop:exp_order}}, we obtain
  \begin{equation*}
    T_{n - p + 1} = T_p + O(1)
  \end{equation*}
  for exponential curves.
\end{lem}
\begin{proof}
  Let $ A \coloneqq \sum_{k = 0}^n a_k$ and  $B \coloneqq \sum_{k = 0}^n b_k $. Then we have
  \begin{equation*}
    V_{n - p + 1} = \{(A, B) - (a, b)\}_{(a, b) \in V_{p}}.
  \end{equation*}
  This implies $ L_{n - p + 1} = L_p $.
\end{proof}
\begin{prob}
  How are  $ T_{n - p + 1} $ and $ T_p $ related for general holomorphic curves?
\end{prob}
In a similar way to Corollary~\ref{cor:perim}, the following corollary can be derived.
\begin{cor}\label{cor:gen_perim}
  Let $ \{(a_i, b_i)\}_{i = 0}^n \subseteq \mathbb{R}^2 $ be a finite set. Let $ p $ and $ i $ be positive integers with $ p \leq n $ and $ 1 \leq i \leq p(n - p + 1) $. Then the following inequality holds:
  \begin{align}\label{eq:gen_perim_ineq}
    L_p \leq \sum_{s = 1}^{i - 1} \max_{\sigma \in \binom{[n + 1]}{p}_{(k_s)}}\left(\sum_{k = 1}^n n_{\lambda(\sigma)}(k)(L_{k - 1} - 2L_k + L_{k + 1})\right) + iL_p \leq L_i^{(p)}.
  \end{align}
\end{cor}
We present an example of Corollary~\ref{cor:perim} and Corollary~\ref{cor:gen_perim}.
\begin{ex}
  Let $ n = 5 $, $ p = 2 $, $ i = 3 \leq p(n - p + 1) = 8 $. \textup{Corollary~\ref{cor:perim}} implies that
  \begin{equation*}
    L_1 + L_3 = L_2^{(2)}.
  \end{equation*}
  Moreover, \textup{Corollary~\ref{cor:gen_perim}} together with \textup{Lemma~\ref{lem:sym_ord}} implies that
  \begin{equation*}
    L_2 \leq \max(2L_3, 2L_1 + L_2) \leq L_3^{(2)}.
  \end{equation*}
  When $ \{(a_i, b_i)\}_{i = 0}^5 = \{(0, 0), (0, 1), (1, 0), (1, 1), (1, 2), (2, 1)\} $, we have
  \begin{align*}
    L_1       = 2   & + 3\sqrt{2}, \quad
    L_2 = 4 + \sqrt{2} + 2\sqrt{5}, \quad
    L_3 = 8 + 2\sqrt{2},                      \\
    \quad L_2^{(2)} & = 10 + 5\sqrt{2}, \quad
    L_3^{(2)} = 8 + 4\sqrt{2} + 4\sqrt{5}.
  \end{align*}
  In this case, \textup{Corollary~\ref{cor:perim}} and \textup{Corollary~\ref{cor:gen_perim}} respectively take the forms:
  \begin{align*}
    L_1 + L_3                                  & = 10 + 5\sqrt{2} = L_2^{(2)},                                        \\
    L_2 < \max(2L_3, 2L_1 + L_2)  = 2L_1 + L_2 & = 8 + 7\sqrt{2} + 2\sqrt{5} = 22.3 \cdots < 22.6 \cdots = L_3^{(2)}.
  \end{align*}
\end{ex}
We now present a special exponential curve that attains the equality in \eqref{eq:gen_perim_ineq} of Corollary~\ref{cor:gen_perim}. This curve was constructed by H. Fujimoto \cite{Fujimoto_1} as the best possible example for the truncated version of the defect relation (see \eqref{eq:trunc_def_rel}).
\begin{prop}\label{prop:sp_curve}
  Define the exponential curve $ \mathbf{e}_0 $ by
  \begin{equation*}
    \mathbf{e}_0 : \mathbb{C} \to \mathbb{P}^n, \quad z \mapsto (1 : e^z : e^{2z} : \cdots : e^{nz}).
  \end{equation*}
  Equality in \eqref{eq:gen_perim_ineq} is attained when $ \mathbf{e} = \mathbf{e}_0 $.
  In particular, the holomorphic curve $ \mathbf{x} = \mathbf{e}_0 $ attains equality in \eqref{eq:gen_pec_ineq} of \textup{Theorem~\ref{thm:gen_pec_rel}}.
\end{prop}
\begin{proof}
  For each $ 1 \leq k \leq n $, the set $ V_k $ is the line segment connecting $ (\frac{k(k - 1)}{2}, 0) $ and $ (kn - \frac{k(k - 1)}{2}, 0) $. Thus we have
  \begin{equation*}
    L_k = 2k(n - k + 1).
  \end{equation*}
  Hence it follows that
  \begin{equation*}
    \begin{split}
      & L_{k - 1} - 2L_k + L_{k + 1}                                                     \\
      & = 2(k - 1)(n - (k - 1) + 1) - 2(2k(n - k + 1)) + 2(k + 1)(n - (k + 1) + 1) = -4.
    \end{split}
  \end{equation*}
  Therefore, the middle term in \eqref{eq:gen_perim_ineq} can be computed as
  \begin{equation*}
    \begin{split}
      \sum_{s = 1}^{i - 1}\max_{\sigma \in \binom{[n + 1]}{p}_{(k_s)}}\left(-4\sum_{k = 1}^n n_{\lambda(\sigma)}(k)\right) + i \cdot 2p(n - p + 1) & = -4\frac{i(i - 1)}{2} + 2p(in -ip + i) \\
      & = 2i(np - p^2 + p - i + 1).
    \end{split}
  \end{equation*}
  On the other hand, $ V_i^{(p)} $ is also the line segment connecting $ (\frac{ip(p - 1)}{2} + \frac{i(i - 1)}{2}, 0) $ and $ (i(pn - \frac{p(p - 1)}{2})- \frac{i(i - 1)}{2}, 0) $. Hence, we obtain
  \begin{equation*}
    L_i^{(p)} = 2i(np - p^2 + p - i + 1).
  \end{equation*}
\end{proof}

\section{Proofs of the Second Main Theorem and Theorem C}\label{sec:SMT}
\subsection{The Proximity and Counting Functions}
\ \par
In Section~\ref{sec:gen_pec_rel}, we proved the generalized Weyl peculiar relation for $ T_i\{\mathbf{X}^{(p)}\} $ using Lemma~\ref{lem:iT_p-T_i{X^p}_rel}. From a viewpoint focused on the holomorphic curve $ \mathbf{X}^{(p)} $, we can also derive a Second Main Theorem-type statement from Lemma~\ref{lem:iT_p-T_i{X^p}_rel}. This provides a new approach to establishing the defect relations.

\begin{defn}[\cite{Weyl_2}, p.~20, \textup{(3.5)}]
  Let $ \mathbf{B}^{(h)} \in \bigwedge^h \mathbb{C}^{n + 1} $ \textup{(\textit{resp.} $ \mathbf{X}^{(p)} \in \bigwedge^p \mathbb{C}^{n + 1} $)} be a nonzero decomposable $ h $-vector \textup{(\textit{resp.} $ p $-vector)}. Then the \textbf{distance} between $ \mathbf{B}^{(h)} $ and $ \mathbf{X}^{(p)} $ is defined by
  \begin{equation*}
    ||\mathbf{B}^{(h)}:\mathbf{X}^{(p)}|| \coloneqq \frac{|\mathbf{B}^{(h)} \wedge  \mathbf{X}^{(p)}|}{|\mathbf{B}^{(h)}||\mathbf{X}^{(p)}|}.
  \end{equation*}
\end{defn}
\begin{rem}\label{rem:dist}
  $ (1) $
  Suppose $ \mathbf{B}^{(h)} = \mathrm{Span}_{\mathbb{C}}\{\mathbf{b}_0,\ldots,\mathbf{b}_{h-1}\} $ with an orthonormal system, and extend it to an orthonormal basis of $ \mathbb{C}^{n + 1} $.
  When $ \mathbf{X}^{(p)} $ is the $ p $-th associated curve of $ \mathbf{x} $, the wedge product $ \mathbf{B}^{(h)} \wedge \mathbf{X}^{(p)} $ is identified with the holomorphic curve obtained by projecting onto $ (\mathbf{B}^{(h)})^{\perp} $ along $ \mathbf{B}^{(h)} $. Hence,
  \begin{equation*}
    0 \leq \|\mathbf{B}^{(h)} : \mathbf{X}^{(p)}\| \leq 1
  \end{equation*}
  holds, as is evident from geometric considerations.

  \noindent $ (2) $
  If $ p = n + 1 - h $, then by the definition of the Hodge star operator $ \star $ on $ \bigwedge^h \mathbb{C}^{n + 1} $, we obtain
  \begin{equation*}
    ||\mathbf{B}^{(h)} : \mathbf{X}^{(p)}|| = \frac{|(\star \mathbf{B}^{(h)}, \mathbf{X}^{(p)})|}{|\mathbf{B}^{(h)}||\mathbf{X}^{(p)}|} = \frac{|(\star \mathbf{B}^{(h)}, \mathbf{X}^{(p)})|}{|\star\mathbf{B}^{(h)}||\mathbf{X}^{(p)}|}.
  \end{equation*}
\end{rem}
\begin{defn}[\cite{Weyl_2}, p.~110, \textup{(7.2)}]\label{def:prox_func}
  Let $ \mathbf{B}^{(h)} \in \bigwedge^h \mathbb{C}^{n + 1} $ be a nonzero decomposable $ h $-vector. Define the $ k $-th \textbf{proximity function} of $ \mathbf{x} $ along $ \mathbf{B}^{(h)} $ by
  \begin{equation*}
    \widetilde{m}_k(r, \mathbf{B}^{(h)}) \coloneqq \frac{1}{2\pi}\int_0^{2\pi} \log \frac{1}{|| \mathbf{B}^{(h)} : \mathbf{X}^{(k)}(re^{\sqrt{-1}\theta})||}d\theta.
  \end{equation*}
  If $ k = n + 1 - h $ and $ \mathbf{A}^{(k)} \coloneqq \star \mathbf{B}^{(h)} \in (\bigwedge^k \mathbb{C}^{n + 1})^{*} $, we denote this function by $ m_k(r, \mathbf{A}^{(k)}) $.
  For the $ p $-th associated curve $ \mathbf{X}^{(p)} $, we denote the corresponding proximity functions by $ \widetilde{m}_k\{\mathbf{X}^{(p)}\}(r, \mathbf{B}^{(h)}) $ and $ m_k\{\mathbf{X}^{(p)}\}(r, \mathbf{A}^{(k)}) $, respectively.
  Here $ \mathbf{B}^{(h)} $ is a nonzero decomposable $ h $-vector in $ \bigwedge^{h}\mathbb{C}^{\binom{n + 1}{p}} $, and $ \mathbf{A}^{(k)} = \star \mathbf{B}^{(h)} \in (\bigwedge^k \mathbb{C}^{\binom{n + 1}{p}})^{*} $, where $ \star $ denotes the Hodge star operator on $ \bigwedge^h \mathbb{C}^{\binom{n + 1}{p}} $, and $k = \binom{n + 1}{p} - h$.
\end{defn}
\begin{rem}
  $ (1) $ \textup{Remark~\ref{rem:dist}} $ (1) $ implies that $ \widetilde{m}_k \geq 0 $.

  \noindent $ (2) $ If $ h = n + 1 - k $, as observed in \textup{Remark~\ref{rem:dist}} $ (2) $, we obtain
  \begin{equation*}
    m_k(r, \mathbf{A}^{(k)}) = \frac{1}{2\pi}\int_0^{2\pi}\log \frac{|\mathbf{A}^{(k)}||\mathbf{X}^{(k)}|}{|(\mathbf{A}^{(k)}, \mathbf{X}^{(k)})|}d\theta.
  \end{equation*}
  This explains the geometric meaning of the proximity function. When $ \mathbf{X}^{(k)} $ approaches the hyperplane defined by the equation $ (\mathbf{A}^{(k)}, \mathbf{Y}^{(k)}) = 0 $ in $ \mathbb{C}^{\binom{n + 1}{k}} $, the denominator of the fraction in the integrand of $ m_k(r, \mathbf{A}^{(k)}) $ tends to zero.
  Accordingly, $ m_k(r, \mathbf{A}^{(k)}) $ becomes large. Namely, the proximity function reflects how closely the curve approaches the hyperplane determined by $ \mathbf{A}^{(k)} $.
\end{rem}
\begin{lem}\label{lem:m_p}
  Let $ h = \binom{n + 1}{p} - 1 $, and
  let $ \mathbf{B}^{(h)} \in \bigwedge^h \mathbb{C}^{\binom{n + 1}{p}} = \bigwedge^h \mathbb{C}^{h + 1} $ be a nonzero decomposable $ h $-vector. Assume that $ \star{\mathbf{B}}^{(h)} \in (\bigwedge^1 \mathbb{C}^{\binom{n + 1}{p}})^{*} = (\mathbb{C}^{\binom{n + 1}{p}})^{*} $ arises from an element of $ (\bigwedge^{p}\mathbb{C}^{n + 1})^{*} $ via the Pl\"ucker embedding, and is decomposable as such an element.
  Then we have
  \begin{equation*}
    \widetilde{m}_1\{\mathbf{X}^{(p)}\}(r, \mathbf{B}^{(h)}) = m_p(r, \star\mathbf{B}^{(h)}).
  \end{equation*}
\end{lem}
\begin{proof}
  Note that, by assumption, $ m_p(r, \star \mathbf{B}^{(h)}) $ is indeed well-defined. Since $ h = \binom{n + 1}{p} - 1 $, we obtain
  \begin{equation*}
    ||\mathbf{\mathbf{B}}^{(h)} : \mathbf{X}^{(p)}|| = \frac{|\mathbf{B}^{(h)} \wedge \mathbf{X}^{(p)}|_{\bigwedge^{h + 1}\mathbb{C}^{\binom{n + 1}{p}}}}{|\mathbf{B}^{(h)}|_{\bigwedge^{h}\mathbb{C}^{\binom{n + 1}{p}}}|\mathbf{X}^{(p)}|_{\bigwedge^{1}\mathbb{C}^{\binom{n + 1}{p}}}} = \frac{|(\star\mathbf{B}^{(h)}, \mathbf{X}^{(p)})_{\bigwedge^{p}\mathbb{C}^{n + 1}}|}{|\star\mathbf{B}^{(h)}|_{\bigwedge^{p}\mathbb{C}^{n + 1}}|\mathbf{X}^{(p)}|_{\bigwedge^{p}\mathbb{C}^{n + 1}}},
  \end{equation*}
  where the notations $ |\cdot|_{\bigwedge^k \mathbb{C}^l} $ and $ (\cdot, \cdot)_{\bigwedge^k \mathbb{C}^{l}} $ denote the application of $ |\cdot| $ and $ (\cdot, \cdot) $ on $ \bigwedge^k \mathbb{C}^l$, respectively. From this, our claim follows.
\end{proof}
\begin{rem}\label{rem:gen_m}
  $ (1) $ A general vector $ \star \mathbf{B}^{\left(\binom{n + 1}{p} - 1\right)} \in (\mathbb{C}^{\binom{n + 1}{p}})^{*} $ does not belong to $ (\bigwedge^p \mathbb{C}^{n + 1})^{*} $ under the Pl\"{u}cker embedding. Moreover, even if it does belong to $ (\bigwedge^p \mathbb{C}^{n + 1})^{*} $, it is not necessarily decomposable.

  \noindent $ (2) $ As a generalization of the definition of $ m_p(r, \mathbf{A}^{(p)}) $ \textup{(Definition~\ref{def:prox_func})}, for a $ p $-vector $ \mathbf{A}^{(p)} \in (\bigwedge^p \mathbb{C}^{n + 1})^{*}$ \textup{(\textit{not necessarily decomposable})}, we define
  \begin{equation*}
    m_p(r, \mathbf{A}^{(p)}) \coloneqq \widetilde{m}_1\{\mathbf{X}^{(p)}\}(r, \star\mathbf{A}^{(p)}),
  \end{equation*}
  where $ \star $ is the Hodge star operator on $ (\bigwedge^1 \mathbb{C}^{\binom{n + 1}{p}})^{*} = (\mathbb{C}^{\binom{n + 1}{p}})^{*} $.
\end{rem}
\begin{defn}[\cite{Weyl_2}, p.~79, \textup(2.6\textup)]\label{def:count_func}
  Let $ \mathbf{A}^{(k)} \in (\bigwedge^k\mathbb{C}^{n + 1})^{*} $ be a nonzero decomposable $ k $-vector, and fix a positive constant $ r_0 > 0 $ \textup{(\textit{sufficiently small})}. Then the $ k $-th \textbf{counting function} $ N_k(r, \mathbf{A}^{(k)}) $ of $ \mathbf{x} $ for $ \mathbf{A}^{(k)} $ is defined by
  \begin{equation*}
    N_k(r, \mathbf{A}^{(k)}) \coloneqq \int_{r_0}^r \sum_{|z| < t}v_k(z, \mathbf{A}^{(k)})\frac{dt}{t} \quad (r \geq r_0),
  \end{equation*}
  where $ v_k(z_0, \mathbf{A}^{(k)}) $ denotes the order of vanishing of $ (\mathbf{A}^{(k)}, \mathbf{X}_{\mathrm{red}}^{(k)}) $ at $ z_0 $.
\end{defn}
\begin{thm}[\textbf{First Main Theorem} \textit{for rank} $ k $, \cite{Weyl_2}, p.~80]\label{thm:FMT}
  Let $ \mathbf{A}^{(k)} $ be a nonzero decomposable $ k $-vector in $ (\bigwedge^k \mathbb{C}^{n + 1})^{*} $. Then the following identity holds:
  \begin{equation*}
    N_k(r, \mathbf{A}^{(k)}) + m_k(r, \mathbf{A}^{(k)}) - m_k(r_0, \mathbf{A}^{(k)}) = T_k(r) \quad (r \geq r_0).
  \end{equation*}
\end{thm}
\begin{rem}\label{rem:FMT}
  As a generalization of the definition of $ N_k(r, \mathbf{A}^{(k)}) $ \textup{(Definition~\ref{def:count_func})}, for a $ k $-vector $ \mathbf{A}^{(k)} \in (\bigwedge^k \mathbb{C}^{n + 1})^{*} $ \textup{(\textit{not necessarily decomposable})}, we define
  \begin{equation*}
    N_k(r, \mathbf{A}^{(k)}) \coloneqq T_k(r) - m_k(r, \mathbf{A}^{(k)}) + m_k(r_0, \mathbf{A}^{(k)}) \quad (r \geq r_0).
  \end{equation*}
\end{rem}
\subsection{The Second Main Theorem and Theorem C}
\ \par
\begin{defn}[\cite{Weyl_2}, pp.~260 \textit{and} 262]\label{def:Psi-M}
  Let $ k \geq 1 $, $ 1 \leq h \leq n + 1 - k $, and $ 0 \leq i \leq h - 1 $ be integers.
  Set $ l \coloneqq h - i $, and fix a nonzero decomposable $ h $-vector $ \mathbf{B}^{(h)} = \mathbf{b}_0 \wedge \mathbf{b}_1 \wedge \cdots \wedge \mathbf{b}_{h - 1} \in \bigwedge^h\mathbb{C}^{n + 1} $. Assume that $ \mathbf{b}_0, \mathbf{b}_1, \ldots, \mathbf{b}_{h - 1} $ form an orthonormal basis of $ \mathbf{B}^{(h)} $. Then we define
  \begin{equation*}
    \Psi_i^k(\mathbf{B}^{(h)})(z) \coloneqq \frac{1}{\binom{h}{l}}\sum_{\mathbf{B}^{(l)} \subseteq \mathbf{B}^{(h)}} \frac{||\mathbf{B}^{(l)} : \mathbf{X}^{(k)}(z)||^2}{||\mathbf{B}^{(l)} : \mathbf{X}^{(k - 1)}(z)||^2},
  \end{equation*}
  where the sum runs over all $ l $-vectors spanned by $ l $ vectors chosen from the set $ \{\mathbf{b}_0, \mathbf{b}_1, \ldots, \mathbf{b}_{h - 1}\} $. We also define
  \begin{equation*}
    M^{(i)}_{k}(r, \mathbf{B}^{(h)}) \coloneqq \frac{1}{4\pi}\int_0^{2\pi}\log \frac{1}{\Psi_i^{k + i}(\mathbf{B}^{(h)})(re^{\sqrt{-1}\theta})}d\theta.
  \end{equation*}
  For the $ p $-th associated curve, we denote this function by $ M_k^{(i)}\{\mathbf{X}^{(p)}\}(r, \mathbf{B}^{(h)}) $. Here, $ 1 \leq h \leq \binom{n + 1}{p} - k $, $ 0 \leq i \leq h - 1 $, and $ \mathbf{B}^{(h)} \in \bigwedge^h\mathbb{C}^{\binom{n + 1}{p}} $ is a nonzero decomposable $ h $-vector.
\end{defn}
\begin{rem}\label{rem:M_k}
  $ (1) $ The function $ \Psi_i^{k} $ does not depend on the choice of the orthonormal basis $ \mathbf{b}_0, \mathbf{b}_1, \ldots, \mathbf{b}_{h - 1} $. Therefore, $ M_k^{(i)}(r, \mathbf{B}^{(h)}) $ is well-defined.

  \noindent $ (2) $ Since $ \Psi_i^k \leq 1$, it follows that $ M_k^{(i)} \geq 0 $.

  \noindent $ (3) $ If $ i = 0 $, then $ l = h $. Since $ \mathbf{B}^{(h)} $ is the unique $ h $-vector spanned by $ h $ vectors chosen from the set $ \{\mathbf{b}_0, \mathbf{b}_1, \ldots, \mathbf{b}_{h - 1}\} $, we obtain
  \begin{equation*}
    M_k^{(0)}(r, \mathbf{B}^{(h)}) = \frac{1}{4\pi}\int_0^{2\pi} \log \frac{||\mathbf{B}^{(h)} : \mathbf{X}^{(k - 1)}(re^{\sqrt{-1}\theta})||^2}{||\mathbf{B}^{(h)} : \mathbf{X}^{(k)}(re^{\sqrt{-1}\theta})||^2}d\theta = \widetilde{m}_k(r, \mathbf{B}^{(h)}) - \widetilde{m}_{k - 1}(r, \mathbf{B}^{(h)}).
  \end{equation*}
\end{rem}
\begin{defn}[\cite{Weyl_2}, p.~259]\label{def:in_gen_pos}
  Let $ k \geq 1 $ be an integer, and let $ \{\mathbf{B}_j^{(h)}\}_{j = 1}^d \subseteq \bigwedge^h \mathbb{C}^{n + 1} $ be a finite set of nonzero decomposable $ h $-vectors. The set $ \{\mathbf{B}_j^{(h)}\}_{j = 1}^d $ is said to be \textbf{in general position for $ k $} if, for every $ 0 \leq i \leq h - 1 $, one of the following holds:
  \begin{itemize}
    \item For all choices of $ \mathbf{B}_{j_0}^{(h)}, \mathbf{B}_{j_1}^{(h)}, \ldots, \mathbf{B}_{j_{W_i^{(k + i)} - 1}}^{(h)} \in \{\mathbf{B}_j^{(h)}\}_{j = 1}^d $, we have
          \begin{equation*}
            \#(\{\mathbf{B}_j^{(h)}\}_{j = 1}^d \cap \mathrm{Span}_{\mathbb{C}}(\mathbf{B}_{j_0}^{(h)}, \mathbf{B}_{j_1}^{(h)}, \ldots, \mathbf{B}_{j_{W_i^{(k + i)} - 1}}^{(h)})) = W_i^{(k + i)}.
          \end{equation*}
    \item $\#\{\mathbf{B}_j^{(h)}\}_{j = 1}^d < W_i^{(k + i)}.$
  \end{itemize}
  Here, $ W_i^{(j)} $ is given by
  \begin{equation*}
    W_i^{(j)} \coloneqq \sum_{s = 0}^{h - i - 1}\binom{j}{h - s}\binom{n - j + 1}{s}.
  \end{equation*}
  When considering $ \{\mathbf{B}_j\}_{j = 1}^d \subseteq \bigwedge^1 \mathbb{C}^{\binom{n + 1}{h}} = \mathbb{C}^{\binom{n + 1}{h}}$ \textup{(\textit{not necessarily arising from $ \bigwedge^h \mathbb{C}^{n + 1} $ via the Pl\"{u}cker embedding})}, if $ k = 1 $, we simply say that $ \{\mathbf{B}_j\}_{j = 1}^d $ is \textbf{in general position} in this paper.
\end{defn}
\begin{rem}\label{rem:W_i}
  $ (1) $ Let $ \mathbf{A}^{(j)} $ be a nonzero decomposable $ j $-vector. A nonzero decomposable $ h $-vector $ \mathbf{B}^{(h)} $ is said to be \textbf{$ i $-incident with $ \mathbf{A}^{(j)} $} if every $ (h - i) $-vector $ \mathbf{B}^{(h - i)} \subseteq \mathbf{B}^{(h)} $ satisfies $ \mathbf{B}^{(h - i)} \wedge \mathbf{A}^{(j)} = 0 $.
  For a general $ h $-vector $ \mathbf{B}^{(h)} $, we say that it is $ i $-incident with $ \mathbf{A}^{(j)} $ if all of its decomposable components are $ i $-incident with $ \mathbf{A}^{(j)} $.
  All decomposable $ h $-vectors that are $ i $-incident with $ \mathbf{A}^{(j)} $, together with the zero vector, form a linear subspace $ \Delta_i^{(j)}(\mathbf{A}^{(j)}) \subseteq \mathbb{C}^{\binom{n + 1}{h}} $. $ W_i^{(j)} $ is then obtained as the dimension of $ \Delta_i^{(j)}(\mathbf{A}^{(j)}) $ \textup{(\textit{see} \cite{Weyl_2}, p.~258)}:
  \begin{equation*}
    W_i^{(j)} = \dim \Delta_i^{(j)}(\mathbf{A}^{(j)}).
  \end{equation*}
  In the definition of in general position, the Weyls' book \cite{Weyl_2} more precisely considers only those subspaces of the form $ \mathrm{Span}_{\mathbb{C}}(\mathbf{B}_{j_0}^{(h)}, \mathbf{B}_{j_1}^{(h)}, \ldots, \mathbf{B}_{j_{W_i^{(k + i)} - 1}}^{(h)}) $ that are ``equivalent to'' the one given by $ \Delta_i^{(k + i)}(\mathbf{e}_0 \wedge \mathbf{e}_1 \wedge \cdots \wedge \mathbf{e}_{k + i - 1}) $. However, for the sake of simplicity, we do not impose this restriction in this paper.

  \noindent $ (2) $ $ W_i^{(1 + i)} $ is computed as
  \begin{equation*}
    W_i^{(1 + i)} = \sum_{s = 0}^{h - i - 1}\binom{1 + i}{h - s}\binom{n - (1 + i) + 1}{s}  = \binom{n - i}{h - i - 1}.
  \end{equation*}
\end{rem}
All analytic difficulties and techniques in obtaining the Second Main Theorem---such as the \emph{lemma on logarithmic derivatives} (LLD), \emph{sum into products estimate}, and others--- are encapsulated in the following theorem. This theorem also plays an crucial and indispensable role in our approach.
\begin{thm}[\cite{Weyl_2}, p.~263, \textup(11.1\textup); \textit{a version of this result can essentially be found in} \cite{Ahlfors}, pp.~22 \textit{and} 26]\label{thm:omega_M_small}
  Assume that the integers $ n, k, h, i $ and $ l $ satisfy the same conditions as those given above.
  Let $ \{\mathbf{B}_j^{(h)}\}_{j = 1}^d $ be a finite set of nonzero decomposable $ h $-vectors in general position for $ k $. Then, for any $ \kappa > 1 $, the following inequality holds:
  \begin{equation*}
    \Omega_{k + i}(r) + \frac{1}{W_i^{(k + i)}}\sum_{j = 1}^d (M_k^{(i)}(r, \mathbf{B}_j^{(h)}) - M_k^{(i + 1)}(r, \mathbf{B}_j^{(h)})) \leq \kappa\log T_{k + i}(r) - \log r \, //.
  \end{equation*}
\end{thm}
Using this theorem, we can derive an important consequence as stated below. One of the advantages of our approach lies in the simplicity of the computation, in comparison with the original proof by L. V. Ahlfors \cite{Ahlfors}.
The central idea of the proof is to regard the holomorphic curve $ \mathbf{X}^{(p)} $ as lying in $ \mathbb{P}^{\binom{n + 1}{p} - 1} $ rather than in $ \mathrm{Gr}(n + 1, p) $, so that we may apply the above theorem to $ \mathbf{X}^{(p)} $. (A similar idea was employed by H. and F. J. Weyl (see \cite{Weyl_1}, p.~535, (6.8)), and Ahlfors also mentioned this perspective in \cite{Ahlfors}, p.~27. Our aim here is to reinterpret the statement about $ \mathbf{X}^{(p)} $ as one about $ \mathbf{x} $ by incorporating the viewpoint of Lemma~\ref{lem:iT_p-T_i{X^p}_rel}.)
Moreover, we are able to remove the assumption that the $ p $-vectors in the set $ \{\mathbf{A}_j^{(p)}\}_{j = 1}^d $ are decomposable, although the geometric meaning of this generalization remains unclear. It also seems important that this idea shares a structural similarity with the original proof of Schmidt's Subspace Theorem \cite{Schmidt}.
\begin{thm}[\textbf{Second Main Theorem} \textit{for} $ \mathbf{X}^{(p)} $]\label{thm:SMT}

  Let $ \mathbf{x} $ be a holomorphic curve in $ \mathbb{C}^{n + 1} $, and assume that $ \mathbf{X}^{(p)} $ is non-degenerate as a holomorphic curve in $ \mathbb{C}^{\binom{n + 1}{p}} $.
  Let $ \{\mathbf{A}_j^{(p)}\}_{j = 1}^d \subseteq (\bigwedge^p\mathbb{C}^{n + 1})^{*} $ be a finite set of nonzero $ p $-vectors \textup{(\textit{not necessarily decomposable})}, in general position \textup{(\textit{i.e. in general position for $ 1 $ when regarded as points in} $ (\mathbb{C}^{\binom{n + 1}{p}})^{*} $)}.
  Then, for any $ \epsilon > 0 $, the following inequality holds:
  \begin{equation*}
    N_{\binom{n + 1}{p}}\{\mathbf{X}^{(p)}\}(r) + \sum_{j = 1}^d m_p(r, \mathbf{A}_j^{(p)}) < \left\{\binom{n + 1}{p} + \epsilon\right\}\overline{T}_p(r) \, //.
  \end{equation*}
  In particular, when $ p = 1 $, we obtain
  \begin{equation*}
    N_{n + 1}(r) + \sum_{j = 1}^d m(r, \mathbf{a}_j)
    < (n + 1 + \epsilon)T(r) \, //.
  \end{equation*}
\end{thm}
\begin{proof}
  Let $ h \coloneqq \binom{n + 1}{p} - 1 $, and fix $ 0 \leq i \leq h - 1 $ (where we choose $ k = 1 $). Let $ \mathbf{B}_j^{(h)} $ $ (j = 1, 2, \ldots, d)$ be the $ h $-vector defined by $ \mathbf{B}_j^{(h)} \coloneqq \star\mathbf{A}_j^{(p)} \in \bigwedge^h \mathbb{C}^{\binom{n + 1}{p}} $, where the Hodge star operator $ \star $ acts on $ (\mathbb{C}^{\binom{n + 1}{p}})^{*} $.
  Since we assume that $ \{\mathbf{A}_j^{(p)}\}_{j = 1}^d $ is in general position, Theorem~\ref{thm:omega_M_small} and Remark~\ref{rem:W_i} $ (2) $ imply
  \begin{equation*}
    \begin{split}
      \Omega_{1 + i}\{\mathbf{X}^{(p)}\}(r) + \frac{1}{h - i}\sum_{j = 1}^d(M_1^{(i)}\{\mathbf{X}^{(p)}\} & (r, \mathbf{B}_j^{(h)})- M_1^{(i + 1)}\{\mathbf{X}^{(p)}\}(r, \mathbf{B}_j^{(h)})) \\
      & \leq \kappa \log T_{1 + i}\{\mathbf{X}^{(p)}\}(r) - \log r \, //.
    \end{split}
  \end{equation*}
  Hence, we have
  \begin{align}\label{eq:SMT_ineq}
    \sum_{k = 1}^{i + 1}\Omega_k\{\mathbf{X}^{(p)}\}(r) + \sum_{j = 1}^d \sum_{k = 0}^i \frac{1}{h - k}(M_1^{(k)}\{\mathbf{X}^{(p)}\}(r, \mathbf{B}_j^{(h)}) - M_1^{(k + 1)} & \{\mathbf{X}^{(p)}\}(r, \mathbf{B}_j^{(h)}))  \notag \\
                                                                                                                                                                             & < \epsilon \overline{T}_p(r) \, //.
  \end{align}
  By summing over $ i = 0 $ to $ h - 1 $, we obtain
  \begin{align}\label{eq:SMT_ineq_2}
    \sum_{i = 0}^{h - 1}\sum_{k = 1}^{i + 1}\Omega_k\{\mathbf{X}^{(p)}\}(r) + \sum_{j = 1}^d \sum_{i = 0}^{h - 1}\sum_{k = 0}^i \frac{1}{h - k}(M_1^{(k)}\{\mathbf{X}^{(p)}\}(r, \mathbf{B}_j^{(h)}) - & M_1^{(k + 1)}\{\mathbf{X}^{(p)}\}(r, \mathbf{B}_j^{(h)})) \notag \\
                                                                                                                                                                                                       & < \epsilon \overline{T}_p(r) \, //.
  \end{align}
  We compute the left-hand side of \eqref{eq:SMT_ineq_2}. By Lemma~\ref{lem:iT_p-T_i{X^p}_rel}, we have
  \begin{equation*}
    \sum_{i = 0}^{h - 1}\sum_{k = 1}^{i + 1}\Omega_k\{\mathbf{X}^{(p)}\}(r) = -\binom{n + 1}{p}\overline{T}_p + N_{\binom{n + 1}{p}}\{\mathbf{X}^{(p)}\} + O(1).
  \end{equation*}
  Next, we evaluate the following quantity:
  \begin{equation*}
    \sum_{i = 0}^{h - 1} \sum_{k = 0}^i \frac{1}{h - k}(M_1^{(k)}\{\mathbf{X}^{(p)}\}(r, \mathbf{B}_j^{(h)}) - M_1^{(k + 1)}\{\mathbf{X}^{(p)}\}(r, \mathbf{B}_j^{(h)})).
  \end{equation*}
  Since $ M_1^{(h)}\{\mathbf{X}^{(p)}\}(r, \mathbf{B}_j^{(h)}) = 0 $, the expression reduces to
  \begin{equation*}
    \begin{split}
      M_1^{(0)}\{\mathbf{X}^{(p)}\}(r, \mathbf{B}_j^{(h)}) & = \widetilde{m}_1\{\mathbf{X}^{(p)}\}(r, \mathbf{B}_j^{(h)}) - \widetilde{m}_{0}\{\mathbf{X}^{(p)}\}(r, \mathbf{B}_j^{(h)}) \\
      & = m_p(r, \star\mathbf{B}_j^{(h)}) = m_p(r, \mathbf{A}_j^{(p)}),
    \end{split}
  \end{equation*}
  where we use Lemma~\ref{lem:m_p} (or Remark~\ref{rem:gen_m}) and Remark~\ref{rem:M_k} $ (3) $. Combining these computations, we obtain that the left-hand side of \eqref{eq:SMT_ineq_2} equals
  \begin{equation*}
    -\binom{n + 1}{p}\overline{T}_p + N_{\binom{n + 1}{p}}\{\mathbf{X}^{(p)}\} + \sum_{j = 1}^d m_p(r, \mathbf{A}_j^{(p)}) + O(1).
  \end{equation*}
\end{proof}
\begin{rem}
  $ (1) $ $ N_{\binom{n + 1}{p}}\{\mathbf{X}^{(p)}\}(r) $ is the \textbf{ramification counting function} for $ \mathbf{X}_{\mathrm{red}}^{(p)} $, which counts the zeros of the Wronskian $ W\{\mathbf{X}^{(p)}\} $ associated with $ \mathbf{X}_{\mathrm{red}}^{(p)} $. By \eqref{eq:iN_p-N_i{X^p}_rel}, another expression of $ N_{\binom{n + 1}{p}}\{\mathbf{X}^{(p)}\} $ is given as follows:
  \begin{equation*}
    N_{\binom{n + 1}{p}}\{\mathbf{X}^{(p)}\} = \sum_{s = 1}^{\binom{n + 1}{p} - 1}\sum_{k = 1}^sV_k\{\mathbf{X}^{(p)}\} + \binom{n + 1}{p}N_p.
  \end{equation*}

  \noindent$ (2) $ \textup{Ahlfors} \cite{Ahlfors} proved the following form of the Second Main Theorem \textup{(\textit{as a special case}, \cite{Ahlfors}, p.~24, $(\text{II}_k^k) $)}:
  \begin{equation}\label{eq:Ahlfors_SMT}
    \binom{n}{p - 1}N_{n + 1}(r) + \sum_{j = 1}^d m_p(r, \mathbf{A}_j^{(p)}) < \left\{\binom{n + 1}{p} + \epsilon\right\}\overline{T}_p(r) \, //.
  \end{equation}
  \textup{(\textit{Note that $ N_{n + 1} $ is the ramification counting function for $ \mathbf{x} $.})}
  Using \textup{Corollary~\ref{cor:sec_diff_sum_id}}, we obtain
  \begin{equation}\label{another form of SMT}
    \sum_{s = 1}^{p(n - p + 1)}\sum_{\sigma \in \binom{[n + 1]}{p}_{(k_s)}}\left(\sum_{k = 1}^n n_{\lambda(\sigma)}(k)\Omega_k\right) + \sum_{j = 1}^d m_p(r, \mathbf{A}_j^{(p)}) < \epsilon\overline{T}_p(r) \, //.
  \end{equation}
  This can be regarded as one manifestation of the `` $\Omega$-$m$ relation'' \textup{(\textit{see} \cite{Weyl_2}, pp.~241--253)}, in which $ \Omega_k $ and $ m_k $ are, in a certain sense, complementary to each other. It should be noted that, in contrast to \cite{Weyl_2}, the estimate for $ m_p $ does not appear as the difference of two proximity functions.
\end{rem}
\begin{prob}
  We may apply the Weyl peculiar relation for $ \Omega\{\mathbf{X}^{(p)}\} $ \textup{(Theorem~\ref{thm:gen_pec_rel_omega})} to the left-hand side of \eqref{eq:SMT_ineq}. However, the result obtained through this application appears to be weaker than the one obtained above. How, then, can we directly prove \eqref{another form of SMT}?
\end{prob}
A similar argument to that used in the proof of Theorem~\ref{thm:SMT} can be applied not only for $ h = \binom{n + 1}{p} - 1 $,
but also for all $ 1 \leq h \leq \binom{n + 1}{p} - 1 $.
\begin{thm}[Theorem C]\label{thm:gen_SMT}
  Let $ 1 \leq h \leq \binom{n + 1}{p} - 1 $ be an integer, and set $ h' \coloneqq \binom{n + 1}{p} - 1 $. Let $ \{\mathbf{B}_j^{(h)}\}_{j = 1}^d \subseteq \bigwedge^h \mathbb{C}^{\binom{n + 1}{p}} $ be a finite set of nonzero decomposable $ h $-vectors in general position for $ 1 $. Then, for any $ \epsilon > 0 $, the following inequality holds:
  \begin{equation*}
    \begin{split}
      \overline{T}_{h + 1}\{\mathbf{X}^{(p)}\}(r) + \sum_{j = 1}^d \frac{h}{\binom{h'}{h - 1}}\widetilde{m}_1\{\mathbf{X}^{(p)}\}(r, \mathbf{B}_j^{(h)}) + \sum_{j = 1}^d\sum_{k = 1}^{h - 1} & \frac{h' - h}{\binom{h' - k + 1}{h - k}}M_1^{(k)}\{\mathbf{X}^{(p)}\}(r, \mathbf{B}_j^{(h)}) \\
      & < (h + 1 + \epsilon)\overline{T}_p(r) \, //
    \end{split}
  \end{equation*}
  The inequality still holds when the overlines are removed from $ \overline{T}_{h + 1}\{\mathbf{X}^{(p)}\} $ and $ \overline{T}_p $. In particular, when $ p = 1 $, we obtain
  \begin{equation*}
    \overline{T}_{h + 1}(r) + \sum_{j = 1}^d \frac{h}{\binom{n}{h - 1}}\widetilde{m}_1(r, \mathbf{B}_j^{(h)}) + \sum_{j = 1}^d\sum_{k = 1}^{h - 1}\frac{n - h}{\binom{n - k + 1}{h - k}}M_1^{(k)}(r, \mathbf{B}_j^{(h)}) < (h + 1 + \epsilon)T(r) \, //.
  \end{equation*}
\end{thm}
\begin{proof}
  The proof is almost the same as that of Theorem~\ref{thm:SMT}. Since \eqref{eq:iN_p-N_i{X^p}_rel} holds, the inequality remains valid under the replacement of $ \overline{T}_{h + 1}\{\mathbf{X}^{(p)}\} $ by $ T_{h + 1}\{\mathbf{X}^{(p)}\} $ and $ \overline{T}_p $ by $ T_p $. Note the following identity in the calculation:
  \begin{equation*}
    \frac{h - k}{\binom{h' - k}{h - k - 1}} - \frac{h - k + 1}{\binom{h' - k + 1}{h - k}} = \frac{h' - h}{\binom{h' - k + 1}{h - k}}.
  \end{equation*}
\end{proof}
\begin{rem}
  Since $ \widetilde{m}_1 \geq 0 $ and $ M_1^{(k)} \geq 0 $, this theorem can be viewed as a strengthening of equation~\eqref{eq:fund_ineq_3}. It also clearly includes \textup{Theorem~\ref{thm:SMT}} as a special case.
\end{rem}
\subsection{Defect Relations}
\ \par
\begin{defn}
  Let $ \mathbf{A}^{(p)} \in (\bigwedge^p\mathbb{C}^{n + 1})^{*} $ be a nonzero $ p $-vector \textup{(\textit{not necessarily decomposable})}. Then the \textbf{defect} $ \delta_p(\mathbf{A}^{(p)}) $ of $ \mathbf{A}^{(p)} $ for the holomorphic curve $ \mathbf{x} $ is defined by
  \begin{equation*}
    \delta_p(\mathbf{A}^{(p)}) \coloneqq \liminf_{r \to \infty} \frac{m_p(r, \mathbf{A}^{(p)})}{T_p(r)} = 1 - \limsup_{r \to \infty}\frac{N_p(r, \mathbf{A}^{(p)})}{T_p(r)},
  \end{equation*}
  where the second equality follows from \textup{Theorem~\ref{thm:FMT}} \textup{(or Remark~\ref{rem:FMT})}.
\end{defn}
The following important corollary follows from Theorem~\ref{thm:SMT}.
\begin{cor}[\textbf{Ahlfors's Defect Relation}; see, for instance, \cite{Wu}, p.~208]
  \begin{equation}\label{eq:def_rel}
    \sum_{j = 1}^d \delta_p(\mathbf{A}^{(p)}_j) \leq \binom{n + 1}{p}.
  \end{equation}
  In particular, this implies \textbf{Borel's theorem} \textup{($ p = 1 $)} and \textbf{Picard's theorem} \textup{($ p = 1$ \textit{and} $ n = 1 $)}.
\end{cor}
\begin{rem}
  The relation \eqref{eq:def_rel} was originally derived as a consequence of Ahlfors's Second Main Theorem \textup{\eqref{eq:Ahlfors_SMT}}.
  Moreover, \textup{H. Fujimoto} \cite{Fujimoto_1} showed a truncated version of the defect relation:
  \begin{equation}\label{eq:trunc_def_rel}
    \sum_{j = 1}^d \widetilde{\delta}_p(\mathbf{A}_j^{(p)}) \leq \binom{n + 1}{p}.
  \end{equation}
  Here, $ \widetilde{\delta}_p(\mathbf{A}^{(p)}) $ is defined by
  \begin{equation*}
    \widetilde{\delta}_p(\mathbf{A}^{(p)}) \coloneqq 1 - \limsup_{r \to \infty}\frac{\widetilde{N}_p(r, \mathbf{A}^{(p)})}{T_p(r)},
  \end{equation*}
  and the \textbf{truncated counting function} $ \widetilde{N}_p(r, \mathbf{A}^{(p)}) $ is given by
  \begin{equation*}
    \widetilde{N}_p(r, \mathbf{A}^{(p)}) \coloneqq \int_{r_0}^r \sum_{|z| < t}\min(v_p(z, \mathbf{A}^{(p)}), p(n - p + 1))\frac{dt}{t} \quad (r \geq r_0).
  \end{equation*}
\end{rem}
\subsection{Future direction}
\ \par
Let $ X = \bigcup_j U_j$ be a non-singular complex projective algebraic variety, and let $ D $ be a divisor on $ X $ defined by local equations $ \{(U_j, \psi_j)\} $.
Let $ \mathbf{x} : \mathbb{C} \to X $ be a holomorphic curve. Let $ [D] $ denote the line bundle over $ X $ associated with $ D $, eqquipped with a smooth Hermitian metric $ a = \{a_j\} $, where each $ a_j $ is a positive $ C^{\infty} $ function on $ U_j $ satisfying $ |\psi_j|^2 a_i = |\psi_i|^2 a_j $ on $ U_i \cap U_j \neq \varnothing $.
Then $ \psi = \{\psi_j\} \in H^0(X, \mathcal{O}([D])) $, and we define the \textbf{norm} of $ \psi $ by $ || \psi ||^2(w) \coloneqq \frac{|\psi_j(w)|^2}{a_j(w)} \, (w \in U_j) $. The \textbf{Chern form} $ \omega_{[D]} $ is defined by $ \frac{\sqrt{-1}}{2\pi}\partial \bar{\partial} \log a_i $ (on $ U_i $).
In this setting, we define the \textbf{proximity function} $ m_{\mathbf{x}, D}(r) $ and the \textbf{order function} $ T_{\mathbf{x}, [D]}(r) $ by
\begin{align*}
  m_{\mathbf{x}, D}(r)   & \coloneqq \frac{1}{2\pi}\int_0^{2\pi} \log \frac{1}{||\psi||(\mathbf{x}(re^{\sqrt{-1}\theta}))}d\theta, \\
  T_{\mathbf{x}, [D]}(r) & \coloneqq \int_{r_0}^r \frac{dt}{t}\int_{\Delta(t)}\mathbf{x}^{\ast}\omega_{[D]} \quad (r \geq r_0).
\end{align*}
See, for example, K. Kodaira \cite{Kodaira}, S. Lang \cite{Lang}, and \cite{Noguchi-Winkelmann}, for some results developed within this framework.
Note that if $ X = \mathbb{P}^n $ and $ D = H $ (the hyperplane defined by a vector $ \mathbf{a} \in \mathbb{P}^n $), then we obtain
\begin{align*}
  m_{\mathbf{x}, D}      & = m(r, \mathbf{a}) + O(1),                                        \\
  T_{\mathbf{x}, [D]}(r) & = T_{\mathbf{x}, \mathcal{O}_{\mathbb{P}^n}(1)}(r) = T(r) + O(1).
\end{align*}
The following is an important conjecture.
\begin{conj}[\textbf{Griffiths--Lang conjecture}; \textit{see} \cite{Lang}, pp.~196--197]
  Let $ D $ be a simple normal crossing divisor on $ X $.
  Then there exists a proper Zariski closed subset $ Z_{D} \subsetneq    X $ such that, for all holomorphic curves $ \mathbf{x} : \mathbb{C} \to X $ with $ \mathbf{x}(\mathbb{C}) \nsubseteq Z_D $, and for any ample line bundle $ E $ over $ X $, the following inequality holds:
  \begin{equation*}
    N_{\mathbf{x}, \mathrm{Ram}}(r) + m_{\mathbf{x}, D}(r) \leq T_{\mathbf{x}, K_
        X^{-1}}(r) + S_{\mathbf{x}}(r) \, //,
  \end{equation*}
  where $ K_X^{-1} $ is the anticanonical bundle over $ X $, and $ S_{\mathbf{x}}(r) \coloneqq O(\log r + \log^{+}T_{\mathbf{x}, E}(r)) $. Giving an appropriate definition of the ``ramification counting function'' $ N_{\mathbf{x}, \mathrm{Ram}}(r) $ is part of the content of this  conjecture.
\end{conj}
We consider the case $ X = \mathrm{Gr}(p, n + 1) \subseteq \mathbb{P}^{\binom{n + 1}{p} - 1} $ and $ D = \sum_{j = 1}^d H_j $, where $ \{H_j\}_{j = 1}^d $ are hyperplanes defined by $ p $-vectors $ \{\mathbf{A}_j^{(p)}\}_{j = 1}^d $ in general position. Then the above inequality reduces to
\begin{equation*}
  N_{\mathbf{X}^{(p)}, \mathrm{Ram}}(r) + \sum_{j = 1}^d m_p(r, \mathbf{A}_j^{(p)}) \leq (n + 1)T_p(r) + S_{\mathbf{X}^{(p)}}(r) \, //,
\end{equation*}
using the fact that $ K_{\mathrm{Gr}(p, n + 1)} = \mathcal{O}_{\mathbb{P}^{\binom{n + 1}{p} - 1}}(-(n + 1)) $. For example, if $ n = 3 $ and $ p = 2 $, this conjecture predicts that
\begin{equation*}
  N_{\mathbf{X}^{(2)}, \mathrm{Ram}}(r) + \sum_{j = 1}^d m_2(r, \mathbf{A}_j^{(2)}) < (4 + \epsilon)T_2(r) \, //
\end{equation*}
holds for all holomorphic curves $ \mathbf{x} : \mathbb{C} \to \mathbb{P}^3 $ with $ \mathbf{X}^{(2)}(\mathbb{C}) \nsubseteq Z_D $.
On the other hand, from Theorem~\ref{thm:SMT}, we obtain
\begin{equation*}
  N_{6}\{\mathbf{X}^{(2)}\}(r) + \sum_{j = 1}^d m_2(r, \mathbf{A}_j^{(2)}) < (6 + \epsilon)T_2(r) \, //.
\end{equation*}
(By \eqref{eq:Ahlfors_SMT}, this inequality remains valid if the function $ N_{6}\{\mathbf{X}^{(2)}\}(r) $ is replaced with $ 3N_4(r) $.) Therefore, this conjecture provides an alternative formulation of the Second Main Theorem by introducing a new ramification counting function $ N_{\mathbf{x}, \mathrm{Ram}}(r) $.

\section*{Acknowledgments.}
The author would like to thank Professor Ryoichi Kobayashi, who patiently engaged in seminar discussions over the course of a year and provided many valuable comments, new perspectives, and geometric insights, as well as continuous encouragement throughout the research process.

The author is also grateful to his supervisor, Professor Sho Tanimoto, whose helpful advice on the paper inspired the discovery of an application to the Second Main Theorem, and who also provided important feedback on the overall structure.

Finally, the author would like to thank Ryotaro Nishimura for helpful daily discussions and useful suggestions regarding the geometric background of the results.
\par\vspace{1em}
We use the \texttt{genyoungtabtikz} package to generate Young diagrams.

\end{document}